\documentclass[10pt]{article}
\usepackage{amsmath,amssymb,amsthm,amsfonts,amstext,amsbsy,amscd}
\usepackage{mathrsfs}
\usepackage{mathtools}

\usepackage{graphics}
\usepackage{float}
\usepackage{multirow}
\usepackage{multicol}
\usepackage{enumerate}
\usepackage[utf8]{inputenc}
\usepackage{xcolor}

\usepackage[top=3cm, bottom=3.5cm, left=3cm, right=3cm]{geometry}

\RequirePackage[colorlinks,citecolor=blue,urlcolor=blue]{hyperref}
\RequirePackage{hypernat}
\usepackage{hyperref}
\newcommand{\PP}{\mathbb{P}}
\newcommand{\QQ}{\mathbb{Q}}

\newcommand{\E}{\mathbb{E}}

\newcommand{\Li}{\mathscr{L}}
\newcommand{\1}{\mathbf{1}}
\newcommand{\R}{\mathbb{R}}
\newcommand{\N}{\mathbb{N}}
\newcommand{\No}{\mathcal{N}}
\newcommand{\V}{\mathbb{V}}

\newcommand{\eps}{\varepsilon}

\newtheorem{theorem}{Theorem}
\newtheorem{result}{Main Result}

\newtheorem{remark}{Remark}
\newtheorem{lemma}{Lemma}

\newtheorem{corollary}{Corollary}
\newtheorem{proposition}{Proposition}

\renewcommand{\bar}{\overline}
\renewcommand{\tilde}{\widetilde}
\renewcommand{\hat}{\widehat}

\begin{document}

\title{Total variation distance for discretely observed Lévy processes: a Gaussian approximation of the small jumps}
\author{Alexandra Carpentier\footnote{Otto von Guericke Universität Magdeburg, Germany. E-mail: \href{mailto: alexandra.carpentier@ovgu.de}{alexandra.carpentier@ovgu.de}.}
\\ Céline Duval  \footnote{Universit\'e Paris Descartes, Laboratoire  MAP5, UMR CNRS
8145, France. E-mail: \href{celine.duval@parisdescartes.fr}{celine.duval@parisdescartes.fr}.} \ \ and \ \ Ester Mariucci \footnote{Universität Potsdam, Germany.
    E-mail: \href{mailto: mariucci@uni-potsdam.de}{mariucci@uni-potsdam.de}.}}
\date{}

\maketitle

\begin{abstract} 
It is common practice to treat small jumps of Lévy processes as Wiener noise and to approximate its marginals by a Gaussian distribution. However, results that allow to quantify the goodness of this approximation according to a given metric are rare. In this paper, we clarify what happens when the chosen metric is the total variation distance. Such  a choice is motivated by its statistical interpretation; if the total variation distance between two statistical models converges to zero, then no test can be constructed to distinguish the two models and they are therefore asymptotically equally informative. We elaborate a fine analysis of a Gaussian approximation for the small jumps of Lévy processes in total variation distance. Non-asymptotic bounds for the total variation distance between $n$  discrete observations of small jumps of a Lévy process  and the corresponding Gaussian distribution are presented and extensively discussed. As a byproduct, new upper bounds for the total variation distance between discrete observations of Lévy processes are provided. The theory is illustrated by concrete examples.
\end{abstract}

\section{Introduction}

Although Lévy processes, or equivalently infinite divisible distributions, are mathematical objects introduced almost a century ago and even though a good knowledge of their basic properties has since long been achieved, they have recently enjoyed renewed interest. This is mainly due to the numerous applications in various fields. To name some examples, Lévy processes or Lévy-type processes (time changed Lévy processes, Lévy driven SDE, etc...) play a central role in mathematical finance, insurance, telecommunications, biology, neurology,  telecommunications, seismology, meteorology and extreme value theory. Examples of applications may be found in the textbooks \cite{N} and \cite{CT} whereas the manuscripts \cite{Bertoin} and \cite{sato} provide a comprehensive presentation of the properties of these processes. 

The transition from the purely theoretical study of Lévy processes to the need to understand Lévy driven models used in real life applications has led to new challenges. For instance, the questions of how to simulate the trajectories of Lévy processes and how to make inference (prediction, testing, estimation, etc...) for this class of stochastic processes have become a key issue.
The literature concerning these two aspects is already quite large; without any claim of completeness we quote  \cite{asmussen2007}, \cite{AR}, Chapter VI in \cite{N}, \cite{BCGM}, \cite{CR} and Part II in \cite{CT}. We specifically focus on statistics and simulation for Lévy processes, because our paper aims to give an exhaustive answer to a recurring question in these areas: 
When are the small jumps of Lévy processes morally Gaussian?

Before entering into details, we take a step back and see where this question comes from.
Thanks to the Lévy-Itô decomposition, the structure of the paths of any Lévy process is well understood and it is well known that any Lévy process $X$ can be decomposed into the sum of three independent Lévy processes: a Brownian motion with drift, a centered martingale $M$ associated with the small jumps of $X$ and a compound Poisson process describing the big jumps of $X$ (see the decomposition \eqref{eq:model} in Section \ref{not} below). If the properties of continuously or discretely observed compound Poisson processes and of Gaussian processes are well understood, the same cannot be said for the small jumps $M$.
As usual in mathematics, when one faces a complex object a natural reflection is whether the problem can be simplified by replacing the difficult part with an easier but, in a sense, equivalent one. 
There are various notions of equivalence ranging from the weakest, convergence in law, to the stronger convergence in total variation. \\

For some time now, many authors have noticed that marginal laws of small jumps of Lévy processes with infinite Lévy measures resemble to Gaussian random variables, see e.g. Figure \ref{fig1} and \ref{fig2}.
 \begin{figure}[h!]
 \begin{minipage}[b]{.5\linewidth}
\hspace{-0.9cm}\includegraphics[scale=0.28]{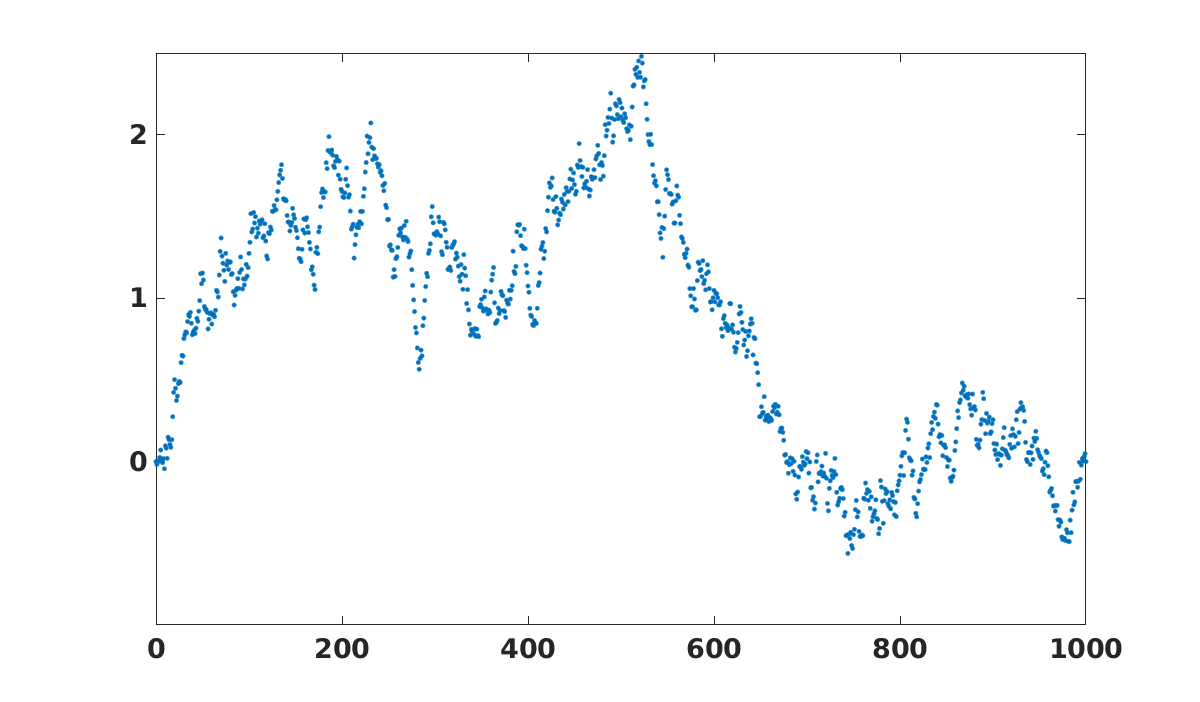}
  \caption{ \footnotesize{Discretized trajectory of a Lévy process $(0,0,\frac{\mathbf{1}_{(0,\eps]}(x)}{x^{1+\beta}})$ for ($n=10^{3},\Delta=1,\eps=10^{-2},\beta=0.9$).}\label{fig1}}
 \end{minipage} \hspace{0.5cm}
 \begin{minipage}[b]{.48\linewidth}
\hspace{-0.8cm}\includegraphics[scale=0.28]{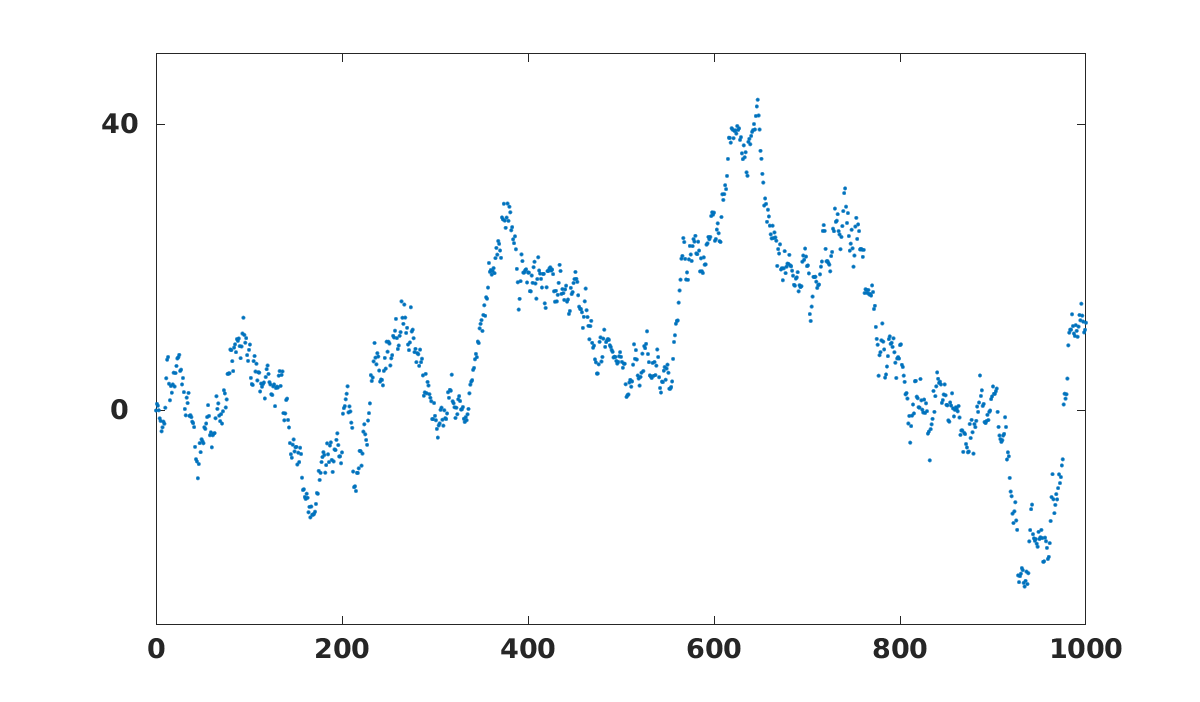}
  \caption{{\footnotesize{Discretized trajectory of a Lévy process $(0,0,\frac{\mathbf{1}_{[-\eps,\eps]\setminus \{0\}}(x)}{x^{1+\beta}})$ for  ($n=10^{3},\Delta=1,\eps=0.5,\beta=1.8$).}}\label{fig2}}
 \end{minipage}
\end{figure}
This remark has led to propose algorithms of simulation of trajectories of Lévy processes based on a Gaussian approximation of the small jumps, see e.g. \cite{CR} or \cite{CT}, Chapter 6. Regarding estimation procedures, a Gaussian approximation of the small jumps has, to the best of our knowledge, not been exploited yet. A fine control of the total variation distance between these two quantities could open the way of new statistical procedures. The choice of this distance is justified by its statistical interpretation: if the total variation distance between the law of the small jumps and the corresponding Gaussian component converges to zero then no statistical test can be built to distinguish between the two models. In terms of information theory, this means that the two models are asymptotically equally informative. 

Investigating the goodness of a Gaussian approximation of the small jumps of a Lévy process in total variation distance makes sense only if one deals with discrete observations. From the continuous observation of a Lévy process, the problem of separating the continuous part from the jumping part does not arise: the jumps are observed. The measure corresponding to the continuous observation of a continuous Lévy process is orthogonal to the measure corresponding to the continuous observation of a Lévy process with non trivial jump part, see e.g. \cite{JS}.
However, the situation changes when dealing with discrete observations. The matter of disentangling continuous and discontinuous part of the processes is much more complex. Intuitively, fine techniques are needed to understand whether, between two observations $X_{t_0}$ and $X_{t_1}$, there has been a chaotic continuous behavior, many small jumps, one single bigger jump, or a mixture of these.

 A criterion for the weak convergence for marginals of Lévy processes is given by Gnedenko and Kolmogorov \cite{GK}:
\begin{theorem}[Gnedenko, Kolmogorov]\label{GK}
Marginals of Lévy processes $X^n=(X_t^n)_{t\geq 0}$ with Lévy triplets $(b_n,\sigma^2_n,\nu_n)$ converge weakly to marginals of a Lévy process $X=(X_t)_{t\geq 0}$ with Lévy triplet $(b,\sigma^2,\nu)$ if and only if 
\[ b_n\to b \text{ and } \sigma_n^2\delta_0+(x^2\wedge 1)\nu_n(dx)\xrightarrow{w} \sigma^2\delta_0+(x^2\wedge 1)\nu(dx),\]
where $\delta_0$ is the Dirac measure in $0$ and $\xrightarrow{w}$ denotes weak convergence of finite measures.
\end{theorem}
A remarkable fact in the previous statement is the non-separation between the continuous and discontinuous parts of the processes: the law at time $t$ of a pure jumps Lévy process can weakly converge to that of a continuous Lévy process. In particular, if $X$ is a Lévy process with Lévy measure $\nu$ then, for any $\eps>0$ and $t>0$, the law of the centered jumps of $X_t$ with magnitude less than $\eps$ converges weakly to a centered Gaussian distribution with variance $t\sigma^2(\eps):=t\int_{|x|<\eps}x^2\nu(dx)$ as $\eps\to 0.$ We aim to understand this phenomenon, using a notion of closeness stronger than the weak convergence, providing a quantitative translation of the result of Gnedenko and Kolmogorov in total variation distance.
 
There exist already several results for distances between Lévy processes. Most of them (see for example \cite{EM}, \cite{JS} and \cite{K}) are distances on the Skorohod space, distances between the continuous observation of the processes, and thus out of the scope of this paper. Concerning discretely observed Lévy processes we mention the results in \cite{liese} and \cite{MR}.
 Liese \cite{liese} proved the following upper bound in total variation distance for marginals of Lévy processes $X^j\sim (b_j,\Sigma_j^2,\nu_j)$, $j=1,2$: for any $t>0$
\begin{align*} \|\mathscr L(X_t^1)-\mathscr L(X_t^2)\|_{TV}&\le 2\sqrt{1-\Big(1-\frac{H^2(\mathcal N(t\tilde b_1,t\Sigma_1^2),\mathcal N(t\tilde b_2,t\Sigma_2^2))}{2}\Big)^2\exp(-t H^2(\nu_1,\nu_2))}
\end{align*}
with $\tilde b_1=b_1-\int_{-1}^1x\nu_1(dx)$, $\tilde b_2=b_2-\int_{-1}^1x\nu_2(dx)$ and $H$ denotes the Hellinger distance. This result is the analogous in discrete time of the result of Mémin and Shyryaev \cite{MS} for continuously observed Lévy processes. There is a clear separation between the continuous and discontinuous parts of the processes, which is unavoidable on the Skorohod space but that can be relaxed when dealing with the marginals. Clearly, from this kind of upper bounds it is not possible to deduce a Gaussian approximation for the small jumps in total variation: the bound is actually trivial whenever $tH^2(\nu_1,\nu_2)>1$. 

However, such an approximation may hold in total variation as proved in \cite{MR}, where the convolution structure of Lévy processes with non-zero diffusion coefficients is exploited to transfer results from Wasserstein distances to the total variation distance. 

In the present paper we complete the work started in \cite{MR}, providing a comprehensive answer to the question: Under which asymptotics, and assumptions on the Lévy triplet, does a Gaussian approximation capture the behavior of the small jumps of a discretely observed Lévy process adequately so that the two corresponding statistical models are equally informative?  Differently from \cite{MR} which deals with marginals, we also establish sharp bounds for the distance between $n$ given increments of the small jumps. This is an essential novelty. Even though from a bound in total variation between marginals one can always deduce a bound for the $n$ sample using $\|P^{\otimes n}-Q^{\otimes n}\|_{TV}\leq \sqrt{2n\|P-Q\|_{TV}}$, this kind of control is in general sub-optimal.\\

Our main results are presented below in a simplified context. First, we established an upper bound for the Gaussian approximation of the small jumps in total variation. 
\begin{result}\label{teointro}
Fix $\eps>0$ and let $\nu_\eps$ be a Lévy measure with support in $[-\eps,\eps]$. Set $\sigma^2(\eps)=\int_{|x|\leq \eps} x^2\nu_\eps(dx)$, $\mu_3(\eps)=\int_{|x|\leq \eps} x^3\nu_\eps(dx)$, $\mu_4(\eps)=\int_{|x|\leq \eps} x^4 \nu_\eps(dx)$ and $b(\eps)$ as in \eqref{beps}. Let $X(\eps)$ be a Lévy process with Lévy triplet $(b,\Sigma^2,\nu_\eps)$, $b\in\R$ and $\Sigma^2\geq 0$.  
Under the assumptions of Theorem \ref{teo:TVgaus} below, there exists a universal constant $C>0$ such that
\begin{align*}
\big\|(X_\Delta(\eps))^{\otimes n}&-\No\big(b(\eps)\Delta,\Delta (\Sigma^2+\sigma^{2}(\eps)\big)^{\otimes n}\big\|_{TV}\leq \nonumber \\
& C \Bigg[\sqrt{\frac{n\mu_4^2(\eps)}{\Delta^{2} (\Sigma^2+\sigma^2(\eps))^4}+\frac{n\mu_3^2(\eps)}{\Delta(\Sigma^2+\sigma^2(\eps))^{3}}} +\frac1n\Bigg].
\end{align*}\end{result}

Main Result \ref{teointro} is non-asymptotic, which allows to quantify just how ``small''  the small jumps must be, in terms of the number of observations $n$ and their frequency $\Delta$, in order for it to be close in total variation to the corresponding Gaussian distribution.

More precisely, fix $n$ and $\Delta$, provided that ${\mu_4^2(\eps)}/{(\sigma^2(\eps))^4} \to 0$ and ${\mu_3^2(\eps)}/{(\sigma^2(\eps))^{3}}\to 0$ as $\eps\to 0$, the total variation in Main Result \ref{teointro} for $\Sigma=0$, that we write  $\tilde r_{n,\Delta}(\eps)$, is bounded under the assumptions of Main Result \ref{teointro} by a quantity of order
\begin{align}\label{eqintro}
\sqrt{n}r_{\Delta}(\eps) +1/n :=   \sqrt n \bigg(\frac{\mu_4(\eps)}{\Delta (\sigma^2(\eps))^2}+\frac{|\mu_3(\eps)|}{\sqrt{\Delta}(\sigma^2(\eps))^{3/2}}\bigg)+\frac1n.\end{align}  
A sufficient condition to ensure $\tilde r_{n,\Delta}(\eps) \to0$ as $\eps\to0$  is that ${\eps}/{\sigma(\eps)} \rightarrow_{\eps\rightarrow 0} 0$ - since we have, taking $N = \max (n, 1/r_{\Delta}(\eps)^{2/3}) \geq n$, that $\tilde r_{n,\Delta}(\eps) \leq  \tilde r_{N,\Delta}(\eps)$ . It is straightforward to see that if $\mu_{3}(\eps)\ne0$ then $\tilde r_{n,\Delta}(\eps) \rightarrow_{\eps\rightarrow 0} 0$ as soon as ${\sqrt{n}\eps}/{\sqrt{\Delta}\sigma(\eps)}\rightarrow_{\eps\rightarrow 0} 0$. When $\mu_{3}(\eps)= 0$, this can be further improved to the condition ${\sqrt{n}\eps^{2}}/{\Delta\sigma^{2}(\eps)}\rightarrow_{\eps\rightarrow 0} 0$.
To exemplify, consider the small jumps of symmetric $\beta$-stable processes with Lévy measure  $\nu_\eps=\1_{|x|\leq \eps}/{|x|^{1+\beta}}$, $\beta\in(0,2)$. Then, a sufficient condition for $\tilde r_{n,\Delta}(\eps) \rightarrow_{\eps\rightarrow 0} 0$ is $\sqrt n {\eps^\beta}/{\Delta} \rightarrow_{\eps\rightarrow 0} 0$ -  see Theorems~\ref{cor:TVgaus} and  \ref{thm6}.

An interesting byproduct of Main Result \ref{teointro} is Theorem \ref{TVn} in Section \ref{sec:TV} which provides a new upper bound for the total variation distance between $n$ given increments of two Lévy processes. A peculiarity of the result is the non-separation between the continuous and discontinuous part of the processes. Then, Theorem \ref{TVn} is close in spirit to Theorem \ref{GK} although it holds for the stronger notion of total variation distance. 

Main Result \ref{teointro} can be sharpened by considering separately large and rare jumps, see~{Theorem \ref{cor:TVgaus}} - this has an impact in the case where the jumps of size of order $\eps$ are very rare. It is optimal, in the sense that whenever the jumps of size of order $\eps$ are not ``too rare'' and whenever the above quantity $r_\Delta(\eps)$ in \eqref{eqintro} is larger than a constant, then the upper bound in Main Result~\ref{teointro} is trivial, but the total variation distance can be bounded away from 0 as shown in Main Result \ref{LBintro} below.

\begin{result}\label{LBintro}
With the same notation used in Main Result \ref{teointro}, for any $\eps>0$, $\Delta>0$ and $n\geq 1$ such that $\eps \leq \sqrt{(\Sigma^2+\sigma^2(\eps))\Delta} \log(e\lor n)^2$ and $\int_{|x|\leq \eps} \nu(dx) \geq \Delta^{-1} \lor \big({\log(e\lor n)}/({n\Delta})\big)$, there exist an absolute sequence $\alpha_n \rightarrow 0$ and an absolute constant $C>0$ such that the following holds:
\begin{align*}
\big\|(X_\Delta(\eps))^{\otimes n}&-\No(b(\eps)\Delta,\Delta (\Sigma^{2}+\sigma^{2}(\eps))^{\otimes n}\big\|_{TV}\geq 
1 - C\Bigg[\frac{\Delta(\Sigma^2+\sigma^2(\eps))^3}{n\mu_3(\eps)^2}\land \frac{\Delta^2(\Sigma^2+\sigma^2(\eps))^4}{n\mu_4(\eps)^2}\Bigg] - \alpha_n.
\end{align*}

\end{result}

 A more technical and general lower bound, that holds without condition on the rarity of the jumps of size of order $\eps$, is also available, see Theorem~\ref{thm:LB}. This lower bound matches in order the general upper bound of Theorem \ref{cor:TVgaus} - implying optimality without conditions of our results. The proof of the lower bound for the total variation is based on the construction of a sharp Gaussian test for Lévy processes. This test combines three ideas, (i) the detection of extreme values that are ``too large'' for being produced by a Brownian motion, (ii) the detection of asymmetries around the drift in the third moment, and (iii) the detection of too heavy tails in the fourth moment for a Brownian motion. It can be of independent interest as it does not rely on the knowledge of the Lévy triplet of the process and detects optimally the presence of jumps. It uses classical ideas from testing through moments and extreme values~\cite{ingster2012nonparametric}, and it adapts them to the specific structure of L\'evy processes. The closest related work is~\cite{reiss2013testing}. We improve on the test proposed there as we go beyond testing based on the fourth moment only, and we tighten the results regarding the presence of rare and large jumps.\\

The paper is organized as follows. In the remaining of this Section we fix notations. Section \ref{sec:gauss} is devoted to the analysis of the Gaussian approximation of the small jumps of Lévy processes. More precisely, in Section \ref{subsubup} we present upper bounds in total variation distance whereas in Section \ref{sec:lb} we provide lower bounds proving the optimality of our findings. In Section \ref{sec:TV} new upper bounds for the total variation distance between $n$ given increments of two general Lévy processes are derived. Most of the proofs are postponed to Section \ref{proofs}. The paper also contains two Appendices. In Appendix \ref{appA} technical results can be found. In Appendix \ref{appB} we recall some results about total variation distance and present some general, and probably not new, results about the total variation distance between Gaussian distributions and discrete observations of compound Poisson processes.
\subsubsection*{Statistical setting and notation}\label{not}

For $X$ a Lévy process with Lévy triplet  $(b,\Sigma^2,\nu)$ (we write $X \sim (b,\Sigma^2,\nu)$), where $b\in\R$, $\Sigma\geq0$ and $\nu$ is a Borel measure satisfying $\nu(\{0\})=0$ and $\int (x^{2}\wedge 1)\nu(dx)<\infty,$ the Lévy-Itô decomposition gives a decomposition of $X$ as the sum of three independent Lévy processes: a Brownian motion with drift $b$, a centered martingale associated with the small jumps of $X$ and a compound Poisson process associated with the large jumps of $X$. More precisely, for any $\eps >0$, $X \sim (b,\Sigma^2,\nu)$  can be decomposed as
\begin{align}
 X_t&=b(\eps) t+ \Sigma W_t+\lim_{\eta\to 0} \bigg(\sum_{s\leq t}\Delta X_s\1_{\eta<|\Delta X_s|\leq\varepsilon}-t\int_{\eta<|x|\leq \varepsilon}x\nu(dx)\bigg)+\sum_{s\leq t}\Delta X_s\1_{|\Delta X_s|> \varepsilon}, \nonumber \\
 &:=b(\eps) t+\Sigma W_t+M_t(\varepsilon)+Z_t(\eps), \quad \forall t\geq 0 \label{eq:model}
 \end{align}
 where $\Delta X_{t}=X_{t}-\lim_{s\uparrow t}X_{s}$ denotes the jump at time $t$ of $X$ and
 \begin{itemize}
 \item the drift is defined as
 \begin{equation}\label{beps}b(\eps):=b+\begin{cases}
-\int_{\eps< |x|\leq1 }x\nu(dx) & \mbox{if }\eps\leq 1,\\
\int_{1< |x|\leq \eps}x\nu(dx) & \mbox{if }\eps> 1.
 \end{cases}\end{equation}
 \item $W=(W_t)_{t\geq 0}$ is a standard Brownian motion;
 \item $M(\eps)=(M_t(\eps))_{t\geq 0}$ is a centered Lévy process (and a martingale) with a Lévy measure $\nu_{\eps}:=\nu\1_{[-\eps,\eps]}$ \textit{i.e.} it is the Lévy process associated to the jumps of $X$ smaller than $\eps$. More precisely, $M(\eps)\sim(-\int_{1<|x|\leq \eps} x\nu_\eps(dx),0,\nu_\eps)$. Observe that $\int_{\eps<|x|\leq 1} x\nu_\eps(dx)=0$ for any $\eps\leq 1$. We write $\sigma^2(\eps)=\int x^2\nu_\eps(dx)$ for the variance at time $1$ of $M(\eps)$ and $\mu_k(\eps)=\int_{|x|\leq \eps}x^k\nu_\eps(dx)$ for the $k$-th moment of the Lévy measure $\nu_\eps$.
 \item $Z(\eps)=(Z_t(\eps))_{t\geq 0}$ is a Lévy process with a Lévy measure concentrated on $\R\setminus [-\varepsilon,\varepsilon]$ \textit{i.e.} it is a compound Poisson process of the form $Z_t(\eps):=\sum_{i=1}^{P_t} Z_i$ with intensity $\Lambda_\eps:=\nu(\R\setminus [-\varepsilon,\varepsilon])$ and jump measure $\PP(Z_1\in B)=\frac1{\Lambda_\eps}{\int_{B}\1_{\R\setminus [-\varepsilon,\varepsilon]}(x)\nu(dx)}$;
 \item $W$, $M(\eps)$ and $Z(\eps)$ are independent.
 \end{itemize}

The total variation distance between two probability measures $P_1$ and $P_2$ defined on the same $\sigma$-field $\mathcal B$ is defined as
\begin{align*}
\|P_1-P_2\|_{TV}:=\sup_{B\in\mathcal B}|P_1(B)-P_2(B)|=\frac{1}{2}\int \bigg|\frac{dP_1}{d\mu}(x)-\frac{dP_2}{d\mu}(x)\bigg|\mu(dx),
\end{align*}
where $\mu$ is a common dominating measure for $P_1$ and $P_2$. 
To ease the reading, if $X$ and $Y$ are random variables with densities $f_X$ and $f_Y$ with respect to a same dominating measure, we sometimes write $\|X-Y\|_{TV}$ or $\|f_X-f_Y\|_{TV}$ instead of $\|\mathscr L(X)-\mathscr L(Y)\|_{TV}$. 
Finally, we denote by $\mathcal N(m,\Sigma^2)$ the law of a Gaussian distribution with mean $m$ and variance $\Sigma^2$. Sometimes, with a slight abuse of notation, we denote with the same symbol $\mathcal N(m,\Sigma^2)$ a Gaussian random variable with mean $m$ and variance $\Sigma^2$. The symbol $\#A$ indicates the cardinality of a set $A$.
\section{Gaussian approximation for the L\'evy process in total variation distance} \label{sec:gauss}
\subsection{Upper bound results}\label{subsubup}

We investigate under which conditions on $\Delta$, $n$, $\eps$ and the Lévy triplet, a Gaussian approximation of the small jumps (possibly convoluted with the continuous part of the process) is valid. Define 
\[
\lambda_{\eta,\eps}:=\int_{\eta<|x|<\eps}\nu(dx),\quad 0\leq \eta<\eps,\quad \mbox{and }\quad \lambda_{0,\eps}=\lim_{\eta\to0}\lambda_{\eta,\eps},
\] where $\lambda_{0,\eps}=+\infty$ if $\nu$ is an infinite Lévy measure.

\begin{theorem}\label{teo:TVgaus}
For any $\eps>0$, let $X(\eps)\sim(b,\Sigma^2,\nu_\eps)$ with $b\in \R$, $\Sigma^2\geq 0$ and $\nu_\eps$ a Lévy measure with support in $[-\eps,\eps].$ For any $\Delta>0$, set $\tilde X_{\Delta}(\eps):=(X_\Delta(\eps)-\Delta b(\eps))/{\sqrt{\Delta(\Sigma^{2}+\sigma^{2}(\eps))}}$. For any $n\geq 1$ such that $\lambda_{0,\eps} \geq {24\log(e\lor n)}/{\Delta}$,
assume that there exist universal constants $c>1$, $C'>0$ and $\tilde c\in (0,1]$ such that,
denoting by $\Psi_\eps(t):=\E[e^{it\tilde X_{\Delta}(\eps)}] - e^{-\Delta \lambda_{0,\eps}}\mathbf 1_{\{\Sigma = 0\}}$, the following three conditions hold:
\begin{align}
\label{ass:psiK}\tag{$\mathcal{H}(\Psi_\eps)$}
 &\int_{c{\log(e\lor n)}}^{+\infty}|\Psi_\eps^{(k)}(t)|^2dt \leq C'2^k k! n^{-2},\quad \forall k\in [0,401\log(e\lor n)],\\
& \label{ass:epsD}\eps \leq \tilde c \sqrt{(\sigma^2(\eps)+\Sigma^2)\Delta/\log(n)}:=\tilde c_n \sqrt{\Delta} \sqrt{\sigma^2(\eps)+\Sigma^2}\tag{$\mathcal{H}_{\eps}$},
\end{align} and $c\tilde c \leq \sqrt{\log(e\lor n)}/4$.
Then, there exists a constant $C>0$, depending only on $c, C'$, such that \begin{align}
\big\|(X_\Delta(\eps))^{\otimes n}&-\No\big(b(\eps)\Delta,\Delta (\Sigma^2+\sigma^{2}(\eps)\big)^{\otimes n}\big\|_{TV}\leq 
 C \Bigg[\sqrt{\frac{n\mu_4^2(\eps)}{\Delta^{2} (\Sigma^2+\sigma^2(\eps))^4}+\frac{n\mu_3^2(\eps)}{\Delta(\Sigma^2+\sigma^2(\eps))^{3}}} +\frac1n\Bigg].\label{eq:th2}
\end{align}
\end{theorem}

\begin{remark}

In \eqref{ass:psiK}, $ \Psi_\eps$ is the characteristic function of the rescaled increment $\tilde X_\Delta(\eps)$ restricted to the event ``the process jumps at least once in the case of a finite pure jump Lévy processes.'' Otherwise, if the L\'evy measure is infinite and/or if  it has a Brownian component, it is simply the characteristic function of $\tilde X_{\Delta}(\eps)$. The addition of the indicator function in $ \Psi_\eps$ permits to keep (some) compound Poisson processes in the study, \textit{e.g.} compound Poisson processes with continuous jump density, for which $\lim_{|t|\rightarrow +\infty}\E[\exp(it\tilde X_{\Delta}(\eps))]$ converges to $e^{-\Delta \lambda_{0,\eps}}$. However these compound Poisson processes should also satisfy $\lambda_{0,\eps} \geq {24\log(e\lor n)}/{\Delta}$, \textit{i.e.} their intensity cannot be too small. The latter assumption implies that the probability of observing no jump in any of our $n$ observations converges to $0$, which is a necessary assumption if $\Sigma = 0$ in order to avoid getting a trivial bound, as explained below in Subsection~\ref{comments}.

\end{remark}

If $\Sigma=0$, we immediately get the following Corollary using that
\begin{equation} \label{bmu}
\mu_4(\eps)\leq \eps^2 \sigma^2(\eps),\quad \mu_3(\eps)\leq \eps\sigma^2(\eps).
\end{equation}
\begin{corollary}\label{cor:sim}
For any $\eps>0$, let $M(\eps)\sim(-\int_{1<|x|\leq \eps} x\nu_\eps(dx),0,\nu_\eps)$ with $\nu_\eps$ a Lévy measure with support in $[-\eps,\eps]$. Under the same assumptions of Theorem \ref{teo:TVgaus}, it holds
\begin{itemize}
\item
If $\nu_\eps$ is symmetric, there exists a universal constant $C$ such that:
\begin{align}\label{eq:s}\|(M_\Delta(\eps))^{\otimes n}-\mathcal{N}(0,\Delta\sigma^{2}(\eps))^{\otimes n}\|_{TV}\leq C\bigg(\sqrt{n\frac{\eps^{4}}{\Delta^{2}\sigma^{4}(\eps)}}+\frac1n\bigg).\end{align}
\item
If $\nu_\eps$ is not symmetric, there exists a universal constant $C$ such that:
\begin{align}\label{eq:nons}\|(M_{\Delta}(\eps))^{\otimes n}-\mathcal{N}(0,\Delta\sigma^{2}(\eps))^{\otimes n}\|_{TV}\leq C\bigg(\sqrt{n\frac{\eps^{2}}{\Delta\sigma^{2}(\eps)}}+\frac1n\bigg).\end{align}
\end{itemize}
\end{corollary}

Theorem \ref{teo:TVgaus} can be improved as follows. We provide a control of the distance between the increments of a Lévy process $(b,\Sigma^{2},\nu_\eps),$ where $\nu_\eps$ has support $[-\eps,\eps]$, and the closest Gaussian distribution, which may not be the one with average $b(\eps)\Delta$ and variance $\Delta(\Sigma^{2}+\sigma^{2}(\eps))$. It permits to weaken the assumption $\eps\lesssim \sqrt{\Delta(\Sigma^{2}+\sigma^{2}(\eps))/\log( e\lor n)}$ through the introduction of the following quantities.
Set
$u^+ := u^+_{\nu_\eps}$ for the largest positive real number such that
$$\lambda_{u^+,\eps}n\Delta \geq \log(e\lor n).$$ 
Note that such a $u^+$ exists and is unique whenever $\nu_{\eps}$ is a Lévy measure such that $\lambda_{0,\eps} \geq  {\log(e\lor n)}/{n\Delta}$, which holds under the assumptions of Theorem \ref{teo:TVgaus}.  Consider the quantity
$$\tilde u^*(\tilde c) = \sup\Big\{ u : u\in [u^+,\eps], u \leq \tilde c \sqrt{(\sigma(u)^2+\Sigma^2)\Delta} /{\sqrt{\log(e \lor n)}}\Big\} \lor u^+,\quad \mbox{where } \tilde c\in(0,1].$$ 
Sometimes we will write $\tilde u^*$ instead of $\tilde u^*(\tilde c)$ to lighten the notation.

The introduction of the latter quantity permits to remove Assumption \eqref{ass:epsD} in Theorem \ref{teo:TVgaus}. 

\begin{theorem} \label{cor:TVgaus}
For any $\eps>0$, let $X(\eps)\sim(b,\Sigma^2,\nu_\eps)$ with $b\in \R$, $\Sigma^2\geq 0$ and $\nu_\eps$ a Lévy measure with support in $[-\eps,\eps].$ Let $\Delta>0$ and $n\geq 1$ be such that $\lambda_{0,\eps} \geq {25\log(e\lor n)}/{\Delta}$. Denote by $\tilde X_{\Delta}(u):=(X_\Delta(u)-\Delta b(u))/\sqrt{\Delta(\Sigma^{2}+\sigma^{2}(u))}$ and by $\Psi_{u}(t):=\E[e^{it\tilde X_{\Delta}(u)}] - e^{-\Delta \lambda_{0, u}}\mathbf 1_{\{\Sigma = 0\}}$, for any $u>0$.
Furthermore, assume that there exist universal constants $c>1$, $\tilde c\in (0,1]$ and $C'>0$ such that the following two conditions hold:
 \begin{align}
\label{ass:psiK2}\tag{$\mathcal{H}_{\Psi}(\tilde u^*)$}
& \int_{c{\log(e\lor n)}}^{+\infty}|\Psi_{\tilde u^*(\tilde c)}^{(k)}(t)|^2dt \leq C'2^k k! n^{-2},\quad \forall k\in [0,401\log(e\lor n)],
\end{align} and $c\tilde c \leq \sqrt{\log(e\lor n)}/4.$
Then, there exists a constant $C>0$ that depends only on $c, C'$ such that \begin{align}
&\nonumber\min_{B \in \mathbb R, S^2\geq 0}\big\|(X_\Delta(\eps))^{\otimes n}-\No(B\Delta,\Delta S^2)^{\otimes n}\big\|_{TV}\\& \leq 1 - e^{-\tilde \lambda^{*}\Delta n}+Ce^{-\tilde \lambda^{*}n\Delta} \sqrt{\frac{n\mu_4^2(\tilde u^*(\tilde c))}{\Delta^{2} (\Sigma^2+\sigma^2(\tilde u^*(\tilde c)))^4}+\frac{n\mu_3^2(\tilde u^*(\tilde c))}{\Delta(\Sigma^2+\sigma^2(\tilde u^*(\tilde c)))^{3}}} +\frac{C}{n} ,\label{eq:cor}
\end{align} where we set $\tilde \lambda^{*}:=\lambda_{\tilde u^*(\tilde c),\eps}.$
\end{theorem}

\subsubsection{Comments} \label{comments}

\paragraph{Discussion on the rates of Theorem \ref{teo:TVgaus} and Corollary \ref{cor:sim}.} 
The results are non-asymptotic and we stress that \eqref{eq:th2}, \eqref{eq:s}, \eqref{eq:nons} and \eqref{eq:cor} hold without assuming to work with high or low frequency observations. An explicit relation between $\eps$, $\Sigma$, $n$ and $\Delta$ depending on $\nu_\eps$ via $\sigma^2(\eps)$, $\mu_3(\eps)$ and $\mu_4(\eps)$ is given. More precisely, we derive from Theorem \ref{teo:TVgaus} that the total variation distance in \eqref{eq:th2} is bounded by
\begin{align}\label{eq:rateTV}
C\Big(\sqrt n \bigg(\frac{\mu_4(\eps)}{\Delta (\Sigma^2+\sigma^2(\eps))^2}+\frac{\mu_3(\eps)}{\sqrt{\Delta}(\Sigma^2+\sigma^2(\eps))^{3/2}}\bigg)+\frac1n\Big).
\end{align}
As highlighted in Corollary \ref{cor:sim}, a sufficient condition (under the assumptions of Theorem \ref{teo:TVgaus}) for the total variation distance to be small is given by ${\eps}/{\sigma(\eps)}\to0$ as $\eps\to0$, with a rate depending on $n,\Delta$. Unsurprisingly, we observe that the rate of convergence for a Gaussian approximation of the small jumps is faster when the Lévy measure is symmetric. 

Assumption \eqref{ass:epsD} of Theorem \ref{teo:TVgaus} imposes a constraint on the cumulants of $\nu_\eps$. It is restrictive, but intuitively meaningful, it means that the jumps cannot take values that are too extreme with respect to the standard deviation of the increments. These extreme values would indeed enable to differentiate it easily from a Gaussian distribution. Theorem \ref{cor:TVgaus} allows to get rid of this assumption.

Finally, the remainder term $1/n$ in \eqref{eq:th2} is a consequence of the strategy of the proof and can be improved, at the expanse of modifying some details in the proof. More precisely, for any $\kappa>0$ and under the assumption of Theorem \ref{teo:TVgaus}, there exists $C_\kappa>0$ that depends only on $\kappa, c, C'$ such that 
\begin{align}\label{eq:rateTVimp}
CC_\kappa \sqrt n \bigg(\frac{\mu_4(\eps)}{\Delta (\Sigma^2+\sigma^2(\eps))^2}+\frac{\mu_3(\eps)}{\sqrt{\Delta}(\Sigma^2+\sigma^2(\eps))^{3/2}}\bigg)+C\frac{1}{n^{\kappa}}:=CC_\kappa\sqrt{n}r_{n}+Cn^{-\kappa}.
\end{align} This remark permits to achieve a meaningful bound for the marginals using that for any $n\geq 1$,
\begin{align*}
\big\|X_\Delta(\eps)&-\No\big(b(\eps)\Delta,\Delta (\Sigma^2+\sigma^{2}(\eps)\big)\big\|_{TV}
\leq\big\|(X_\Delta(\eps))^{\otimes n}-\No\big(b(\eps)\Delta,\Delta (\Sigma^2+\sigma^{2}(\eps)\big)^{\otimes n}\big\|_{TV}.
\end{align*}
Applying \eqref{eq:rateTVimp} for $n=\lfloor r_{n}^{-1/\kappa}\rfloor$, we get that there exists $C_\kappa' >0$ that depends only on $\kappa, c, C'$ such that
\begin{align}\label{eq:TVdim1}
\big\|X_\Delta(\eps)&-\No\big(b(\eps)\Delta,\Delta (\Sigma^2+\sigma^{2}(\eps)\big)\big\|_{TV}\leq C_\kappa' \bigg(\frac{\mu_4(\eps)}{\Delta (\Sigma^2+\sigma^2(\eps))^2}+\frac{\mu_3(\eps)}{\sqrt{\Delta}(\Sigma^2+\sigma^2(\eps))^{3/2}}\bigg)^{1-\frac{1}{2\kappa}}.
\end{align} Which can be rendered arbitrarily close to $\Big(\frac{\mu_4(\eps)}{\Delta (\Sigma^2+\sigma^2(\eps))^2}+\frac{\mu_3(\eps)}{\sqrt{\Delta}(\Sigma^2+\sigma^2(\eps))^{3/2}}\Big)$ for $\kappa$ large enough - at the expense of the increasing constant $C'_\kappa$ in $\kappa$.

\paragraph{Discussion on Theorem \ref{cor:TVgaus}.} 
A restrictive assumption in Theorem \ref{teo:TVgaus} and Corollary \ref{cor:sim} is that $\eps \leq \tilde c \sqrt{(\sigma^2(\eps) + \Sigma^2) \Delta} /{\sqrt{\log(e\lor n)}}$, \textit{i.e.} $\eps$ smaller than the standard deviation $\sqrt{\Delta(\sigma^2(\eps) + \Sigma^2)}$ of the increments (up to a multiplicative constant and a logarithmic term). This assumption is avoided in Theorem~\ref{cor:TVgaus}, which follows directly from Theorem~\ref{teo:TVgaus} by dividing the total variation on two events, where respectively all jumps are smaller than $\tilde u^*$, or where there is at least one jump larger than $\tilde u^*$. The idea behind this subdivision is that the Gaussian distribution that is the closest to $(X_\Delta(\eps))^{\otimes n}$ is not $\No(b(\eps)\Delta,\Delta(\Sigma^2+\sigma^2(\eps)))^{\otimes n}$, but rather $\No(b(\eps)\Delta,\Delta(\Sigma^2+\sigma^2(\tilde u^*)))^{\otimes n}$. Indeed all jumps that are larger than $\tilde u^*$ are very rare and large enough to be recognized as non-Gaussian. The upper bound on the total variation distance is then composed of two terms, the first one that appears in \eqref{eq:rateTV}, but expressed in $\tilde u^*$ and not $\eps$, and the second one is the probability of observing a jump larger than $\tilde u^*$, namely $\exp(-n\Delta \lambda_{\tilde u^*, \eps})$.

\paragraph{Remark on the assumption on $\lambda_{0,\eps}$.} What one needs for establishing Theorems~\ref{teo:TVgaus} and~\ref{cor:TVgaus} and Corollary \ref{cor:sim}, is that $\lambda_{0,\eps} \geq {24\log(e\lor n)}/{\Delta}.$ For establishing Theorem~\ref{cor:LB} (lower bound on the total variation distance in Section \ref{sec:lb}) we only need that $\lambda_{0,\eps} \geq \Delta^{-1} \lor (\log(e\lor n) /(n\Delta))$. Indeed, if this is not the case, the asymptotic total variation is either $0$ or $1$.
\begin{itemize}
\item \textbf{Case 1 : $a_n\Delta^{-1} n^{-1}\geq \lambda_{0,\eps}$ where $a_n \rightarrow 0$. } In this case one does never observe any jump with probability going to $1$. So the total variation distance goes to $0$ as $n\to\infty$. 
\item \textbf{Case 2 (only in the noiseless case $b=0, \Sigma = 0$): $a_n\Delta^{-1} n^{-1}\leq \lambda_{0,\eps} \leq A_n \log(e\lor n)\Delta^{-1}$ where $a_n \rightarrow \infty$ and $A_n \rightarrow 0$.} In this case the probability of observing at least a time step where one, and only one, jump occurs goes to 1, as well as the probability of having at least one time step where no jump occur goes to 1 as $n\to\infty$. We conclude therefore that the total variation distance goes to $1$: such a process is very different from a Gaussian process.
\end{itemize}
 
\paragraph{Discussion on Assumption~\eqref{ass:psiK}.} 
Assumption \eqref{ass:psiK} is technical, but it does not seem to restrict drastically the class of Lévy processes one can consider. It holds for instance for the large class of $\beta$-stable processes - see Proposition \ref{prop:exstable} in Section \ref{ex} - which describes well the behavior of many Lévy densities around 0. It also holds as soon as $\Sigma$ is large, \textit{i.e.} whenever $\Sigma$ is larger than $\sigma(\eps)$ - see Proposition \ref{prop:Siglar} in Section \ref{ex}.  Most usual compound Poisson processes with Lebesgue continuous jump density and intensity $\lambda_{0,\eps}$ large enough also seem to satisfy it. It imposes a condition on the decay of the derivatives of the characteristic function of the rescaled increment, on the event where at least a jump is observed.  A condition related to Assumption \eqref{ass:psiK} in the particular case where $k=0$ has already been investigated (see \textit{e.g.} Trabs \cite{trabs2014infinitely}), but the results therein do not apply to infinite Lévy densities.

This assumption is not straightforward to interpret, but we report the following observation. A necessary condition for it to be satisfied is that the characteristic function $\Psi_\eps$ of the rescaled increment - on the event where at least a jump is observed - goes to $0$. Examples for which \eqref{ass:psiK} does not hold are for instance Lévy processes such that $\Sigma = 0$ and $\nu_\eps$ contains a finite number of Dirac masses.  This is coherent, observations of a pure jump process with jump law taking finitely many values are perfectly detectable from Gaussian observations, \textit{i.e.} the total variation is $1$. However, if the law of the increments contains Dirac masses but if the total probability of these masses is much smaller than $1/n$, this in principle does not disturb the total variation bound. This is why, our analysis allows to consider compound Poisson processes whenever $\lambda_{0,\eps} \geq 24 \log(e\lor n) /\Delta$ (and whenever $\lambda_{0,\eps} \lesssim \log(e\lor n) /\Delta$, the bound on the total variation becomes trivial for $\Sigma = 0$ as noted above).

\paragraph{Relation to other works.}
To the best of our knowledge, the only work in which non-asymptotic results are found for a Gaussian approximation of the small jumps of Lévy processes is \cite{MR}. 
Proposition 8 in \cite{MR} states that, for any $\Sigma>0$ and $\eps\in(0,1]$:
\begin{align}\label{eq:TVM}
\|\No(0,t\Sigma^2)*M_t(\eps)-\No(0, t(\sigma^2(\eps)+\Sigma^2))\|_{TV}\leq \frac{1}{\sqrt{2\pi t\Sigma^2}}\min\Big(2\sqrt{t\sigma^2(\eps)},\frac{\eps}{2}\Big).
\end{align}
The assumption $\Sigma>0$ is an artifact of the proof and arises from the need of a convolutional structure enabling to transfer results in Wasserstein distance of order $1$ to results in total variation distance, see Proposition 4 in \cite{MR}. However, this type of approach leads in many cases to suboptimal results as Theorem \ref{teo:TVgaus} in the particular case where $n=1$ and $\Sigma>0$ gives (see the modified version of the bound given in \eqref{eq:TVdim1})\begin{align*}
\|\No(0,t\Sigma^2)*M_t(\eps)-\No(0, t(\sigma^2(\eps)+\Sigma^2))\|_{TV}\leq C_\kappa'\bigg(\frac{\mu_4(\eps)}{t(\Sigma^2+\sigma^2(\eps))^2}+\frac{\mu_3(\eps)}{\sqrt{t}(\Sigma^2+\sigma^2(\eps))^{3/2}}\bigg)^{1-\frac{1}{2\kappa}},
\end{align*} where $\kappa>0$ can be chosen large and $C_\kappa'$ is a constant depending  on $\kappa, c, C'$.
This latter rate is tighter than \eqref{eq:TVM} thanks to \eqref{bmu}. Note also that the Assumption~\eqref{ass:psiK} is implied by Proposition~\ref{prop:Siglar} below as soon as $\sigma(\eps)$ is smaller up to a multiplicative constant than $\Sigma$ - which is the case whenever the bound in Equation~\eqref{eq:TVM} is non trivial. Contrary to \eqref{eq:TVM} we do not have explicit constants.

\subsection{Lower bound}\label{sec:lb}

Theorem \ref{teo:TVgaus} is optimal in the following sense. If the upper bound of Theorem  \ref{teo:TVgaus} does not converge, then the total variation distance between the random vector associated with the increments of the small jumps -possibly convoluted with Gaussian distributions- and the corresponding Gaussian random vector does not converge to 0. \begin{theorem}\label{cor:LB}
For any $\eps>0$, let $X(\eps)\sim(b,\Sigma^2,\nu_\eps)$ with $b\in \R$, $\Sigma^2\geq 0$ and $\nu_\eps$ a Lévy measure with support in $[-\eps,\eps].$
For any $n\geq 1$ and $\Delta>0$ such that  $\lambda_{0,\eps} \geq \Delta^{-1} \lor \big({\log(e\lor n)}/({n\Delta})\big)$ and $\eps \leq  \sqrt{(\sigma^2(\eps) + \Sigma^2) \Delta} \log(e\lor n)^2$, there exist an absolute sequence $\alpha_n \rightarrow 0$ and an absolute constant $C>0$ such that the following holds:
\begin{align*}
\big\|(X_\Delta(\eps))^{\otimes n}&-\No(b(\eps)\Delta,\Delta (\Sigma^{2}+\sigma^{2}(\eps)))^{\otimes n}\big\|_{TV}\geq 
1 - C\Bigg[\frac{\Delta(\Sigma^2+\sigma^2(\eps))^3}{n\mu_3(\eps)^2}\land \frac{\Delta^2(\Sigma^2+\sigma^2(\eps))^4}{n\mu_4(\eps)^2}\Bigg] - \alpha_n.
\end{align*}
\end{theorem}

To establish Theorem \ref{cor:LB} we construct an appropriate statistical test and use the following fact.
\begin{lemma}\label{lem:TvTest}
Let $\PP$ and $\QQ$ be two probability distributions and $\Phi$ a test of level $\alpha_0\in(0,1)$ that separates $\PP$ from $\QQ$ from $n$ i.i.d. observation with power larger than $1-\alpha_1$. Then, $\|\PP^{\otimes n}-\QQ^{\otimes n}\|_{TV}\geq 1-\alpha_0-\alpha_1.$
\end{lemma}

Let $X\sim (b, \Sigma^2, \nu)$, with $b$, $\Sigma^2$ and $\nu$ possibly unknown and consider the problem of testing whether the process contains jumps of size smaller than $\eps$ or not, \textit{i.e.} whether $\nu_\eps = 0$ or not, recall that $\nu_{\eps}=\nu\mathbf{1}_{[-\eps,\eps]}$ and that we defined
$u^+ $ for the largest positive real number such that
$\lambda_{u^+,\eps}n\Delta \geq \log(e\lor n).$
Write now $u^*$ for the largest $u\in[u^+,\eps]$ such that
\begin{align*}
{u^{*}=\sup\{u,u\in[ u^{+},\eps],u\leq \sqrt{\Delta (\Sigma^2+\sigma^2(u))}  \log(e\lor n)^2 \}\vee u^{+}},
\end{align*} where $\sup\emptyset=-\infty$.

We prove the following {result}, of which Theorem~\ref{cor:LB} is an immediate corollary.

\begin{theorem}\label{thm:LB}
For any $\eps>0$, let $X(\eps)\sim(b,\Sigma^2,\nu_\eps)$ with $b\in \R$, $\Sigma^2\geq 0$ and $\nu_\eps$ a Lévy measure with support in $[-\eps,\eps].$
For any $n\geq 1$ and $\Delta>0$ such that $\lambda_{0,\eps} \geq \Delta^{-1} \lor \big({\log(e\lor n)}/({n\Delta})\big)$, there exists an absolute sequence $\alpha_n \rightarrow 0$ and an absolute constant $C>0$ such that the following holds:

\begin{align*}
\min_{B \in \mathbb R, S^2\geq 0}&\big\|(X_\Delta(\eps))^{\otimes n}-\No(B\Delta,\Delta S^2)^{\otimes n}\big\|_{TV}\geq \\
& 
\Bigg\{1 - C\Bigg[\frac{\Delta(\Sigma^2+\sigma^2(u^*))^3}{n\mu_3(u^*)^2}\land \frac{\Delta^2(\Sigma^2+\sigma^2(u^*))^4}{n\mu_4(u^*)^2}\Bigg] - \alpha_n\Bigg\}\lor 
\Big(1 - \exp(-\lambda_{u^*,\eps}n\Delta) - \alpha_n\Big).
\end{align*}
\end{theorem}

The construction of the test we use to derive Theorem \ref{thm:LB} is actually quite involved, we refer to Section \ref{4tests} for more details. Here, we only illustrate the main ideas. 

First, the intuitions behind  the quantities $u^{+}$ and $u^{*}$ are the following. The quantity $u^+$ is chosen such that, with probability going to $1$, there is (i) at least one jump larger than $u^+$ but (ii) not too many of such jumps, \textit{i.e.} less than $2\log(e\lor n)$, and finally (iii) at most one jump larger than $u^+$ per time increment $\Delta$. Therefore, the discretized process of jumps larger than $u^+$ and smaller than $\eps$ (in absolute value), does not look Gaussian at all. It is composed of many null entries and a few larger than $u^+$. Now $u^*$ is the largest quantity (larger than $u^+$) such that $u^*$ is smaller than a number slightly larger than the standard deviation of the increment $X_{i\Delta}(\eps) - X_{(i-1)\Delta}(\eps)$ \textit{conditional to the event} there are no jumps larger than $u^*$ in $((i-1)\Delta,i\Delta]$. In other words, any increment having a jump larger than $u^*$ is going to be quite visible.

Second, the idea behind the test, is to build an event that occurs with high probability if $\nu_\eps =0$ and with small probability otherwise. This would then allow to bound the total variation distance between the two discretized processes using Lemma \ref{lem:TvTest}. The sketch of the proof  is the following (all results mentioned below are stated and proved in Section \ref{proofs} and Appendix \ref{appA}):
\begin{itemize}
\item First, we show that $u^{+}$ defined as above satisfies (i)-(ii)-(iii) with probability going to 1 and we bound the deviations of the difference of some\footnote{Of all increments when $\nu_\eps = 0$ and of those where a jump larger than $u^*$ occurs otherwise.} of the increments $X_{2i\Delta}(\eps) - X_{(2i-1)\Delta}(\eps) - (X_{(2i-1)\Delta}(\eps) - X_{2(i-1)\Delta}(\eps))$ (Lemma \ref{lem:defxi}).
\item Second, we build an estimator of the standard deviation of the increments of the L\'evy process $(b,\Sigma^2,\nu_{u^+})$. In order to do so, we use a robust estimator of the mean which drops large increments, and thus the ones larger than $u^+$ (Lemma \ref{lem:med}).
\item From these preliminary steps, we prove that a test comparing the largest entry and the expected standard deviation if $\nu_{\eps}=0$ detects if there is a jump larger than $u^*$ in the sample (Proposition \ref{prop:testmax}). In the following steps, we focus on tests \textit{conditional to the event there is no jump larger than $u^*$ in  the sample} - otherwise they are eliminated by the latter test. Two cases remain to be studied.
\item If the dominant quantity in Theorem \ref{teo:TVgaus} is $\Delta \mu_3(u^*)$: we first construct a test for detecting if $\Delta \mu_6(u^*)$ is larger than a constant times $[\Delta(\Sigma^2+\sigma^2(u^*))]^3$, to remove distributions that are too skewed (Proposition \ref{prop:test6}). Then, we build a test comparing the (estimated) third moment of the increments to the expected behavior if $\nu_{\eps}=0$ (Proposition \ref{prop:LBNsym}).
\item If the dominant quantity  in Theorem \ref{teo:TVgaus} is $\Delta \mu_4(u^*)$: we build a test comparing the (estimated) fourth moment of the increments to the expected behavior if $\nu_{\eps}=0$ (Proposition \ref{prop:LBsym}).
\end{itemize}

\subsubsection{Comments}

\paragraph{Tightness of the lower bound on the total variation distance.} The bounds on the total variation we establish in Theorems \ref{cor:TVgaus} and \ref{thm:LB} are tight, up to a $\log(e\lor n)$ factor, due to the differences in the definitions\footnote{Recall that $u^*$ is the largest $u$ larger than $u^+$ such that $u \leq \sqrt{\Delta (\Sigma^2 + \sigma^2(u))}\log(e\lor n)^2$, and $\tilde u^*$ is the largest $u$ larger than $u^+$ such that $u \leq \tilde c\sqrt{\Delta (\Sigma^2 + \sigma^2(u))}/\sqrt{\log(e\lor n)}$ where $\tilde c$ is a constant.} of $\tilde u^*$ and $u^*$, in the following sense. Whenever $$r_{n}(u^{*}):=\big(\lambda_{u^*,\eps}n\Delta\big) \lor \frac{n\mu_4^2(u^*)}{\Delta^{2} (\Sigma^2+\sigma^2(u^*))^4} \lor \frac{n\mu_3^2(u^*)}{\Delta(\Sigma^2+\sigma^2(u^*))^{3}}$$ does not converge to $0$ with $n$, the total variation distance does not converge to $0$. If $r_{n}(u^{*})$ converges to $+\infty$ with $n$, the total variation converges to $1$. Moreover, if $r_{n}(\tilde u^{*})$ converges to $0$ with $n$, then the total variation converges to $0$ by Theorem \ref{cor:TVgaus}. 
Another implication of these bounds is that the Gaussian random variable closest to $(X_\Delta(\eps))^{\otimes n}$ is not necessarily $\No(b(\eps)\Delta,\Delta(\Sigma^2 +\sigma^2(\eps))^{\otimes n}$, \textit{e.g.} when rare and large jumps are present, a tighter Gaussian approximation is provided by $\No(b(\eps)\Delta,\Delta(\Sigma^2 +\sigma^2(\tilde u^*))^{\otimes n}$, as pointed out in Section \ref{comments}.

\paragraph{The lower bound on the total variation distance is a jump detection test.} The proof of Theorem \ref{thm:LB} is based on the construction of a test of Gaussianity, adapted to Lévy processes, that detects whether the discrete observations we have at our disposal are purely Gaussian, or whether they are realizations of a Lévy process with non trivial Lévy measure. More precisely, (see the proof of Theorem \ref{thm:LB} for details) we build a uniformly consistent test for the testing problem
\begin{align*}
&H_{0}: \nu_\eps = 0,\hspace{1cm}\mbox{against}\hspace{1cm}
H_{1} :  \lambda_{0,\eps} = +\infty~~~\text{and}~~E
\end{align*}
where\small{
\begin{align*}E=
\Big\{&\mu_{3}(u^*)^2 \geq {\frac{C\Delta(\Sigma^2+\sigma^2(u^*))^3}{n}}\ \ \text{or}\ \  \mu_{4}(u^*)^2 \geq \frac{C\Delta^2(\Sigma^2+\sigma^2(u^*))^4}{n}
\text{or}\ \mbox{A jump larger than}~u^*~\mbox{occurs}\Big\}.
\end{align*}}
This test is of interest in itself: it does not rely on the knowledge of the Lévy triplet.

\paragraph{Remark on the assumptions.} Theorem~\ref{cor:LB} requires $\eps \leq \sqrt{\Delta (\Sigma^2 + \sigma^2(\eps))}\log(e\lor n)^2$, \textit{i.e.} that $\eps$ is smaller (up to a multiplicative $\log(e\lor n)^2$ term) than the standard deviation of the increment. It implies that all moments of order $k$ of the increment can be efficiently bounded -up to a constant depending on $k$- by $\big(\sqrt{\Delta (\Sigma^2 + \sigma^2(\eps))}\log(e\lor n)^2\big)^k$, which is helpful for bounding the deviations of the test statistics. This assumption is restrictive and is removed in Theorem~\ref{thm:LB} by introducing $u^*$ and considering two different types of tests in the construction of the lower bound: a test for the third and fourth moments and a test for extreme values. This latter test allows to detect -with very high probability- when a jump larger than $u^*$ occurred.

Therefore, both theorems only rely on the assumption that $\lambda_{0,\eps} \geq \Delta^{-1} \lor ({\log(e\lor n)}/({n\Delta}))$. This bound is larger than ${\log(e\lor n)}/{\Delta}$ (see Theorems \ref{teo:TVgaus} and \ref{cor:TVgaus}). As explained in Section~\ref{comments}, whenever $\lambda_{0,\eps}$ is smaller than $\frac{\log(e\lor n)}{\Delta}$ (up to a multiplicative constant) simple arguments enable to bound the total variation distance when $\Sigma=0$. In this sense, assumption $\lambda_{0,\eps} \geq \Delta^{-1} \lor ({\log(e\lor n)}/({n\Delta}))$ is not constraining as it permits to treat all relevant cases.

\paragraph{Improvement of Theorem~\ref{cor:LB} for mixtures.}  An immediate corollary of Theorem~\ref{cor:LB}  (see its proof) is a lower bound on the total variation distance between any two mixture of Gaussian random variables and mixture of Lévy measures concentrated in $[-\eps,\eps]$. 
More precisely, let $d\Lambda(b,\Sigma^2,\nu_\eps)$ et $d\Lambda'(b,\Sigma^2)$ be two priors on L\'evy processes and linear Brownian motions, respectively. Assume that the support of $d\Lambda(b,\Sigma^2,\nu_\eps)$ is included in a set $\mathcal A$, and that for any $(b,\Sigma^2,\nu_\eps) \in \mathcal A$, we have $\eps \leq \sqrt{(\sigma^2(\eps)+\Sigma^2)\Delta} \log(e\lor n)^2$. Then, it holds
\begin{align*}
\bigg\|\int (\No(b\Delta,\Delta\Sigma^2)*&M^{(\nu)}_\Delta(\eps))^{\otimes n}d\Lambda(b,\Sigma^2,\nu)-\int\No(b\Delta,\Delta (\Sigma^2+\sigma^{2}(\eps)))^{\otimes n}d\Lambda'(b,\Sigma^2)\bigg\|_{TV}\geq \nonumber \\
&\min_{(b,\Sigma^2,\nu) \in \mathcal A}\Bigg[1 - C\Bigg[\frac{\Delta(\Sigma^2+\sigma_\nu^2(\eps))^3}{n\big(\mu_3^{(\nu)}(\eps)\big)^2}\land \frac{\Delta^2(\Sigma^2+\sigma_\nu^2(\eps))^4}{n\big(\mu_4^{(\nu)}(\eps)\big)^2}\Bigg] - \alpha_n\Bigg],
\end{align*}
where 
$M^{(\nu)}_\Delta(\eps), \sigma_\nu^2(\eps), \mu_3^{(\nu)}(\eps)$ and $\mu_4^{(\nu)}(\eps)$ correspond to $M_\Delta(\eps), \sigma^2(\eps), \mu_3(\eps)$ and $\mu_4(\eps)$ for the L\'evy measure $\nu_\eps$. A related result can be achieved for Theorem~\ref{thm:LB}. Note that the corresponding lower bound on the total variation distance is a direct corollary of Theorem~\ref{teo:TVgaus}. The lower bound displayed above is not trivial, it holds because the test that we construct in the proof of Theorem~\ref{thm:LB} does not depend on the parameters of the Gaussian random variable nor on the L\'evy triplet.

\subsection{Examples}\label{ex}
\subsubsection{Preliminaries on Assumption~\eqref{ass:psiK}}

Before displaying the results implied by Theorems \ref{teo:TVgaus} and \ref{cor:TVgaus} on the class of $\beta$-stable processes, we provide two contexts in which Assumptions \eqref{ass:psiK} and \eqref{ass:psiK2} are fulfilled. 
\paragraph{When $\Sigma$ is large enough.}

We first present the following proposition which proves that Assumption~\eqref{ass:psiK} is satisfied, whenever $\Sigma$ is large enough - namely, $\sigma(\eps) \lesssim \Sigma$.

\begin{proposition}\label{prop:Siglar}
 Let $\eps >0$ and consider a L\'evy  measure $\nu_\eps$ with support in $[-\eps,\eps]$. Assume that $\lambda_{0,\eps} > 1/\Delta$ and $\eps \leq \sqrt{\Delta (\Sigma^2 + \sigma^2(\eps))}$  and for a constant $c_\Sigma>0$ it holds that $c_{\Sigma}\Sigma  \geq \sigma(\eps)$. Then, there exists a constant $c>0$, that depends only on $c_\Sigma$, such that
$$\int_{c\log(e\vee n)}^\infty|\Psi_{\eps}^{(k)}(t)|^2 dt \leq k! n^{-4},\quad \forall k\in [0,401\log(e\lor n)].$$
\end{proposition}
In this case we can apply directly Theorem~\ref{teo:TVgaus} (and Theorem~\ref{cor:TVgaus}, provided that we apply the previous proposition at $\tilde u^*$ instead of $\eps$).

 \begin{proposition}\label{prop:Siglar2}
 Let $\eps >0$ and consider a L\'evy  measure $\nu_\eps$ with support in $[-\eps,\eps]$. Assume that $\Sigma$ is such that $c_{\Sigma}\Sigma \geq \sigma(\eps)$, for some constant $c_\Sigma>0$, and that $\lambda_{0,\eps} \geq 24 \log(e\lor n)/\Delta$. Then, there exists two constants $C>0$, $\tilde c>0$, that depend only on $c_\Sigma$, such that  
\begin{align*}
&\min_{B \in \mathbb R, S^2\geq 0}\big\|(X_\Delta(\eps))^{\otimes n}-\No(B\Delta,\Delta S^2)^{\otimes n}\big\|_{TV}\leq 1 - e^{-\tilde\lambda^{*}\Delta n}\nonumber \\
& \quad +Ce^{-\tilde\lambda^{*}n\Delta} \bigg(\sqrt{\frac{n\mu_4^2(\tilde u^*(\tilde c))}{\Delta^{2} (\Sigma^2+\sigma^2(\tilde u^*(\tilde c)))^4}+\frac{n\mu_3^2(\tilde u^*(\tilde c))}{\Delta(\Sigma^2+\sigma^2(\tilde u^*(\tilde c)))^{3}}}\bigg) +\frac{C}{n} ,
\end{align*} where we set $\tilde\lambda^{*}:=\lambda_{\tilde u^*(\tilde c),\eps}$.
\end{proposition}

\paragraph{When $\nu_\eps$ is polynomially controlled at 0.}The following result, whose proof can be found in Appendix \ref{app:stable}, implies that whenever $\nu_\eps$ satisfies Assumption~\eqref{eq:nuRV} below, then Assumption \eqref{ass:psiK} is fulfilled.  Assumption~\eqref{eq:nuRV} describes a class of functions that contains any L\'evy measure that is regularly varying at $0$.
\begin{proposition}
\label{prop:exstable} Let $b\in\R$, $\Sigma^{2}\geq0$, $\Delta>0$, $\eps>0$, $n\geq 1$ and let $\nu$ be a L\'evy  measure absolutely continuous with respect to the Lebesgue measure. Suppose that there exists two positive constants $c_+>c_->0$ such that, $\forall x\in[-\eps,\eps]\setminus\{0\} $,
\begin{align}\label{eq:nuRV}\frac{c_-}{|x|^{\beta+1}} \leq \frac{d\nu(x)}{dx} \leq \frac{c_+}{|x|^{\beta+1}},\quad \beta\in(0,2).\end{align}
Assume that there exists $c_{\max}\geq 0$ such that $n^{c_{\max}}\Delta\geq 1$ and $ \log(\Sigma^2 +\sigma^2(\eps))/\log(e\lor n) \leq c_{\max}$.
Then, for any $c>0$ large enough depending only on $\beta, c_+, c_-, c_{\max}$, there exists $\tilde c<1$ small enough depending only on $\beta, c_+, c_-, c_{\max},c$ such that if
${\eps} \leq \tilde c {\sqrt{\Delta (\Sigma^2+\sigma^2(\eps))}}/\sqrt{\log(e\lor n)}$, then it holds that
\begin{align*}
\int_{t\geq c \log(e\lor n)}|\Psi_\eps^{(k)}(t)|^2dt \leq 3 k! n^{-4},\quad \forall k\in[0,401 \log(e\lor n)].
\end{align*}
\end{proposition}

\begin{remark}
Whenever there exists $\kappa>0$ a constant that depends only on $\beta, c_+, c_-$ such that $n^{\kappa}\Delta \geq 1$, and $(\Sigma^2 + \eps^{2-\beta}) n^{-\kappa} \leq 1$, then $c_{\max}$ is an absolute constant and  the dependence on $c_{\max}$ in Proposition \ref{prop:exstable} is not constraining. Moreover, the condition on ${\eps}$ is the same condition as in Theorems \ref{teo:TVgaus} and \ref{cor:TVgaus}. Finally, in Theorems \ref{teo:TVgaus} and \ref{cor:TVgaus} the constraints  on $c, \tilde c$ are $c>1$, $\tilde c \leq 1$ and  $c\tilde c \leq \sqrt{\log(e\lor n)}/4$, that are easy to satisfy provided that $\tilde c$ can be chosen small enough. As $\Sigma^2+\sigma^2(\eps)$ is of order $( \Sigma^2 + \eps^{2-\beta})$, even in the most constraining case $\Sigma = 0$, $\tilde c$ can be chosen small enough provided that 
$\eps^\beta \leq \tilde c'\Delta/\log(e\lor n),$
where $\tilde c'$ is chosen small enough (depending on $\tilde c, \beta, c_+, c_-, c_{\max}$).
\end{remark}

Then, we state the following general result which is a consequence of Proposition \ref{prop:exstable} and Theorem \ref{teo:TVgaus}.
\begin{theorem}\label{thm6}
Let $b\in\R$, $\Sigma^{2}\geq0$, $\Delta>0$, $\eps>0$, $n\geq 1$ and let $\nu$ be a L\'evy  measure absolutely continuous with respect to the Lebesgue measure. Suppose that there exists two positive constants $c_+>c_->0$ such that, $\forall x\in[-\eps,\eps]\setminus\{0\} $,
\[\frac{c_-}{|x|^{\beta+1}} \leq \frac{d\nu(x)}{dx} \leq \frac{c_+}{|x|^{\beta+1}},\quad \beta\in(0,2).\]
Assume there exists $\kappa>0$ depending only on $\beta, c_+, c_-$ such that $n^{\kappa}\Delta \geq 1$. The following results hold

 \begin{enumerate} 
 \item If $(\Sigma^2 + \eps^{2-\beta}) n^{-\kappa} \leq 1$, there exist two constants $C>0, \tilde c>0$ that depend only on $\beta, c_+, c_-, \kappa$ such that
\begin{align*}
\min_{B \in \mathbb R, S^2\geq 0}\big\|(X_\Delta(\eps))^{\otimes n}&-\No(B\Delta,\Delta S^2)^{\otimes n}\big\|_{TV} \leq 1 - e^{-\lambda_{\tilde u^*(\tilde c),\eps}\Delta n}\\
&\quad+Ce^{-\lambda_{\tilde u^*(\tilde c),\eps}n\Delta} \bigg(\sqrt{\frac{n\mu_4^2(\tilde u^*(\tilde c))}{\Delta^{2} (\Sigma^2+\sigma^2(\tilde u^*(\tilde c)))^4}+\frac{n\mu_3^2(\tilde u^*(\tilde c))}{\Delta(\Sigma^2+\sigma^2(\tilde u^*(\tilde c)))^{3}}} \bigg) +\frac{C}{n}.\label{eq:cor22}
\end{align*}

\item  If $(\Sigma^2 + \eps^{2-\beta}) n^{-\kappa} \leq 1$ and $\eps\leq \tilde c\sqrt{(\Sigma^{2}+\sigma^{2}(\eps))n^{-\kappa}}/\sqrt{\log (e\lor n)}$, it holds for some constant $C>0$, depending on $\beta, c_+, c_-, \kappa,$
\begin{align*}
\big\|(X_\Delta(\eps))^{\otimes n}&-\No\big(b(\eps)\Delta,\Delta (\Sigma^2+\sigma^{2}(\eps)\big)^{\otimes n}\big\|_{TV}\leq 
 C \sqrt{\frac{n\mu_4^2(\eps)}{\Delta^{2} (\Sigma^2+\sigma^2(\eps))^4}+\frac{n\mu_3^2(\eps)}{\Delta(\Sigma^2+\sigma^2(\eps))^{3}}}.
\end{align*}
 
\end{enumerate}

\end{theorem}

\subsubsection{Stable processes}

In this Section we illustrate the implications of Theorem \ref{thm6} 
 on the class of infinite stable processes. It is possible to extend the results valid for this example to other types of Lévy processes (e.g. inverse Gaussian processes, tempered stable distributions, etc...) as, around 0, stable measures well approximate many Lévy measures.  Let $\beta\in(0,2)$, $c_{+},c_{-}\geq 0$, ($c_{+}=c_{-}$ if $\beta=1$) and assume that the Lévy measure $\nu$ has a density with respect to the Lebesgue measure of the form 
\[\nu(dx)=\frac{c_{+}}{x^{1+\beta}}\mathbf{1}_{(0,+\infty)}(x)dx+\frac{c_{-}}{|x|^{1+\beta}}\mathbf{1}_{(-\infty,0)}(x)dx,\quad \forall x\in [-\eps,\eps]\setminus\{0\}.\]
These processes satisfy Equation~\eqref{eq:nuRV}. Let $M(\eps)$ be a Lévy process with Lévy triplet $(-\int_{1<|x|\leq \eps} x\nu_\eps(dx),0,\nu_\eps)$ where $\nu_\eps:=\nu\1_{|x|\leq \eps}$ and $b>0$, $\Sigma^{2}\geq0$, $\Delta>0$, $\eps>0$, $n\geq 1$. In the sequel, we use the symbols $\approx,\ \lesssim,$ and $ o(1)$ defined as follows. For $a, b\in \mathbb R$, $a \approx b$ if there exists $c>0$ depending only on $\beta, c_+, c_-$ such that $a = cb$ and $a\lesssim b$ if there exists $c>0$ depending only on $\beta, c_+, c_-$ such that $a \leq cb$. For a sequence $(a_n)_n$ in $\mathbb R^+$, we have that $a_n = o(1)$ if $\lim_{n\rightarrow \infty} a_n = 0$. \\

We are interested in the question: ``Given $n$ and $\Delta$, what is the largest (up to a constant) $\eps^* \geq 0$ such that it is not possible to distinguish between $n$ independent realizations of  
$\No(b(\eps^*)\Delta,\Delta\Sigma^2)*M_\Delta(\eps^*)$ and the closest i.i.d.~Gaussian vector?" 
The answer to this question is provided by Theorem \ref{thm6}. 
The following two Tables summarize these findings and give the ordre of magnitude of $\eps^{*}$. We distinguish four scenarios (depending on whether the process is symmetric or not, and on whether $\Sigma$ is large with respect to $\sigma^{2}(\eps^{*})$ or not) and provide for each the optimal rate of magnitude for $\eps^*$ such that (i) if $\eps/\eps^* = o(1)$, then 
$$\inf_{B\in \mathbb R, S\geq 0}\Big\|\big(\No(b(\eps)\Delta,\Delta\Sigma^2)*M_\Delta(\eps)\Big)^{\otimes n} - \Big(\No(B\Delta,\Delta S^2)\big)^{\otimes n}\Big\|_{TV} \rightarrow 0,$$
and (ii) else if $\eps^*/\eps = o(1)$, then
$$\inf_{B\in\mathbb R, S\geq 0}\Big\|\big(\No(b(\eps)\Delta,\Delta\Sigma^2)*M_\Delta(\eps)\big)^{\otimes n} - \big(\No(B\Delta,\Delta S^2)\big)^{\otimes n}\Big\|_{TV} \rightarrow 1.$$
In all cases we require, additionally to $\nu$ being the Lévy measure of a $\beta$-stable process, that there exists a constant $\kappa>0$ that depends only on $\beta, c_+, c_-$ such that $n^{\kappa}\Delta \geq 1$, and $(\Sigma^2 + \eps^{2-\beta}) n^{-\kappa} \leq 1$. \begin{table}[h!!]\begin{center}
\begin{tabular}{|c||c|}
\hline
 \multicolumn{2}{|c|}{ $\nu$ \ is symmetric}  \\

\hline
\hline
 ${\Sigma^2 \gtrsim \big(\frac{\Delta}{n}\big)^{\frac{2-\beta}{\beta}}}$ & $\eps^*\approx\big(\frac{\Delta \Sigma^4}{\sqrt{n}}\big)^{\frac{1}{4-\beta}}$\\
\hline
  $\Sigma^2 \lesssim \big(\frac{\Delta}{n}\big)^{\frac{2-\beta}{\beta}} $  &$\eps^*\approx \big(\frac{\Delta}{\sqrt{n}}\big)^{\frac{1}{\beta}}$ \\
\hline

\end{tabular}
\begin{tabular}{|c||c|}
\hline
 \multicolumn{2}{|c|}{ $\nu$ \ is non symmetric}  \\
\hline
\hline
 $\Sigma^2 \gtrsim \big(\frac{\Delta}{\sqrt{n}}\big)^{\frac{2-\beta}{\beta}}$ &  $\eps^*\approx\big(\frac{\sqrt{\Delta} \Sigma^3}{\sqrt{n}}\big)^{\frac{1}{3-\beta}}$\\
\hline
  $\Sigma^2 \lesssim \big(\frac{\Delta}{\sqrt{n}}\big)^{\frac{2-\beta}{\beta}} $  &$\eps^*\approx \big(\frac{\Delta}{n}\big)^{\frac{1}{\beta}}$ \\
\hline

\end{tabular}
\end{center}
\end{table}

\section{Total variation distance between Lévy processes}\label{sec:TV}

In this Section, let $X^{i}\sim(b_{i},\Sigma^{2}_{i},\nu_{i})$, $i=1,2$, be two distinct Lévy processes. We shall use the notation introduced in Section \ref{not} properly modified to take into account the dependencies on $X^1$ and $X^2$. For instance, $\mu_3(\eps)$ and $\mu_4(\eps)$ become 
\[
\mu_{j,i}(\eps)=\int_{|x|\leq\eps} x^j\nu_i(dx),\quad i=1,2,\ j=3,4,\]
where $\mu_{j,1}(\eps)$ (resp. $\mu_{j,2}(\eps)$), $j=3,4$, denote the $3$rd and $4$th moment of $\nu_1$ (resp. $\nu_2$) restricted on $\{ x:|x|\leq\eps\}$.

By means of the Lévy-Itô decomposition recalled in Section \ref{not}, for any $t>0$ and $\eps>0$ we have that the law of $X_t^ i$, $i=1,2$, is the convolution between a Gaussian distribution and the law of the marginal at time $t$ of the processes $M^i(\eps)$ and $Z ^i(\eps)$, i.e.
$$X_t ^i(\eps)=\No\big(b_i(\eps)t, t\Sigma^2_i\big)*M_t ^i(\eps)* Z_t^i(\eps),\quad i=1,2.$$
By subadditivity of the total variation distance, see Lemma \ref{subadditivity} in Appendix \ref{appB}, for any $n\geq 1$ we have:
\begin{align*}
\|(X_t^1)^{\otimes n}-(X_t^2)^{\otimes n}\|_{TV}&\leq \|(\No\big(b_1(\eps)t, t\Sigma^2_1\big)*M_t ^1(\eps))^{\otimes n}-(\No\big(b_2(\eps)t, t\Sigma^2_2\big)*M_t ^2(\eps))^{\otimes n}\|_{TV}\\
&\quad+\|(Z_t ^1(\eps))^{\otimes n}-(Z_t ^2(\eps))^{\otimes n}\|_{TV}.
\end{align*}
By triangular inequality first and subadditivity of the total variation distance together with Lemma \ref{TV:gaussians}  in Appendix \ref{appB} then, we obtain
\begin{align*}
\|(\No\big(b_1(\eps)t, t\Sigma^2_1\big)*M_t ^1(\eps))^{\otimes n}&-(\No\big(b_2(\eps)t, t\Sigma^2_2\big)*M_t ^2(\eps))^{\otimes n}\|_{TV}\\\leq &\|(\No\big(b_1(\eps)t, t\Sigma^2_1\big)*M_t ^1(\eps))^{\otimes n}-\big(\No\big(b_1(\eps)t, t\Sigma^2_1\big)*\No\big(0, t\sigma^2_1(\eps)\big)\big)^{\otimes n}\|_{TV}\\
&\quad +\|\big(\No\big(b_2(\eps)t, t\Sigma^2_2\big)*M_t ^2(\eps)\big)^{\otimes n}-\big(\No\big(b_2(\eps)t, t\Sigma^2_2\big)*\No\big(0, t\sigma^2_2(\eps)\big)\big)^{\otimes n}\|_{TV}\\
&\quad +\|\big(\No\big(b_1(\eps)t, t\Sigma^2_1\big)*\No\big(0, t\sigma^2_1(\eps)\big)\big)^{\otimes n}-\big(\No\big(b_2(\eps)t, t\Sigma^2_2\big)*\No\big(0, t\sigma^2_2(\eps)\big)\big)^{\otimes n}\|_{TV}\\
&\leq \sum_{i=1}^2\|\big(\No\big(b_i(\eps)t, t\Sigma^2_i\big)*M_t ^i(\eps)\big)^{\otimes n}-\big(\No\big(b_i(\eps)t, t(\sigma^2_i(\eps)+\Sigma^2_i)\big)\big)^{\otimes n}\|_{TV}\\
&\quad +\frac{\sqrt{\frac{t}{2\pi}} \big|b_1(\varepsilon)-b_2(\varepsilon)\big|+\Big|\sqrt{\Sigma_1^2+\sigma_1^2(\varepsilon)}-\sqrt{\Sigma_2^2+\sigma_2^2(\varepsilon)}\Big|}{\sqrt{\Sigma_1^2+\sigma_1^2(\varepsilon)}\vee \sqrt{\Sigma_2^2+\sigma_2^2(\varepsilon)}}.
\end{align*}
The terms $\|\big(\No\big(b_i(\eps)t, t\Sigma^2_i\big)*M_t ^i(\eps)\big)^{\otimes n}-\big(\No\big(b_i(\eps)t, t(\Sigma_i^2+\sigma^2_i(\eps)\big)\big)^{\otimes n}\|_{TV}$, $i=1,2$ can be bounded by means of Theorem \ref{teo:TVgaus} whereas for $\|(Z_t ^1(\eps))^{\otimes n}-(Z_t ^2(\eps))^{\otimes n}\|_{TV}$ we can use Lemma \ref{TV:CPP}: 
$$\|(Z_t ^1(\eps))^{\otimes n}-(Z_t ^2(\eps))^{\otimes n}\|_{TV}\leq nt\big|\Lambda_1(\varepsilon)-\Lambda_2(\varepsilon)\big| +nt\big(\Lambda_1(\varepsilon)\wedge\Lambda_2(\varepsilon)\big)
 \Big\|\frac{\nu_1^\varepsilon}{\Lambda_1(\varepsilon)}-\frac{\nu_2^\varepsilon}{\Lambda_2(\varepsilon)}\Big\|_{TV},$$
with $\nu_j^\varepsilon=\nu_j(\cdot\cap (\R\setminus[-\varepsilon,\varepsilon]))$ and $\Lambda_j(\varepsilon)=\nu_j^\varepsilon(\R)$.
We thus obtain the following upper bounds for the total variation distance between $n$ equidistant observations of the increments of Lévy processes:
\begin{theorem}\label{TVn}
Let $X^i\sim (b_i,\Sigma_i^2,\nu_i)$ be any Lévy process with $b_i\in \R$, $\Sigma_i\geq 0$ and $\nu_i$ Lévy measures $i=1,2$. For all $\Delta>0$, $\varepsilon>0$ and $n\geq 1$ and under the Assumptions of Theorem \ref{teo:TVgaus}, there exists a positive constant $C$ such that
 \begin{align*}
  \|(X_{k\Delta}^1-X_{(k-1)\Delta}^1)_{k=1}^n&-(X_{k\Delta}^2-X_{(k-1)\Delta}^2)_{k=1}^n\|_{TV}\leq \frac{\sqrt{n\Delta}}{\sqrt{2\pi}}\frac{|b_1(\eps)-b_2(\eps)|}{\max(\sqrt{\Sigma_1^2+\sigma_1^2(\eps)},\sqrt{\Sigma_2^2+\sigma_2^2(\eps)})}\nonumber \\ \nonumber
  &\quad +1-\bigg(\frac{\min(\sqrt{\Sigma_1^2+\sigma_1^2(\eps)},\sqrt{\Sigma_2^2+\sigma_2^2(\eps)})}{\max(\sqrt{\Sigma_1^2+\sigma_1^2(\eps)},\sqrt{\Sigma_2^2+\sigma_2^2(\eps)})}\bigg)^n\\
&\quad +C\sum_{i=1}^2 \sqrt{\frac{n\mu_{4,i}^2(\eps)}{\Delta^{2} (\sigma_i^2(\eps)+\Sigma_i^2)^4}+\frac{n\mu_{3,i}^2(\eps)}{\Delta(\sigma_i^2(\eps)+\Sigma_i^2)^{3}}}+\frac{2C}{n}\\
&\quad +1-\exp\big(-n\Delta \big|\Lambda_1(\varepsilon)-\Lambda_2(\varepsilon)\big|\big)
 +n\Delta\big(\Lambda_1(\varepsilon)\wedge\Lambda_2(\varepsilon)\big)
 \Big\|\frac{\nu_1^\varepsilon}{\Lambda_1(\varepsilon)}-\frac{\nu_2^\varepsilon}{\Lambda_2(\varepsilon)}\Big\|_{TV}. \nonumber
\end{align*}
\end{theorem}
\begin{proof}
It directly follows from the Lévy-Itô decomposition together with Lemmas \ref{subadditivity}, \ref{TV:gaussians}, \ref{TV:CPP} and Theorem \ref{teo:TVgaus}. 
\end{proof}

\section{Proofs}\label{proofs}
Several times a compound Poisson approximation for the small jumps of Lévy processes will be used in the proofs, see e.g. the proofs of Theorems \ref{cor:TVgaus} and \ref{thm:LB}. 
More precisely, for any $0<\eta<\eps$, we will denote by $M(\eta,\eps)$ the centered compound Poisson process that approximates $M(\eps)$ as $\eta\downarrow0$, i.e.
\begin{equation}\label{cppapprox}
M_t(\eta,\eps)=\sum_{s\leq t}\Delta X_s\1_{\eta<|\Delta X_s|\leq \varepsilon}-t\int_{\eta<|x|\leq \varepsilon}x\nu(dx)=\sum_{i=1}^{N_t(\eta,\eps)} Y_i-t\int_{\eta<|x|\leq \varepsilon}x\nu(dx),
\end{equation}
where 
$N(\eta,\eps)$ is a Poisson process with intensity $\lambda_{\eta,\eps}:=\int_{\eta<|x|\leq \eps}\nu(dx)$ and the
$(Y_i)_{i\geq 1}$ are i.i.d. random variables with jump measure
\begin{equation}\label{cppY}
\PP(Y_1\in B)=\frac1{\lambda_{\eta,\eps}}{\int_{B\cap \{\eta<|x|\leq \eps\}}\nu(dx)},\quad \forall B\in\mathscr B(\R).
\end{equation}
Then, it is well known (see e.g. \cite{sato}) that $M(\eta,\eps)$ converges to $M(\eps)$ almost surely and in $L_2$, as $\eta\to0$.

\subsection{Proof of Theorem \ref{teo:TVgaus}}

\subsubsection{Assumptions and notations\label{sec:prfthm2}}

We begin by introducing some notations and by reformulating the assumptions of Theorem \ref{teo:TVgaus}. For a real function $g$ and an interval ${\mathcal I}$, we write $g_{\mathcal I} := g \mathbf 1_{\{\mathcal I\}}$. Given a density $g$ with respect to a probability measure $\mu$ and a measurable set $A$, we denote by $\PP_g(A)=\int_A g(x)\mu(dx)$.
Also, we denote by $s^{2}:=\Sigma^{2}+\sigma^{2}(\eps)$.
In what follows, we write $\mu$ for a measure that is the sum of the Lebesgue measure and (countably) many Dirac masses, which dominates the measure associated to $\tilde X_{\Delta}(\eps)=(X_\Delta(\eps) - b(\eps)\Delta)/\sqrt{\Delta (\sigma^2(\eps)+\Sigma^2)}$. Moreover, $f$ will indicate the density, with respect to the measure $\mu$, of the rescaled increment $\tilde X_{\Delta}(\eps)$ and $\varphi$ will be the density, with respect to the measure $\mu$, of a centered Gaussian random variable with unit variance. Whenever we write an integral involving $f$ or $\varphi$ in the sequel, it is with respect to $\mu$ (or the corresponding product measure).

Recall that
\begin{align*}\Psi(t) :&=\Psi_\eps(t) = \E[e^{it\tilde X_{\Delta}(\eps)}] - e^{(-\lambda_{0,\eps}n\Delta )}\mathbf{1}_{\{\Sigma=0\}}\\&= \exp\Big(- \frac{\Sigma^2}{s^2}\frac{t^2}{2} +\Delta \int \Big(\exp\Big(\frac{iut}{s\sqrt{\Delta}}\Big) - \frac{iut}{s\sqrt{\Delta}} - 1\Big) d\nu(u)\Big)-  e^{(-\lambda_{0,\eps}n\Delta )}\mathbf{1}_{\{\Sigma=0\}}.\end{align*}
We establish the result under the following assumptions which are implied by the assumptions of Theorem \ref{teo:TVgaus}. Let $\mathcal{I}$ be an integration interval of the form $\mathcal{I}:=[-c_{sup}\sqrt{\log(n)}, c_{sup}\sqrt{\log(n)}],$ with $c_{sup}\geq2$ and let us assume here that $n \geq 3$ - but note that the bound on the total variation distance for $n=3$ is also a bound on the total variation distance for $n=1$ or $n=2$.

\begin{itemize}
\item  Set $K:=c_{int}^{2}\log(n)$, where $c_{int}>2c_{sup}$. Then, for some constant $c>1$ it holds that
\begin{align}
 \int_{c{\log(n)}}^{+\infty}|\Psi^{(k)}|^2 \leq C'2^k k! n^{-2},\quad \forall \ 0\leq k\leq K,\tag{$\mathcal{H}_{\Psi}$}\end{align}
where $C'$ is a universal constant.

\item For some constant $0< c_p <1/8$, it holds 
\begin{align}\label{ass:0}\mathbb P_f({{\mathcal I}^{c}}) \leq c_p/n.\tag{$\mathcal{H}_{0}$}\end{align}

\item For some small enough universal constant $0< \tilde c\leq 1$, such that $\tilde cc\leq \sqrt{\log n}/4$, it holds
\begin{align}\label{ass:epsD}\eps \leq \tilde c \sqrt{(\sigma(\eps)^2+\Sigma^2)\Delta/\log(n)}:= \tilde c s\sqrt{\Delta} /{\sqrt{\log(n)}}:=\tilde c_{n}s\sqrt{\Delta}.\tag{$\mathcal{H}_{\eps}$}\end{align}
Note that this assumption permits to simply derive \eqref{ass:0} from the following lemma.
\begin{lemma}\label{lem:ass0}
For $\eps>0$, $\Delta>0$ and $n\geq 3$, let $\nu_\eps$ be a (possibly infinite) L\'evy measure such that
$\lambda_{0,\eps}\geq{24\log(n)}/{\Delta}.$
Then, whenever $c_{sup} \geq  10$, $\tilde c_n \leq 1$ and \eqref{ass:epsD} holds, we get
$\mathbb P_f(\mathcal I^c) \leq 3/n^3.$
\end{lemma}
Lemma \ref{lem:ass0} implies that under the assumptions of Theorem \ref{teo:TVgaus}, \eqref{ass:0} is satisfied with $c_{p}=3/n^{2}.$

\begin{remark}
For the proof of Theorem \ref{teo:TVgaus}, Assumption \eqref{ass:epsD} can be weakened in $\eps \leq \tilde c \sqrt{(\sigma(\eps)^2+\Sigma^2)\Delta} $, the extra log is used to establish Lemma \ref{lem:ass0} related to \eqref{ass:0}. 
\end{remark}

\item For some constant $0< c_{m}\leq 1/2$, it holds
\begin{align}\label{ass:M}M := \tilde c_{n}^{-4}\bigg(\frac{|\mu_3(\eps)|}{\sqrt{\Delta}s^3} + \frac{\mu_4(\eps)}{\Delta s^4}\bigg) \leq \frac{c_m}{\sqrt{n}}.\tag{$\mathcal{H}_{M}$}\end{align}
\begin{remark} Assumption \eqref{ass:M} will be used in the proof of Theorem \ref{teo:TVgaus}, it is not limiting as if \eqref{ass:M} is not satisfied, the upper bound of Theorem \ref{teo:TVgaus} is not small and is therefore irrelevant.
\end{remark}
\end{itemize}
In the sequel, $C$ stands for a universal constant, whose value may change from line to line.
\subsubsection{Proof of Theorem \ref{teo:TVgaus}.}
To ease the reading of the proof of Theorem \ref{teo:TVgaus}, we detail the case where there is a non-zero Gaussian component on $X(\eps)$ and/or the Lévy measure of $X(\eps)$ is infinite. The case of compound Poisson processes ($\lambda_{0,\eps}<\infty$) can be treated similarly, considering separately the sets $A_n = \{ \forall i\leq n, \,N_{i\Delta}(0,\eps) - N_{(i-1)\Delta}(0,\eps) \geq1\}$ and its complementary, where $N(0,\eps)$ is the Poisson process with intensity $\lambda_{0,\eps}$ associated to the jumps of $M(\eps)$, see \eqref{cppapprox}. The reason being that on the set $A_{n}$ the distribution of the process is absolutely continuous with respect to the Lebesgue measure, and on its complementary it is not. On the set $A_{n}$ the techniques employed below can be adapted, and on the complementary set $A_{n}^{\mathsf c}$, the total variation can be trivially controlled using $\lambda_{0,\eps}\geq 24\log(e\vee n)/\Delta.$\\

First, by means of a change of variable we get
$$\big\|(X_\Delta(\eps))^{\otimes n}-\No(b(\eps)\Delta,\Delta (\Sigma^{2}+\sigma^{2}(\eps)))^{\otimes n}\big\|_{TV}=\|f^{\otimes n}-\varphi^{\otimes n}\|_{TV}.$$
To bound the total variation distance we consider separately the interval $\mathcal{I}$ and its complementary (recall that the integrals are with respect to $\mu^{\otimes n}$):
\begin{align}\label{eq:TVI} \|f^{\otimes n} -\varphi^{\otimes n}\|_{TV}=\big\|f_{\mathcal I}^{\otimes n} -\varphi_{\mathcal I}^{\otimes n}\big\|_{TV}+\frac{1}{2}\int_{(\mathcal{I}^{n})^{c}} |f^{\otimes n}-\varphi^{\otimes n}|.
\end{align}
Under \eqref{ass:0} and using that $\PP_{\varphi^{\otimes n}}({\mathcal{I}^n})^{c}\leq n\big(\PP_{\varphi}(\mathcal{I}^{c})\big)\leq  n\exp(- c_{sup}^2 \log(n)/2) \leq 1/n$ for $ c_{sup} \geq 2$, the second term in \eqref{eq:TVI} is bounded by $\frac{1}{2}(c_p+\frac{1}{n})$. Let us now focus on the first term.
Introduce a positive function $h>0$ such that \[\int h_{\mathcal I} <+\infty,~~~\int \frac{f^2_{\mathcal I}}{h_{\mathcal I} } <+\infty,~~~\int \frac{\varphi^2_{\mathcal I}}{h_{\mathcal I}} <+\infty.\]
By the Cauchy-Schwarz inequality, we get 
\begin{align*}
&\frac{2\big\|(f_{\mathcal I}^{\otimes n} -\varphi_{\mathcal I}^{\otimes n})\big\|_{TV}^{2}}{(\int h_{\mathcal I})^n}
\leq \int \frac{(f_{\mathcal I}^{\otimes n} - \varphi_{\mathcal I}^{\otimes n})^2 }{h_{\mathcal I}^{\otimes n}}
= \int  \frac{(f_{\mathcal I}^{\otimes n})^2 -2 f_{\mathcal I}^{\otimes n} \varphi_{\mathcal I}^{\otimes n} + (\varphi_{\mathcal I}^{\otimes n})^2}{h_{\mathcal I}^{\otimes n}}\\
&= \bigg( \int\frac{(f_{\mathcal I} - \varphi_{\mathcal I})^2 + 2\varphi_{\mathcal I} (f_{\mathcal I}-\varphi_{\mathcal I})+ \varphi_{\mathcal I}^2}{h_{\mathcal I}}\bigg)^{ n} -2\bigg(\int \frac{ (f_{\mathcal I}-\varphi_{\mathcal I}) \varphi_{\mathcal I} + \varphi_{\mathcal I}^2}{h_{\mathcal I}}\bigg)^n + \bigg(\int \frac{\varphi_{\mathcal I}^2}{h_{\mathcal I}}\bigg)^n.
\end{align*}
For $K = c_{int}^2\log(n)$, define 
$$h^{-1}_{\mathcal I}(x) = \sqrt{2\pi}\mathbf 1_{\{\mathcal I\}} \sum_{k \leq K} \frac{x^{2k}}{2^{k}k!},$$
and consider the quantities $$A^{2}:=\int\frac{\varphi_{\mathcal I}^2}{h_{\mathcal I}},\quad D^2 :=\int  \frac{ (f_{\mathcal I}-\varphi_{\mathcal I})^2}{h_{\mathcal I}}\quad  \mbox{and}\quad E := \int \frac{\varphi_{\mathcal I}}{h_{\mathcal I}} (f_{\mathcal I} - \varphi_{\mathcal I}).$$ It holds:
\begin{align}
\frac{2\big\|(f_{\mathcal I}^{\otimes n} -\varphi_{\mathcal I}^{\otimes n})\|_{TV}^{2}}{(\int h_{\mathcal I})^n}
&\leq  \Big[D^2 + 2E+A^2\Big]^{ n} -2\Big[ E+A^2\Big]^n + A^{2n}\nonumber\\
&\leq \sum_{2\leq k \leq n}{n \choose k}\Big[ D^2 + 2E\Big]^{k} A^{2(n-k)} +  nD^2 A^{2(n-1)} -2\sum_{2\leq k \leq n}{n \choose k} E^k A^{2(n-k)}\nonumber\\
&\leq \sum_{2\leq k \leq n}{n \choose k} 2^kD^{2k}  A^{2(n-k)}  + \sum_{2\leq k \leq n}{n \choose k} 2^k[ 2|E|]^{k}A^{2(n-k)} + nD^2 A^{2(n-1)}\nonumber\\ 
&\hspace{1cm}+2\sum_{2\leq k \leq n}{n \choose k} |E|^k A^{2(n-k)}.\label{eq:TVDEA}
\end{align}
We bound this last term by means of the following Lemma (the proof is postponed in Appendix \ref{prooflemma}).

\begin{lemma}\label{lem:boundint}
Assume \eqref{ass:0} and suppose that $c_{int} \geq 2c_{sup}\vee 1$. Then, there exists a universal constant $c_h>0$ such that
\begin{equation}\label{lemmabi}0 \leq A^2 \leq 1+c_{h}/n^2~~\text{and}~~\int h_{\mathcal{I}} \leq 1+c_{h}/n^2~~~\text{and}~~~|E|  \leq c_p/n + n^{-c_{sup}^2/2}  + 2c_h/n^2,
\end{equation} 
where the constant $c_{p}$ is defined in \eqref{ass:0}.
\end{lemma}

Using Lemma~\ref{lem:boundint} and the fact that ${n \choose k}\leq n^{k}$, from \eqref{eq:TVDEA} we derive
\begin{align*}
\frac{2\big\|(f_{\mathcal I}^{\otimes n} -\varphi_{\mathcal I}^{\otimes n})\big\|_{TV}^{2}}{(\int h_{\mathcal I})^n}&\leq \exp(c_h)\Bigg[\sum_{2\leq k \leq n}  2^k(nD^2)^{k}   + \sum_{2\leq k \leq n} 2^k[ 2n |E|]^{k} +  nD^2
+2\sum_{2\leq k \leq n}(n|E|)^k\Bigg]\\
&\leq nD^{2}\exp(c_h)\Bigg[\sum_{2\leq k \leq n}  2^k(nD^2)^{k-1}+   1\Bigg] +3\exp(c_h)\sum_{2\leq k \leq n} 2^k[ 2n |E|]^{k}.
\end{align*}Moreover, thanks to Lemma \ref{lem:boundint}, if $c_{sup}\geq 2$ and $c_{p}<\frac18$ it holds
\begin{align*}
\sum_{2\leq k \leq n} 2^k[ 2n |E|]^{k}&\leq  \sum_{2\leq k\leq n}4^{k}\big(c_{p}+\frac{2c_{h}+1}{n}\big)^{k}\leq\Big(4\big(c_{p}+\frac{2c_{h}+1}{n}\big)\Big)^{2}\frac{1}{1-4\big(c_{p}+\frac{2c_{h}+1}{n}\big)}\leq C c_{p}^{2},\end{align*} where $C$ is a universal constant.

To complete the proof, we are only left to control the order of $D^{2}$. Indeed, applying Lemma~\ref{lem:boundint} to bound $\int h_{\mathcal{I}}$ we derive, for $c_{sup}\geq 2$ and $c_{p}<\frac18$,
\begin{align}\label{eq:bonneTV}
2\big\|(f_{\mathcal I}^{\otimes n} -\varphi_{\mathcal I}^{\otimes n})\big\|_{TV}^{2} &\leq  nD^{2}\exp(c_h)\Bigg[\sum_{2\leq k \leq n}  2^k(nD^2)^{k-1}+   1\Bigg]+C\exp(c_{h}) c_{p}^{2}.
\end{align}
To control the order of $nD^{2}$ in \eqref{eq:bonneTV}, introduce
 $G(x) = f(x) - \varphi(x)$ and notice that
\begin{align*}
D^2
&= \int  \frac{ (f_{\mathcal I}-\varphi_{\mathcal I})^2}{h_{\mathcal I}}  = \int \mathbf 1_{\{\mathcal I\}} \frac{G^2}{h}.
\end{align*}
Denote by $P_k(x) = x^k$, the Plancherel formula leads to
\begin{align}\label{eq:Dcontrol}
D^{2} &= \sqrt{2\pi}\int \mathbf 1_{\{\mathcal I\}} G^2(x) \sum_{k\leq K} \frac{x^{2k}}{2^kk!} dx
\leq \sqrt{2\pi} \sum_{k \leq K} \frac{1}{2^k k!} \|P_k G\|_2^2
= \frac{1}{\sqrt{2\pi}}\sum_{k \leq K} \frac{1}{2^k k!} \|\widehat{G}^{(k)}\|_2^2,
\end{align}
using that $\widehat{P_k G} = i^k \widehat{G}^{(k)} /\sqrt{2\pi}$.
Moreover, observe that the Fourier transform of $G$ can be written as 
\begin{align*}
\widehat{G}(t) &= \exp\Big(-\frac{t^2}{2} -i \frac{\mu_3(\eps)t^{3}}{6\sqrt{\Delta}s^3} + \sum_{m\geq 4} \Delta \mu_m(\eps) \frac{(ti)^m}{(\Delta s^2)^{m/2}m!} \Big) - \exp(-\lambda_{0,\eps} n\Delta) \mathbf 1_{\{\Sigma = 0\}} - \exp(-t^2/2)\\
&=\exp(-t^2/2) \Big[\exp\Big(- \frac{it^3\mu_3(\eps)}{6\sqrt{\Delta}s^3} + \sum_{m\geq 4} \Delta \mu_m(\eps) \frac{(ti)^m}{(\Delta s^2)^{m/2}m!} \Big) - 1\Big] - \exp(-\lambda_{0,\eps} n\Delta) \mathbf 1_{\{\Sigma = 0\}}.
\end{align*}
 Assumption \eqref{ass:epsD}, i.e. $\eps \leq \tilde c_{n} s\sqrt{\Delta}$, implies that $|\mu_m(\eps)| \leq (\tilde c_{n})^{m-4}\mu_4(\eps)s^{m-4} \Delta^{m/2-2}$ for any $m> 4$. Therefore,
\begin{align*}
\widehat{G}(t)
&=\exp(-t^2/2) \Big[\exp\Big(-it^3 \frac{\mu_3(\eps)}{6\sqrt{\Delta}s^3} + \frac{\mu_4(\eps)}{\Delta s^4}\sum_{m\geq 4}a_m \frac{(ti)^m}{m!} \Big) - 1\Big]- \exp(-\lambda_{0,\eps} n\Delta) \mathbf 1_{\{\Sigma = 0\}},
\end{align*}
where $a_m = \frac{\Delta \mu_m(\eps)}{(\Delta s^2)^{m/2}} \frac{\Delta s^4}{\mu_4(\eps)}$ is such that $a_m \leq (\tilde c_{n})^{m-4}$.

Set $M:=\tilde c_{n}^{-4}\Big[\frac{|\mu_3(\eps)|}{\sqrt{\Delta}s^3} + \frac{\mu_4(\eps)}{\Delta s^4}\Big]$ and observe that, if $\|\widehat{G}^{(k)}\|_2^2$ is bounded by a quantity much smaller in $k$ than $ 2^{k}k! M^{2}$ for any $k \leq K=c^{2}_{int}\log(n)$, then, thanks to \eqref{eq:Dcontrol}, $D^2$ would be bounded by $M^2$.
For illustration, consider first the term $k=0$. It holds
\begin{align}\label{eq:Ghat}
|\widehat{G}(t)|
&\leq \exp(-t^2/2) \Big[\exp\Big(M e^{|t|\tilde c_{n}} \Big) - 1\Big] +  \exp(-\lambda_{0,\eps} n\Delta) \mathbf 1_{\{\Sigma = 0\}}.
\end{align}
Then, for any $c >0$
\begin{align*}
\int_{-c\log(n)}^{c \log(n)} |\widehat{G}(t)|^2 dt
&\leq 4\int_0^{c\log(n)} \exp(-t^2) \Big[\exp\Big( M  e^{|t|\tilde c_{n}} \Big) - 1\Big]^2 dt + 4c\log(n)\exp(-2\lambda_{0,\eps} n\Delta) \mathbf 1_{\{\Sigma = 0\}}.
\end{align*}
Assumption \eqref{ass:M} ensures that for a small enough universal constant $c_m\geq0$ we have $M  \leq c_m/\sqrt{n}$. In this case, we use a Taylor expansion on $[0,c\log(n)]$ and get if $c\tilde c_{n} \leq 1/2$, and since $\lambda_{0,\eps} \geq 24 \log(n)/\Delta$
\begin{align*}
\int_{-c\log(n)}^{c \log(n)} |\widehat{G}(t)|^2 dt
&\leq C'\int_0^{c\log(n)} \exp(-t^2) \Big[ M e^{t\tilde c_{n}}\Big]^2 dt +4c/n^2
\leq C'' M^2+2/n^2,
\end{align*}
where $C', C''$ are two universal constants. This together with \eqref{ass:psiK} permits to bound $\|\hat G\|_{2}^{2}$ by $M^{2}$. The $k$th derivatives of $\hat G$ are treated similarly, see Lemma \ref{cor:derk} below, though the procedure is more cumbersome.\\

\begin{lemma}\label{cor:derk}
Suppose \eqref{ass:epsD} with $\tilde c_{n}\leq 1$, $\tilde c_{n} c\leq 1/4$ and \eqref{ass:M} with  $c_{m}\leq 1/2$. There exists a constant $C_{c_{int}}$ that depends on $c_{int}$ only, such that we have for any $t\in\mathcal{I}$, 
\begin{align}\label{cor:lem}|\widehat{G}^{(k)}(t)|^2 \leq C_{c_{int}} k^2 M^2&  \sup_{d\leq k-2}\Big( 2^{-8(k-d)}{k \choose d}^2 (k-d)^{k-d}   |\phi^{(d)}(t)|^2 \Big)\\ &\lor  k^4  (\tilde c_{n} M)^{2} e^{2\tilde c_{n}|t|}   |H_{k-1}(t)\phi(t)|^2 \lor |H_k(t)\widehat G(t)|^2,\nonumber\end{align}
where $H_k$ is the Hermite polynomial of degree $k$ and $\phi(t)=e^{-t^2/2}.$ Also, there exist two constants $\bar C,\bar c>0$ such that
\begin{align}\label{lem:calc}
&\int_{-c\log(n)}^{c \log(n)}   \sup_{d\leq k-2}\Big( 2^{-8(k-d)}{k \choose d}^2 (k-d)^{k-d}   |\phi^{(d)}(t)|^2 \Big)  dt
\leq 2  k!  \Big( \bar C2^{k(1-\bar c/16)} \lor  1 \Big), \quad \forall k \leq K.
\end{align}
Finally,
\begin{equation}\label{lem:intHk}
\int_{-c\log(n)}^{c \log(n)}  \exp(-t^2) e^{2|t|\tilde c_{n}} |H_k(t)| ^2 dt  \leq \frac{4 e^{\tilde c_{n}^{2}}}{\sqrt{2\pi}} k (k!) (1+  \tilde c_{n}^{2})^{k},\quad \forall k \leq K.
\end{equation}

\end{lemma}

To complete the proof, we bound the terms appearing in \eqref{cor:lem} by means of \eqref{lem:calc} and \eqref{lem:intHk}. This leads to a bound on $D^{2}$ using \eqref{eq:Dcontrol}, and so to a bound on the total variation thanks to \eqref{eq:bonneTV} and \eqref{eq:TVI}. 
To that aim, we begin by observing that, for $|t|\leq c\log(n)$,
$$|H_k(t)\widehat G(t)| \leq { C|H_{k}(t)|\exp(-t^2/2)M e^{\tilde c_{n}|t|}},$$
thanks to  \eqref{eq:Ghat}. Equation \eqref{lem:intHk} implies that
\begin{align}\label{eq:term2}
\int_{|t|\leq c\log n} \Big(k^4 (\tilde c_{n}M)^2 e^{2\tilde c_{n}|t|}   |H_{k-1}(t)|^2 \exp(-t^2)   \Big)\lor |H_k(t)\widehat G(t)|^2dt \leq C  M^{2} k^{5} (k!) (1+  \tilde c_{n}^{2})^{k},
\end{align} where $C$ is an absolute constant.
Combining \eqref{cor:lem}, \eqref{lem:calc} and \eqref{eq:term2} we derive:
\begin{align*}
&\int_{-c\log(n)}^{c \log(n)} |\widehat G^{(k)}(t)|^2dt \leq  {C'_{c_{int}} k^{5} } k! M^2 2^{k-\bar c'},\quad \forall k \leq K,
\end{align*}
where $\bar c>0$ is a universal constant {strictly positive} and $C'_{c_{int}}>0$ depends only on $C_{int}$. Therefore, from \eqref{eq:Dcontrol} joined with Assumption \eqref{ass:psiK}, we deduce that
$$D^{2}\leq M^{2} \frac{C'_{c_{int}}}{\sqrt{2\pi}}\sum_{k \leq K} \frac{k^{5}}{2^{\bar c' k} } \leq C''_{c_{int}} M^{2}.$$
In particular, recalling the definitions of $M$ and $s$ and using \eqref{eq:bonneTV}, \eqref{ass:M} and Lemma \ref{lem:ass0} (which ensures that under \eqref{ass:epsD}, Assumption \eqref{ass:0} holds with $c_{p}=1/n^{2}$) we finally obtain that
 \begin{align*}\big\|f^{\otimes n} -\varphi^{\otimes n}\big\|_{TV}&\leq \sqrt{nD^{2}\exp(c_h)\Bigg[\sum_{2\leq k \leq n}  2^k(nD^2)^{k-1}+   1\Bigg]+C\exp(c_{h}) c_{p}^{2}}+ c_p+ \frac{3+2c}{n}\\
&\leq C\sqrt{n}\Bigg(\frac{|\mu_3(\eps)|}{\sqrt{\Delta(\Sigma^{2}+\sigma^{2}(\eps))^3}} + \frac{\mu_4(\eps)}{\Delta (\Sigma^{2}+\sigma^{2}(\eps))^2}\Bigg)+\frac{4(c+1)}{n},\end{align*} where $C$ 
depends on $\tilde c_{n}$, $c_{m}$, $c_{sup}$ and $c_{int}$.
The proof of Theorem \ref{teo:TVgaus} is now complete.

\subsection{Proof of Theorem \ref{cor:TVgaus}}
Consider first the case $\tilde u^* = u^+$. In this case the theorem naturally holds for any $C\geq 2$ as $1 - e^{-\lambda_{\tilde u^*,\eps}n\Delta} = 1 - 1/n$. Assume from now on that $\tilde u^* > u^+$. Note that this implies that $\tilde u^* \leq \tilde c\sqrt{\Delta (\Sigma^2 + \sigma^2(\tilde u^*))}/\sqrt{\log(n \lor e)}$. In this case note that the assumptions of Theorem~\ref{teo:TVgaus} are satisfied for $\tilde u^*$.

Fix $0<u\leq \eps$ and write $M(\eps)=M(u)+M(u,\eps)$ where $M(u,\eps)$ is a compound Poisson process with intensity $\lambda_{u,\eps}$ independent of $M(u)$, see \eqref{cppapprox}. Decomposing on the values of the Poisson process $N(u,\eps)$ at time $n\Delta$, we have 
\begin{align}
\big\|(&\No(b(\eps)\Delta,\Delta\Sigma^2)*M_\Delta(\eps))^{\otimes n}-\No(b(\eps)\Delta,\Delta (\Sigma^{2}+\sigma^2(\tilde u^*)))^{\otimes n}\big\|_{TV} \nonumber\\
&\leq e^{-\lambda_{\tilde u^*,\eps}n\Delta}\big\|(\No(b(\eps)\Delta,\Delta\Sigma^2)*M_\Delta(\tilde u^*))^{\otimes n}-\No(b(\eps)\Delta,\Delta (\Sigma^{2}+\sigma^2(\tilde u^*)))^{\otimes n}\big\|_{TV} + 1 - e^{-\lambda_{\tilde u^*,\eps}n\Delta}\nonumber\\
&\leq C e^{-\lambda_{\tilde u^*,\eps}n\Delta} \Bigg(\sqrt{\frac{n\mu_4^2(\tilde u^*)}{\Delta^{2} (\Sigma^2+\sigma^2(\tilde u^*))^4}+\frac{n\mu_3^2(\tilde u^*)}{\Delta(\Sigma^2+\sigma^2(\tilde u^*))^{3}}} +\frac1n\Bigg) + 1 - e^{-\lambda_{\tilde u^*,\eps}n\Delta}.\nonumber
\end{align} 
Indeed, to obtain the last inequality we adapt the result of Theorem \ref{teo:TVgaus}, as from its proof the result holds regardless the drift and variance of both terms as long as they are equal. Finally, the result follows using 
\begin{align*}\min_{B \in \mathbb R, S^2\geq 0}\big\|(\No(b(\eps)\Delta,\Delta\Sigma^2)*&M_\Delta(\eps))^{\otimes n}-\No(B\Delta,\Delta S^2)^{\otimes n}\big\|_{TV}\leq \\&\big\|(\No(b(\eps)\Delta,\Delta\Sigma^2)*M_\Delta(\eps))^{\otimes n}-\No(b(\eps)\Delta,\Delta (\Sigma^{2}+\sigma^2(\tilde u^*)))^{\otimes n}\big\|_{TV}.\end{align*}

\subsection{Proof of Theorem \ref{thm:LB}}

\subsubsection{Preliminary: Four statistical tests}\label{4tests}

Let $X(\eps)\sim(b,\Sigma^2,\nu_\eps)$, $\eps>0$. In particular the increments $X_{i\Delta}(\eps)-X_{(i-1)\Delta}(\eps)$ are i.i.d. realizations of 
$$X_{\Delta}(\eps)=\Delta b(\eps) +\Sigma W_{\Delta}+ M_{\Delta}(\eps), \quad \textnormal{where } W_{\Delta}\sim\mathcal{N}(0,\Delta).$$
For any $n\in\N$, set $\tilde n = \lfloor n/2\rfloor$ and define
\begin{align*}
Z_{i}(\eps) &:= |(X_{2i\Delta}(\eps)-X_{(2i-1)\Delta}(\eps))-(X_{(2i-1)\Delta}(\eps)-X_{(2i-2)\Delta}(\eps))|,\quad i = 1,...,\tilde n= \lfloor n/2\rfloor,\\
S_n(\eps) &= \frac{1}{\tilde n}\sum_{i=1}^{\tilde n-2\log(n)}Z_{(i)}(\eps),~~~\text{and}~~~Z_{(\tilde n)}(\eps)=\max\{Z_{i}(\eps),\ 1\leq i\leq \tilde n\},
\end{align*} 
where for any sequence $a_{.}$, the sequence $a_{(.)}$ is a reordering of $a$ by increasing order.

For any $0<u \leq \eps$, we write $X(\eps)$ as
$X (\eps)= \bar X(u,\eps) + M(u,\eps),$
where $$\bar X(u,\eps)_t = b(\eps) t + \Sigma W_t + M_t(u)$$ is a Lévy process with jumps of size smaller (or equal) than $u$ and $M(u,\eps) = M_t(\eps) - M_t(u)$ is a pure jumps Lévy process with jumps of size between $u$ and $\eps$. In accordance with the notation introduced in \eqref{cppapprox}, we write $N(u,\eps)$ for the number of jumps larger than $u$ and smaller than $\eps$, that is, for any $t>0$, $N_t(u,\eps)$ is a Poisson random variable with mean $t\lambda_{u,\eps}$.

Furthermore, in order to present the test needed to prove Theorem \ref{thm:LB}, we introduce the following notations:
\begin{align*}
\bar X_{\Delta,n}(\eps)&:=\frac{1}{n- \lfloor n/2\rfloor}\sum_{i=\lfloor n/2\rfloor +1}^{n}(X_{i\Delta}(\eps)-X_{(i-1)\Delta}(\eps)),\\
\bar Y_{n,3}(\eps)&:=\frac{1}{\lfloor n/2\rfloor}\sum_{i=1}^{\lfloor n/2\rfloor}\big((X_{i\Delta}(\eps)-X_{(i-1)\Delta}(\eps))-\bar X_{\Delta,n}(\eps)\big)^{3}\Big),\\
\bar Y_{n,2}(\eps)&:=\frac{1}{\lfloor n/4\rfloor}\sum_{i=1}^{\lfloor n/4\rfloor}Z_i^2(\eps), \quad
\bar Y_{n,2}'(\eps):=\frac{1}{\lfloor n/4\rfloor}\sum_{i=\lfloor n/4\rfloor+1}^{\lfloor n/2\rfloor}Z_i^2(\eps),\\
\bar Y_{n,4}(\eps)&:=\frac{1}{\lfloor n/2\rfloor}\sum_{i=1}^{\lfloor n/2\rfloor}Z_i^4(\eps),\quad \bar Y_{n,6}(\eps):=\frac{1}{\lfloor n/2\rfloor}\sum_{i=1}^{\lfloor n/2\rfloor}Z_i^6(\eps),\\
T^{(3)}_{n}(\eps)&:=\frac{1}{1- {(n-\lfloor n/2\rfloor)^{-2}}}\overline Y_{n,3}(\eps), \quad  T^{(4)}_n(\eps):=\frac{1}{4}\Big(\bar Y_{n,4}(\eps)-3\bar Y_{n,2}(\eps)\bar Y_{n,2}'(\eps)\Big).
\end{align*}
By definition, $\bar X_{\Delta,n}(\eps)$ is the empirical version of $\E[X_{\Delta}(\eps)]$ computed on the second half of the sample only and $\bar Y_{n,3}(\eps)$ (resp. $\bar Y_{n,6}(\eps)$) is an estimator of $\E[(X_{\Delta}(\eps)-\Delta b(\eps))^{3}]$ (resp. of $8\E[(X_{\Delta}(\eps)-\Delta b(\eps))^{6}]$) computed on the first half of the sample. 
Moreover, since
$\E[(X_{\Delta}(\eps)-b(\eps)\Delta)^{3}]=\Delta\mu_{3}(\eps)$, using Corollary \ref{cor:coco2} joined with the independence of $X_\Delta(\eps)$ and $\bar X_{\Delta,n}(\eps)$, we have that $T^{(3)}_{n}(\eps)$ is an unbiased estimator of $\Delta\mu_{3}(\eps)$. Instead $\bar Y_{n,4}(\eps)$ is an estimator of $4\E[(X_{\Delta}(\eps)-\Delta b(\eps))^{4}]$ while $\bar Y_{n,2}(\eps)$ and $\bar Y_{n,2}'(\eps)$ are two independent estimators of $2\E[(X_{\Delta}(\eps)-\Delta b(\eps))^{2}]$ (see also Corollary \ref{cor:coco2}).
Using that $\E[Z_1^{4}(\eps)]-3(\E[Z_1^{2}(\eps)])^{2}=4\Delta\mu_{4}(\eps)$,
it is easy to prove that $T^{(4)}_n(\eps)$ is an unbiased estimator of $\Delta\mu_{4}(\eps)$ (see, e.g. \cite{dodge1999complications}).

Let $C>0$ be the absolute constant introduced in Lemma~\ref{lem:mo6} below and consider the following events:
\begin{itemize}
\item If $\nu_\eps = 0$, set
\begin{align}\label{eq:defxi2}
\xi_n &:=  \big\{\forall i, Z_i(\eps) \leq 4\sqrt{\Delta\Sigma^2\log(n)}\big\},\\
\xi'_n &:= \bigg\{\frac{\sqrt{\Delta\Sigma^2}}{\sqrt\pi} \leq S_n(\eps)\bigg\}, \label{eq:defxip}\\ 
\xi''_n &:= \{C(\sqrt{\Delta\Sigma^2})^6 \geq \bar Y_{n,6}(\eps)\}. \label{eq:defxipe}
\end{align}

\item If $\nu_\eps \neq 0$, set
\begin{align}
\xi_n &:= \{1\leq N_{n\Delta}(u^+,\eps) \leq 2\log(n)\} \cap \{\forall  i \leq n, N_{i\Delta}(u^+,\eps) - N_{(i-2)\Delta}(u^+,\eps) \leq 1\}\nonumber \\ &\quad \cap \{\forall i~s.t.~N_{i\Delta}(u^+,\eps) - N_{(i-2)\Delta}(u^+,\eps) \neq 0\},\nonumber\\
|\bar X_{i\Delta}&(u^*) - \bar X_{(i-1)\Delta}(u^*) - (\bar X_{(i-1)\Delta}(u^*) - \bar X_{(i-2)\Delta}(u^*))| \leq 2\sqrt{(\Sigma^2+\sigma^2(u^*))\Delta} \log(n)\},\label{eq:defxi}\\ 
\xi'_n &:= \{S_n(\eps) \leq 2\sqrt{2\Delta(\Sigma^2+\sigma^2(u^+))}\},\label{eq:defxipup}\\ \nonumber
\xi''_n &:= \Big\{\bar Y_{n,6}(\eps) \geq \frac{\Delta \mu_6(u^*) + (\Delta(\Sigma^2+\sigma^{2}(u^*)))^3}{2}\Big\}.
\end{align}
\end{itemize}

\begin{lemma}\label{lem:defxi}
There exists a universal sequence $\alpha_n \rightarrow 0$ such that 
$\mathbb P(\xi_n)\geq 1-\alpha_n.$
\end{lemma}

\begin{lemma}\label{lem:med}
There exists a universal sequence $\alpha_n \rightarrow 0$ such that 
$\mathbb P(\xi'_n)\geq 1-\alpha_n.$

\end{lemma}

\begin{lemma}\label{lem:mo6}
There exist a universal sequence $\alpha_n \rightarrow 0$ and a universal constant $C>0$ such that the following holds.

Whenever $\nu_\eps = 0$, with probability larger than $1-\alpha_n$ we have
$$C(\sqrt{\Delta\Sigma^2})^6 \geq \bar Y_{n,6}(\eps).$$

In any case, with probability larger than $1-\alpha_n$
and conditional on $ N_{n\Delta}(u^*,\eps) = 0$, it holds
 $$\bar Y_{n,6}(\eps) \geq \frac{\Delta \mu_6(u^*) + (\Delta(\Sigma^2+\sigma^{2}(u^*)))^3}{2}.$$
\end{lemma}

Observe that Lemmas~\ref{lem:defxi}, \ref{lem:med} and \ref{lem:mo6} joined with Equation~\eqref{eq:probaxixi}, imply the existence of two absolute sequences $\alpha_n\rightarrow 0$ and $\beta_n\rightarrow 0$ such that
\begin{align}\label{eq:probaxixi}
\mathbb P(\xi_n \cap \xi'_n) \geq 1-\alpha_n,\\
\mathbb P(\xi_n \cap \xi'_n \cap \xi''_n) \geq 1-\beta_n.\nonumber
\end{align}
We are now ready to introduce the four tests we use to establish Theorem \ref{thm:LB}:
\begin{align*}
\Phi_{n}^{(\max)}&=\mathbf{1}\big\{Z_{(\tilde n)(\eps)}\geq  \log(n)^{3/2}S_n(\eps)\big\},\hspace{2cm}
\Phi_{n,c}^{(6)}=\mathbf{1}\big\{\bar Y_{n, 6}(\eps) \geq  cS_n^6(\eps)\big\},\\
\Phi_{n,c,\alpha}^{(3)}&=\mathbf{1}\Big\{|T^{(3)}_n(\eps)|\geq \frac{c}{\sqrt{\alpha}} \sqrt{\frac{ S_n^{6}(\eps)}{n}} \Big\},\hspace{1cm}
\Phi_{n,c,\alpha}^{(4)}=\mathbf{1}\Big\{|T^{(4)}_n(\eps)|\geq \frac{c}{\sqrt{\alpha}} S_n^{4}(\eps)\sqrt{\frac{1}{n}}\Big\}.
\end{align*}
Their properties are investigated in Propositions \ref{prop:testmax}, \ref{prop:test6}, \ref{prop:LBNsym} and \ref{prop:LBsym} below.
Finally recall that for any $\eps>0$, the null hypothesis that we consider is $H_0:$ $\nu_\eps=0$.
\begin{proposition}\label{prop:testmax}

Under $H_0$, for any $n> e^{4\sqrt \pi}$, it holds that 
$\xi_n \cap \xi'_n \subset \{\Phi_{n}^{(\max)} = 0\}.$
Moreover,  for any $n>e^2$, it holds
$\xi_n \cap \xi'_n \cap \{N_{n\Delta}(u^*,\eps) \geq 1\} \subset  \{N_{n\Delta}(u^*,\eps) \geq 1\} \cap \{\Phi_{n}^{(\max)} = 1\}.$
\end{proposition}

\begin{proposition}\label{prop:test6}
There exist $c>0$ a universal constant and $C_c$ depending only on $c$ such that the following holds, for $n$ large enough.
Under $H_0$, it holds that $\xi'_n \cap \xi''_n \subset \{\Phi_{n,c}^{(6)} = 0\}$. Moreover if 
$\Delta\mu_6(u^*) \geq C_c \Delta^3(\Sigma^2+\sigma^2(u^*))^3,$
then $\xi'_n \cap \xi''_n \cap \{N_{n\Delta}(u^*,\eps) = 0\} \subset  \{N_{n\Delta}(u^*,\eps) = 0\} \cap \{\Phi_{n,c}^{(6)} = 1\}$.
\end{proposition}

\begin{proposition}\label{prop:LBNsym}
Let $\alpha >2\log(n)^{-1}$. Let $c>0$ and $c'>0$ be a large enough absolute constant and let $C_{c,c'} >0$ be a large enough absolute constant depending only on $c$ and $c'$. Then, the following holds.

Under $H_0$, $\Phi_{n,c,\alpha}^{(3)} = 0$ with probability larger than $1-\alpha- \mathbb P({\xi'_n}^{\mathsf c})$. 

Under the hypothesis
$H_{1, \rho_n^{(3)}}^{(3)}: \mu_{3}(u^*)>\rho_n^{(3)}$ and conditional to the event $N_{n\Delta}(u^*) = 0$, 
if 
\begin{equation}\label{mu6}
u^* > u^+,\quad  \Delta\mu_6(u^*) \leq c'\Delta^3(\Sigma^2+\sigma^2(u^*))^3,\quad 
\rho_n^{(3)}\geq C_{c,c'}\frac{\sqrt{\Delta(\Sigma^2+\sigma^2(u^*))^3}}{\sqrt{ n\alpha}},
\end{equation}
it holds that $\Phi_{n,c,\alpha}^{(3)} = 1$ with probability larger than $1-\alpha- \mathbb P({\xi'_n}^{\mathsf c})$.
\end{proposition}

\begin{proposition}\label{prop:LBsym}
Let $\alpha >2\log(n)^{-1}$. Let $c>0$ and $c'>0$ be a large enough absolute constant and let $C_{c,c'} >0$ be a large enough absolute constant depending only on $c$ and $c'$. Then, the following holds.

Under $H_0$, it holds that $\Phi_{n,c,\alpha}^{(4)} = 0$ with probability larger than $1-\alpha - \mathbb P({ \xi'_n}^{\mathsf c})$. 

Under the hypothesis
$H_{1, \rho_n^{(4)}}^{(4)}: \mu_{4}(u^*)>\rho_n^{(4)}$ and conditional to the event $N_{n\Delta}(u^*,\eps) = 0$, if 
$$u^* > u^+,\quad \rho_n^{(4)}\geq C_c \frac{\Delta(\Sigma^2+\sigma^2(u^*))^2}{\sqrt {n\alpha}},$$
it holds that $\Phi_{n,c,\alpha}^{(4)} = 1$ with probability larger than $1-\alpha- \mathbb P({\xi'_n}^{\mathsf c})$.
\end{proposition}

\subsubsection{Proof of Theorem \ref{thm:LB}}

Let $(\tilde b,\tilde \Sigma^2, \tilde \nu_\eps)$ be a Lévy triplet where $\tilde \nu_\eps$ is a Lévy measure with support in $[-\eps,\eps]$. Assume that we want to test
\begin{align*}
H_0 :  \nu_\eps = 0,\hspace{1cm}\mbox{against}\hspace{1cm}H_{1} :  ( b, \Sigma^2,  \nu_\eps) = (\tilde b,\tilde \Sigma^2, \tilde \nu_\eps).
\end{align*}
We write $\tilde \mu_{.}, \tilde \lambda_{.,.}$ and $\tilde u^*$ for all the quantities related to $(\tilde b,\tilde \Sigma^2, \tilde \nu_\eps)$.

We can choose $c^{(3)}, c^{(4)},c^{(6)} >0$ large enough universal constants and $C^{(3)}, C^{(4)}, C^{(6)} >0$ large enough depending only on $c^{(3)}, c^{(4)}, c^{(6)}$, and an absolute sequence $\alpha_n$ that converges to $0$ such that Propositions~\ref{prop:testmax},~{\ref{prop:test6}},~\ref{prop:LBNsym}~and~\ref{prop:LBsym} hold.
Set $$\alpha =\Big\{ \Big(\frac{C^{(3)}}{\tilde \mu_3(\tilde u^*)}\frac{\sqrt{\Delta(\Sigma^2+\sigma^2(\tilde u^*))^3}}{\sqrt n}\Big)^2\land \Big(\frac{C^{(4)}}{\tilde \mu_4(u^*)}\frac{\Delta(\Sigma^2+\sigma^2(\tilde u^*))^2}{\sqrt n}\Big)^2\Big\} \lor \alpha_n .$$ Write $i=3$ if  $\Big(\frac{C^{(3)}}{\tilde \mu_3(\tilde u^*)}\frac{\sqrt{\Delta(\Sigma^2+\sigma^2(\tilde u^*))^3}}{\sqrt n}\Big)^2\leq \Big(\frac{C^{(4)}}{\tilde \mu_4(u^*)}\frac{\Delta(\Sigma^2+\sigma^2(\tilde u^*))^2}{\sqrt n}\Big)^2$ and $i=4$ otherwise. In the remaining of the proof $\alpha_{n}$ denotes a vanishing sequence whose value may change from line to line.

\paragraph{Case 1 : $1 - \exp(-\tilde \lambda_{u^*,\eps}n\Delta) \geq 1-\alpha$.} In this case, consider the test
$\Phi_{n} = \Phi_{n}^{(\max)}.$
 If $X\sim(b,\Sigma^2,\nu_\eps)$ is in $H_0$ (\textit{i.e.} $\nu_\eps = 0$), an application of Proposition~\ref{prop:testmax} and Lemmas \ref{lem:defxi} and \ref{lem:med} yields
$
\mathbb P(\Phi_{n} = 0) \geq 1 - \alpha_n.
$
If, instead, $X$ is such that $( b, \Sigma^2,  \nu_\eps) = (\tilde b,\tilde \Sigma^2, \tilde \nu_\eps)$, by means of Proposition~\ref{prop:testmax} and Lemmas \ref{lem:defxi}, \ref{lem:med} we get 
\begin{align*}
\mathbb P(\Phi_{n} = 1) \geq \mathbb P(\{N_{n\Delta}(\tilde u^*,\eps) \neq 0\}\cap \xi_{n} \cap \xi_{n}') \geq 1 - \exp(-\tilde \lambda_{\tilde u^*,\eps}n\Delta) - \alpha_n.
\end{align*}
So by Lemma~\ref{lem:TvTest} it follows that the total variation between the observations of $n$ increments of $X$ at the sampling rate $\Delta$ and the closest Gaussian random variable is larger than
$1 - \exp(-\tilde \lambda_{\tilde u^*,\eps}n\Delta) - \alpha_n.$

\paragraph{Case 2 : $1 - \exp(-\tilde \lambda_{\tilde u^*,\eps}n\Delta) \leq 1-\alpha$.} In this case consider the test
$$\Phi_{n, c^{(i)},c^{(6)}, \alpha} = \Phi_{n}^{(\max)} \lor \Phi_{n,c^{(i)},\alpha}^{(i)}\lor \Phi_{n,c^{(6)}}^{(6)}.$$

If $X$ is in $H_0$ (\textit{i.e.} $\nu_\eps = 0$), by Propositions~\ref{prop:testmax},~{\ref{prop:test6}},~\ref{prop:LBNsym}~and~\ref{prop:LBsym} we have that 
\begin{align*}
\mathbb P(\Phi_{n, c^{(i)},c^{(6)},\alpha} = 0) \geq 1-\alpha - \alpha_n.
\end{align*}

If $X$ is such that $( b, \Sigma^2,  \nu_\eps) = (\tilde b,\tilde \Sigma^2, \tilde \nu_\eps)$, we distinguish two cases.

\begin{itemize}
\item If $\Delta\mu_6(u^*) \geq C^{(6)} (\Delta (\Sigma^2 + \sigma^2(u^*)))^3$: Propositions~\ref{prop:testmax},~\ref{prop:test6} yield
\begin{align*}
\mathbb P(\Phi_{n, c_\alpha^{(i)},c^{(6)},\alpha} = 1) \geq \mathbb P(\{N_{n\Delta}(\tilde u^*,\eps) \neq 0\}\cap \xi_{n} \cap \xi_{n}') + \mathbb P(\{N_{n\Delta}(\tilde u^*,\eps) = 0\})(1- \alpha_n).
\end{align*}
\item If $\Delta\mu_6(u^*) < C^{(6)} (\Delta (\Sigma^2 + \sigma^2(u^*)))^3$: Propositions~\ref{prop:testmax},~\ref{prop:LBNsym},~\ref{prop:LBsym} joined with $\{u^{*}>u^{+}\}$ yield 
\begin{align*}
\mathbb P(\Phi_{n, c_\alpha^{(i)},c^{(6)},\alpha} = 1) \geq \mathbb P(\{N_{n\Delta}(\tilde u^*,\eps) \neq 0\}\cap \xi_{n} \cap \xi_{n}') + \mathbb P(\{N_{n\Delta}(\tilde u^*,\eps) = 0\})(1-\alpha - \alpha_n).
\end{align*}
\end{itemize}
In both cases we conclude that,
\begin{align*}
\mathbb P(\Phi_{n, c^{(i)},c^{(6)},\alpha} =1) &\geq \mathbb P(\{N_{n\Delta}(\tilde u^*,\eps)\neq 0\}) + \mathbb P(\{N_{n\Delta}(\tilde u^*,\eps) = 0\})(1-\alpha) - \alpha_n\\
&\geq 1 - \alpha\exp(-\tilde \lambda_{\tilde u^*,\eps}n\Delta) - \alpha_n.
\end{align*}
By Lemma~\ref{lem:TvTest} we deduce that the total variation distance between the observations of $n$ increments of $X$ at the sampling rate $\Delta$ and the closest Gaussian random variable is larger than
$1 - 2\alpha - \alpha_n.$

\paragraph{Acknowledgements.} \small{The work of A. Carpentier is partially supported by the Deutsche Forschungsgemeinschaft (DFG) Emmy Noether grant MuSyAD (CA 1488/1-1), by the DFG - 314838170, GRK 2297 MathCoRe, by the DFG GRK 2433 DAEDALUS, by the DFG CRC 1294 'Data Assimilation', Project A03, and by the UFA-DFH through the French-German Doktorandenkolleg CDFA 01-18.\\
The work of E. Mariucci has been partially funded by the Federal Ministry for Education and Research through the Sponsorship provided by the Alexander von Humboldt Foundation, by
the Deutsche Forschungsgemeinschaft (DFG, German Research Foundation) – 314838170,
GRK 2297 MathCoRe, and by Deutsche Forschungsgemeinschaft (DFG) through grant CRC 1294 'Data Assimilation'.}

\bibliographystyle{abbrv}
\bibliography{BiblioCDM}

\appendix   
\section{Technical results}\label{appA}

\subsection{Proofs of the auxiliary Lemmas used in the proof of Theorem \ref{teo:TVgaus}}
\subsubsection{Proof of Lemma \ref{lem:ass0}}
Consider a compound Poisson approximation of the increment $\tilde M_\Delta(\eps):=(M_\Delta(\eps))/{\sqrt{\Delta(\sigma^2(\eps)+\Sigma^2)}}$.
Let $0<\eta<\eps$, and define
\begin{align*}
\tilde M_{\Delta}(\eta,\eps) &= \frac{\sum_{i=0}^{N_\Delta(\eta,\eps)} Y_i  - \Delta\int_{\eta\leq |x|\leq \eps} xd\nu}{\sqrt{\Delta s^2}}= \bar M_{\Delta}(\eta,\eps) +  \frac{ N_\Delta(\eta,\eps) \lambda^{-1}_{\eta,\eps} \int_{\eta\leq |x|\leq \eps} xd\nu  - \Delta\int_{\eta\leq |x|\leq \eps} xd\nu}{\sqrt{\Delta s^2}},
\end{align*}
where 
$\bar M_{\Delta}(\eta,\eps) = {\sum_{i=0}^{N_\Delta(\eta,\eps)} \big(Y_i - \lambda_{\eta,\eps}^{-1}\int_{\eta\leq |x|\leq \eps} xd\nu\big)}/{\sqrt{\Delta s^2}},$
and $N(\eta,\eps)$, $\lambda_{\eta,\eps}$ and the sequence $(Y_i)_{i\geq 0}$ are defined as in \eqref{cppapprox} and \eqref{cppY}. 
Note that for any $N$, $\mathbb E[\bar M_{\Delta}(\eta,\eps)|N_\Delta(\eta,\eps)=N] =0,$ and if $|N_\Delta(\eta,\eps)- \Delta \lambda_{\eta,\eps}| \leq \Delta \lambda_{\eta,\eps}/2$ we have $\lambda_{\eta,\eps}\mathbb V(Y_i) = \int_{\eta\leq |x|\leq \eps} x^2d\nu \leq \sigma^2(\eps) \leq s^2$ and 
$\mathbb V[\bar M_{\Delta}(\eta,\eps)|N_\Delta(\lambda,\eps)=N] \leq 2.$ Finally, the random variables $|Y_{i}|$ are bounded by $\eps$.
For any $N$ such that $|N - \Delta \lambda_{\eta,\eps}| \leq \Delta \lambda_{\eta,\eps}/2$, the Bernstein's inequality, conditional on $N_\Delta(\eta,\eps) =N$, leads to 
\begin{align*}\PP\big(|\bar M_{\Delta}(\eta,\eps)|>c_{sup}\sqrt{\log(n)}/2\big|N_\Delta(\lambda,\eps)=N\big)
&\leq 2 \exp\Bigg(-\frac18\frac{c_{sup}^{2}\log n}{2+\frac16\tilde c_{n} c_{sup}}\Bigg),
\end{align*}where we used \eqref{ass:epsD}. Therefore, for any $N$ such that $|N - \Delta \lambda_{\eta,\eps}| \leq \Delta \lambda_{\eta,\eps}/2$, it holds
\begin{align*}\PP\big(|\bar M_{\Delta}(\eta,\eps)|>c_{sup}\sqrt{\log(n)}/2\big|N_\Delta(\eta,\eps)=N\big)&\leq n^{-3},
\end{align*} if $c_{sup}\geq 10$. Now by assumption on $\lambda_{0,\eps}$, there exists $\bar \eta:=\bar \eta_\delta >0$ such that for any $\eta \leq \bar \eta$, we have $\Delta \lambda_{\eta,\eps} \geq 1$ 
and $\frac{24\log(n)}{\Delta} \leq \lambda_{\eta,\eps}$. 
Moreover, for $\eta \leq \bar \eta$, since $N_\Delta(\eta,\eps)$ is a Poisson random variable of parameter $\Delta \lambda_{\eta,\eps} \geq 1$, we have for any $0 \leq x \leq \sqrt{\Delta \lambda_{\eta,\eps}}$ 
$$\mathbb P(|N_\Delta(\eta,\eps) - \Delta \lambda_{\eta,\eps}| \geq  x\sqrt{\Delta \lambda_{\eta,\eps}}/2) \leq \exp(-x^2/8).$$ 
This implies that for $x:= \sqrt{24\log(n)} \leq \sqrt{\Delta \lambda_{\eta,\eps}}$, 
$$\mathbb P(|N_\Delta(\lambda,\eps) - \Delta \lambda_{\eta,\eps}| \geq \sqrt{6\Delta \lambda_{\eta,\eps} \log(n)}) \leq n^{-3}.$$ 
Removing the conditioning on $N_\Delta(\eta,\eps)$ (noting that $\sqrt{6\Delta \lambda_{\eta,\eps} \log(n)} \leq \Delta \lambda_{\eta,\eps}/2$) we get 
\[\PP\bigg(\Big|\bar M_{\Delta}(\eta,\eps) + \frac{N_\Delta(\eta,\eps)}{\sqrt{\Delta s^2}} \lambda_{\eta,\eps}^{-1} \int_{\eta\leq |x| \leq \eps} x d\nu(x) - \frac{\Delta}{\sqrt{\Delta s^2}}\int_{\eta\leq |x| \leq \eps} x d\nu(x) \Big|>c_{sup}\sqrt{\log(n)}\bigg)\leq 2 n^{-3},
\]using that $c_{\sup} \geq 10$ and that by the Cauchy-Schwarz inequality $|\int_{\eta\leq |x| \leq \eps} x d\nu(x)| \leq \sqrt{\lambda_{\eta,\eps}}  \sigma(\eps)$.
Taking the limit as $\eta\to 0$ we get 
\[\PP\bigg(|\tilde M_\Delta(\eps)|>c_{sup}\sqrt{\log(n)}\bigg)\leq 2 n^{-3}.\]
As $\tilde X_\Delta(\eps)= \frac{\No(0,\Sigma^2\Delta)}{\sqrt{\Delta(\sigma^2(\eps)+\Sigma^2)}}+\tilde M_\Delta(\eps)$, by Gaussian concentration we get
$\PP\big(|\tilde X_{\Delta}(\eps)|>c_{sup}\sqrt{\log(n)}\big)\leq 3 n^{-3}.$
This implies the result whenever $c_{sup} \geq 10$.

\subsubsection{Proof of Lemma \ref{lem:boundint}}\label{prooflemma}

In the following, we will use repeatedly that for $n \geq 1$, $n! \geq (n/e)^n.$ 
For any $x \in \mathcal I$, using that $\sqrt{K} \geq 2c_{sup} \sqrt{\log(n)}$ (since $c_{int}\geq 2c_{sup}$), we derive that
$$|\sqrt{2\pi}\exp(x^2/2)  - h^{-1}_{\mathcal I}(x)| = \sqrt{2\pi}\mathbf 1_{\{\mathcal I\}}\sum_{k=K+1}^{+\infty} \frac{x^{2k}}{2^{2k} k!} \leq \sqrt{2\pi}\mathbf 1_{\{\mathcal I\}}\sum_{k=K+1}^{+\infty} \frac{(\sqrt{K}/2)^{2k}e^k}{2^{2k} k^k} \leq \frac{1}{1 - \frac e4} \frac{1}{2^{K}}.$$
Therefore, for $c_{int} \geq 1$ and $x\in\mathcal{I}$, we have  
\begin{align}\label{eq:hphiminus1}|\varphi^{-1}(x)-h^{-1}(x)|=|\sqrt{2\pi}\exp(x^2/2)  - h^{-1}_{\mathcal I}(x)| \leq \frac{1}{1 - e/4} \frac{1}{n^2}.\end{align}
Equation \eqref{eq:hphiminus1} implies that
$$A=\int\frac{\varphi_{\mathcal I}^2}{h_{\mathcal I}} = \int \Big(\varphi_{\mathcal I} + \varphi_{\mathcal I}^2(h^{-1}_{\mathcal I} - \varphi_{\mathcal I}^{-1})\Big) \leq 1+\frac{1}{1 - e/4}\frac{1}{n^2},$$
and 
$$\bigg|\int h_{\mathcal I} - \int \varphi_{\mathcal I}\bigg| =\bigg|\int_{\mathcal I} h\varphi (\varphi^{-1} - h^{-1})\bigg| \leq  \int_{\mathcal I} h(x)\varphi(x) |\sqrt{2\pi}\exp(x^2/2)  - h^{-1}(x)| \leq \frac{1}{1 - e/4} \frac{1}{n^2},$$ 
since, by definition, $h_{\mathcal I} \leq 1/\sqrt{2\pi}$. This, together with $ \mathbb P_{\varphi}(\mathcal I^c) \leq n^{-c_{sup}^2/2}$, leads to the second inequality in \eqref{lemmabi}.
Finally, using \eqref{eq:hphiminus1}, we get
\begin{align*}
|E|=\bigg|\int \varphi_{\mathcal I}h^{-1}_{\mathcal I} (f_{\mathcal I} - \varphi_{\mathcal I}) - \int (f_{\mathcal I} - \varphi_{\mathcal I})\bigg| &=  \bigg|\int \varphi_{\mathcal I}(h^{-1}_{\mathcal I} - \varphi_{\mathcal I}^{-1}) (f_{\mathcal I} - \varphi_{\mathcal I}) \bigg|
\leq 2\frac{1}{1 - e/4}\frac{1}{n^2}.
\end{align*}
By means of \eqref{ass:0} and using that $\mathbb P_{\varphi}(\mathcal I^c) \leq n^{-c_{sup}^2/2}$, we derive
$$\bigg| \int(f_{\mathcal I} - \varphi_{\mathcal I})\bigg| \leq  |\mathbb P_f(\mathcal I^c) - \mathbb P_{\varphi}(\mathcal I^c)| \leq n^{-c_{sup}^2/2}+c_p/n,$$ hence the bound on $|E|$ is  established.

\subsubsection{Proof of Lemma \ref{cor:derk}}

Write
$\widehat{G} +\exp(-\lambda_{0,\eps} n\Delta) \mathbf 1_{\{\Sigma = 0\}} = \phi V,$
where $\phi(t) = \exp(-t^2/2)$ and $V = \exp(g)-1$, $g(t) = -it^3 \frac{\mu_3(\eps)}{6\sqrt{\Delta}s^3} + \frac{\mu_4(\eps)}{\Delta s^4}\sum_{m\geq 4}a_m \frac{(ti)^m}{m!}$ with $a_m= \frac{\Delta \mu_m(\eps)}{(\Delta s^2)^{m/2}} \frac{\Delta s^4}{\mu_4(\eps)}$. Recall that $M=\tilde c_{n}^{-4}\Big(\frac{|\mu_3(\eps)|}{\sqrt{\Delta}s^3} + \frac{\mu_4(\eps)}{\Delta s^4}\Big)$.
We start with two preliminary Lemmas.
\begin{lemma}\label{lem:rec}
Suppose that \eqref{ass:epsD} holds true with $\tilde c_{n}\leq 1$. Then, for any $m\geq 1$, we have
$$|V^{(m)}| = |(\exp(g)-1)^{(m)}| \leq 2^m  e^{M e^{|t|\tilde c_{n}}} \max_{1\leq u \leq m} u^{m-u} (\tilde c_{n} M)^u e^{\tilde c_{n}|t|u}.$$
\end{lemma}

\begin{proof}[Proof of Lemma \ref{lem:rec}]
First, note that for any $j \geq 1$, we have
$$g^{(j)}(t)=-i(t^{3})^{(j)}\frac{\mu_{3}(\eps)}{6\sqrt{\Delta }s^{3}}+\frac{\mu_{4}(\eps)}{\Delta s^{4}}\sum_{m\geq (j\land 4)}a_{m}\frac{i^{m}t^{m-j}}{(m-j)!},$$
and $$|g^{(j)}(t)| \leq \bigg(\frac{|\mu_{3}(\eps)|}{\sqrt{\Delta }s^{3}}+\frac{\mu_{4}(\eps)}{\Delta s^{4}}\bigg)\sum_{m\geq (j\land 3)}|a_{m}|\frac{|t|^{m-j}}{(m-j)!},$$
where $a_3 = 1 $. Using that for any $m\geq 4$ it holds that $a_m \leq \tilde c_{n}^{m-4}$ with $\tilde c_{n}\leq 1$, we derive
\begin{align}\label{eq:gjsup}|g^{(j)}(t)| \leq \bigg(\frac{|\mu_{3}(\eps)|}{\sqrt{\Delta }s^{3}}+\frac{\mu_{4}(\eps)}{\Delta s^{4}}\bigg)\sum_{m\geq (j\land 3)}\tilde c_{n}^{m-4}\frac{|t|^{m-j}}{(m-j)!}\leq \tilde c_{n} M e^{|t|\tilde c_{n}}.\end{align}
Let us write
$$R_m = \frac{(\exp(g))^{(m)}}{\exp(g)}$$
and note that $R_{m+1} = R_m^{(1)} + g^{(1)}R_m.$
For any $d \geq 0$
\begin{align}\label{eq:Rmd}
|R_{m+1}^{(d)}| &= |R_m^{(d+1)} + (g^{(1)}R_m)^{(d)}|\leq |R_m^{(d+1)}| + \sum_{j\leq d} C_d^j |g^{(d-j+1)}| |R_m^{(j)}|,
\end{align}
 by the Leibniz formula. For $m\geq 1$, let us consider the following induction assumption:
$$H(m) : \forall d \in \mathbb N, |R_m^{(d)}| \leq 2^m\max_{u\in \{1, \ldots, m\}} (\tilde c_{n} M)^u e^{\tilde c_{n} |t| u} u^{m-u} u^d.$$
$H(1)$ is true since $R_1=g^{(1)}$ and by \eqref{eq:gjsup} we obtain $|g^{(j)}| \leq  M\tilde c_{n} e^{\tilde c_{n} |t|}$ for $j\geq 1$. Let us now assume that $H(m)$ holds for some $m\geq 1$. By \eqref{eq:Rmd}, joined with $H(m)$ and \eqref{eq:gjsup}, we get
 \begin{align*}
|R_{m+1}^{(d)}| 
&\leq 2^m\max_{u\in \{1, \ldots, m\}} (\tilde c_{n} M)^u e^{\tilde c_{n} |t| u} u^{m-u} u^{d+1}+ \Big[  M\tilde c_{n} \exp(\tilde c_{n} |t|) \Big]\sum_{j\leq d}C_d^j \Big[2^m\max_{u\in \{1, \ldots, m\}} (\tilde c_{n} M)^u e^{\tilde c_{n} |t| u} u^{m-u} u^{j}\Big].
\end{align*} It follows that
 \begin{align*}
|R_{m+1}^{(d)}| 
&\leq 2^{m}\max_{u\in \{1, \ldots, m\}} (\tilde c_{n} M)^u e^{\tilde c_{n} |t| u} u^{m+1-u} u^{d}+ \sum_{j\leq d}C_d^j  \Big[2^m\max_{u\in \{1, \ldots, m\}} (\tilde c_{n} M)^{u+1} e^{\tilde c_{n} |t| (u+1)} u^{m-u} u^{j}\Big]\\
&= 2^{m}\max_{u\in \{1, \ldots, m\}}  (\tilde c_{n} M)^u e^{\tilde c_{n} |t| u} u^{m+1-u} u^{d}
+  \Big[2^m\max_{u\in \{1, \ldots, m\}} (\tilde c_{n} M)^{u+1} e^{\tilde c_{n} |t| (u+1)} u^{m-u} (1 +  u)^{d}\Big]\\
&\leq 2^{m}\max_{u\in \{1, \ldots, m\}}  (\tilde c_{n} M)^u e^{\tilde c_{n} |t| u} u^{m+1-u} u^{d}
+  \Big[2^m\max_{u\in \{1, \ldots, m\}} (\tilde c_{n} M)^{u+1} e^{\tilde c_{n} |t| (u+1)} (u+1)^{m-u+d} \Big],
\end{align*}
where we used the binomial formula  for the second equation. Finally,
 \begin{align*}
|R_{m+1}^{(d)}| 
&\leq 2^{m}\max_{u\in \{1, \ldots, m\}}  (\tilde c_{n} M)^u e^{\tilde c_{n} |t| u} u^{m+1-u} u^{d}+  \Big[2^m\max_{u\in \{2, \ldots, m+1\}} (\tilde c_{n} M)^{u} e^{\tilde c_{n} |t| u} u^{m+1-u+d} \Big]\\
&\leq 2^{m+1}\max_{u\in \{1, \ldots, m+1\}} (\tilde c_{n} M)^u e^{\tilde c_{n} |t| u} u^{m+1-u}.
\end{align*}
Therefore, $H({m+1})$ holds and the induction hypothesis is thus proven. In particular,
$$|R_m| \leq 2^m \max_{u\in \{1, \ldots, m\}} u^{m-u} (\tilde c_{n} M)^u e^{\tilde c_{n} |t| u},$$
and for any $m \geq 1$ we obtain
$$|V^{(m)}| = |(\exp(g)-1)^{(m)}| \leq |R_{m}||\exp(g)|\leq2^m  e^{M e^{|t|\tilde c_{n}}} \max_{1\leq u \leq m} u^{m-u} (\tilde c_{n} M)^u e^{\tilde c_{n}|t|u}.$$
\end{proof}

\begin{lemma}\label{lem:Gk}
Suppose that \eqref{ass:epsD} holds true with $\tilde c_{n}\leq 1$. Then, for any $k$$$|\widehat{G}^{(k)}(t)|^2 \leq \bigg(k^2 e^{{2}M e^{|t|\tilde c_{n}}}  \sup_{d\leq k-1} {k \choose d}^2 
 |\phi^{(d)}(t)|^2   2^{2(k-d)}\max_{1\leq u \leq k-d} u^{2(k-d-u)} (\tilde c_{n} M)^{2u} e^{{2}\tilde c_{n}|t|u}\bigg) \lor \Big(|\phi^{(k)}|^2 V^2\Big).$$
\end{lemma}

\begin{proof}
[Proof of Lemma \ref{lem:Gk}]
By the binomial formula we bound
$|\widehat{G}^{(k)}| \leq k\sup_{d\leq k} C_k^d |\phi^{(d)}| |V^{(k-d)}|.$
Finally, an application of Lemma~\ref{lem:rec} yields
$$|V^{(m)}| = |(\exp(g)-1)^{(m)}| \leq 2^m  e^{M e^{|t|\tilde c_{n}}} \max_{1\leq u \leq m} u^{m-u} (\tilde c_{n} M)^u e^{\tilde c_{n}|t|u}.$$ 
\end{proof}

\paragraph{Proof of Equation \eqref{cor:lem} in Lemma \ref{cor:derk}.} An application of Lemma~\ref{lem:Gk} yields
$$|\widehat{G}^{(k)}(t)|^2 \leq \Bigg(k^2 e^{{2}M e^{|t|\tilde c_{n}}}  \sup_{d\leq k-1}\Bigg( {k \choose d}^2  |\phi^{(d)}(t)|^2   2^{2(k-d)}\max_{1\leq u \leq k-d} u^{2(k-d-u)} (\tilde c_{n} M)^{2u} e^{{2}\tilde c_{n}|t|u} \Bigg)\Bigg) \lor \Big(|\phi^{(k)}|^2 V^2\Big).$$
The term $|\phi^{(k)}|^2 V^2$ leads to the last term in \eqref{cor:lem} since $H_m \phi = \phi^{(m)}$ for any $m$ and $|V|=|\phi^{-1}\hat G|$. The term corresponding to $d = k-1$ leads to the second term in \eqref{cor:lem}, using \eqref{ass:M} and the fact that for $|t| \leq c\log(n)$,  $e^{{2}Me^{\tilde c_{n}|t|}}\leq   e^{2c_{m}n^{\tilde c_{n} c-\frac{1}{2}}}<e$ whenever $\tilde c_{n} c \leq \frac{1}{2},$  $c_{m}\leq \frac12$.

Next, we control the remaining term using the decomposition $(\tilde c_{n}M)^{2u}=((\tilde c_{n}M)^{\frac{(u-1)}{2}})^{2}(\tilde c_{n}M)^{u-1}(\tilde c_{n}M)^{2}$. First notice that by means of \eqref{ass:M} together with $c_{m}\leq \frac12$, we deduce that for any integer $u \geq 2$ and $t$ such that $|t|\leq c\log(n)$, if  $\tilde c_{n} c\leq \frac14$ it holds$$M^{u-1} \exp({2}\tilde c_{n}|t|u) \leq 1\quad \mbox{and}\quad e^{{2}Me^{|t|\tilde c_{n}}}\leq 1.$$
Moreover, for any $u \geq 2$, \begin{align*}2^{2(k-d)} \max_{2\leq u \leq k-d} &u^{2(k-d-u)} (\tilde c_{n} M)^{2\frac{(u-1)}{2}} \\ &\leq \big[M(k-d)^{k-d}\big] \lor\Big[ 2^{2(k-d)} \max_{(k-d)^{1/4}\leq u \leq k-d} e^{2(k-d)\log(u)-{\frac14}\log(n)(u-1)}  M^{\frac{(u-1)}{2}} \Big].\end{align*} 
Since $k\leq K = c_{int} ^{2}\log(n)$, we know that $2(k-d)\log(u)-\log(n)(u-1)/4$ is negative whenever $ 4c_{int}^{2}\log(u)\leq u$. Thus, using \eqref{ass:M}, it follows that for $k\leq c_{int}^{2} \log(n)$  there exists a constant $C_{c_{int}}$, that depends only on $c_{int}$, such that
$$2^{2(k-d)} \max_{ 2\leq u \leq k-d} u^{2(k-d-u)} (\tilde c_{n} M)^{(u-1)} \leq C_{c_{int}} 2^{-8(k-d)} (k-d)^{k-d}.$$ This completes the proof of Equation \eqref{cor:lem}.

\paragraph{Proof of Equation \eqref{lem:calc} in Lemma \ref{cor:derk}.} By means of the Stirling approximation
\begin{align}\label{eq:stirling}\sqrt{\frac{\pi k}2} \Big(\frac ke\Big)^k \leq k! \leq 2\sqrt{\pi k}  \Big(\frac ke\Big)^k ,\quad \forall k\geq 1,\end{align}
we derive that if $Z\sim\mathcal{N}(0,\omega^{2})$ then,
\begin{align}\label{eq:mmtN}
\E[Z^{2m}]\leq 4(2\omega^{2})^{m}\frac{m^{m}}{e^{m}}\leq 4(2\omega^{2})^{m}m!,\quad \forall m\geq 1.
\end{align}
By Plancherel theorem and \eqref{eq:mmtN} applied with $\omega^{2}=1/2$, we deduce 
\begin{equation}\label{eq:plancherel}
\int  |\phi^{(m)}(t)|^2 dt = \| P_m \varphi\|_2^2 = \frac{1}{2\pi}\int x^{2m} \exp(-x^2) dx  \leq \frac{4}{\sqrt{2\pi}}\frac{m^m}{e^m},\quad \forall m \geq 1.
\end{equation}
Equation~\eqref{eq:plancherel} and \eqref{eq:stirling} imply
\begin{align*}
\int_{-c\log(n)}^{c \log(n)}  k^2& M^2  \sup_{d\leq k-2}\Big[ 2^{-8(k-d)}{k \choose d}^2 (k-d)^{k-d}   |\phi^{(d)}(t)|^2 \Big]  dt\\
& \leq k^2 M^2  \sup_{d\leq k-2}\Big[ 2^{-8(k-d)}{k \choose d}^2 (k-d)^{k-d}    \frac{2 d^d}{e^d} \Big]
\leq 2k^2 M^2  \sup_{d\leq k-2}\Big[ 2^{-8(k-d)} e^{k-d} k! {k \choose d} \Big]\\
&\leq {2 k^2} k! M^2  \sup_{d\leq k-2}\Big[2^{-4(k-d)} {k \choose d} \Big] = {2 k^2} k! M^2  \sup_{d\leq k-2}\Big[2^{-4d} {k \choose d} \Big],
\end{align*}where we used that $2^{-4}e\leq 1$. Moreover, we observe that by sub-Gaussian concentration of the binomial distribution there exist $\bar C,\bar c>0$ universal constant such that $\frac{{k \choose d}}{2^k} \leq \bar C e^{-\bar c(|d-k/2|/\sqrt{k})^2}$. Therefore,
\begin{align*}
\int_{-c\log(n)}^{c \log(n)}  k^2 M^2  \sup_{d\leq k-2}&\Big[ 2^{-8(k-d)}{k \choose d}^2 (k-d)^{k-d}  |\phi^{(d)}(t)|^2 \Big]   dt\\
&\leq {2 k^2} k! M^2 \Bigg[ \sup_{d\leq k/4}\Big[2^{-4d} {k \choose d} \Big] \lor  \sup_{k/4 +1\leq d\leq k-2}\Big[2^{-4d} {k \choose d} \Big] \Bigg]\\
&\leq {2 k^2} k! M^2 \Bigg[ \sup_{d\leq k/4}\Big[2^{-4d} \bar C2^k e^{-\bar c(|d-k/2|/\sqrt{k})^2} \Big] \lor  1 \Bigg].
\end{align*}
Finally, we get
\begin{align*}
\int_{-c\log(n)}^{c \log(n)}  \Bigg[k^2 M^2 & \sup_{d\leq k-2}\Big[ 2^{-8(k-d)}{k \choose d}^2 (k-d)^{k-d}   |\phi^{(d)}(t)|^2 \Big]  \Bigg] dt\\
&\leq 2 k^2 k! M^2 \Bigg[ \sup_{d\leq k/4}\Big[\bar C2^{k-4d}  e^{-\bar ck/16} \Big] \lor  1 \Bigg]\leq 2 k^2 k! M^2 \bigg[ \Big[ \bar C2^{k(1-\bar c/16)} \Big] \lor  1 \bigg],
\end{align*}
as desired.
\paragraph{Proof of Equation \eqref{lem:intHk} in Lemma \ref{cor:derk}.}
First, we begin by observing that 
\begin{align*}
\int_{|t|\leq c\log n} \exp(-t^2) &e^{2|t|\tilde c_{n}} |H_k(t)| ^2 dt 
= 2e^{\tilde c_{n}^2} \int_{0 \leq t \leq c\log n} \exp(-(t - \tilde c_{n})^2) |H_k(t)| ^2 dt\\
&= 2e^{\tilde c_{n}^2} \int_{-\tilde c_{n} \leq t \leq c\log n-\tilde c_{n}} \exp(-t^2) |H_k(t+\tilde c_{n})| ^2 dt
\leq 2e^{\tilde c_{n}^2} \int_{\mathbb R} \exp(-t^2) |H_k(t+\tilde c_{n})|^2 dt.
\end{align*}
By means of the following property of Hermite polynomials 
$$H_k(t+\tilde c_{n}) = \sum_{u=0}^k {k \choose u} H_u(t) \tilde c_{n}^{k-u}$$
joined with the Cauchy Schwarz inequality, we derive
$$H_k(t+\tilde c_{n})^2 \leq k \sum_{u=0}^k {k \choose u}^2 \tilde c_{n}^{2k-2u} |H_u(t)|^2.$$
Finally, using \eqref{eq:plancherel} and recalling the definition of Hermite polynomials, we get
\begin{align*}
\int_{\mathbb R} \exp(-t^2) |H_k(t+\tilde c_{n})|^2 dt &\leq   {2 e^{\tilde c_{n}^{2}}}k \sum_{u=0}^k {k \choose u}^2  \tilde c_{n}^{2k-2u} \int_{\mathbb R} \exp(-t^2)  |H_u(t)|^2 dt\\
&\leq   
\frac{4 e^{\tilde c_{n}^{2}}}{\sqrt{2\pi}} k \sum_{u=0}^k {k \choose u}^2  \tilde c_{n}^{2k-2u}  u!=\frac{4 e^{\tilde c_{n}^{2}}}{\sqrt{2\pi}} k (k!) \sum_{u=0}^k \frac{k!}{u!((k-u)!)^{2}}  \tilde c_{n}^{2k-2u}  \\
&\leq \frac{4 e^{\tilde c_{n}^{2}}}{\sqrt{2\pi}} k (k!) (1+  \tilde c_{n}^{2})^{k},
\end{align*}
which completes the proof.

\subsection{Proofs of the Propositions involved in the proof of Theorem~\ref{thm:LB}}
Since it will be used several times in the rest of the paper, we write BCI for the Bienaymé-Chebyshev inequality which states that, if $Z$ is a random variable with finite variance, then with probability larger than $1-\alpha$, it holds 
$$\E[Z]-\sqrt{\V(Z)/\alpha}\leq Z\leq \E[Z]+\sqrt{\V(Z)/\alpha}.$$
Also, to lighten the notation, we will sometimes avoid indexing with $\eps$, writing for example $X$ instead of $X(\eps)$, or $Z_i$ instead of $Z_i(\eps)$ and so on. For the same reason, we will sometimes write $N(u)$ instead of $N(u,\eps)$, for $0<u\leq\eps$.
Finally, in several occasions we will use that $\sigma(u^+)\leq \sigma(u^*).$

\subsubsection{Proof of Proposition~\ref{prop:testmax}}

\paragraph{Under $H_0$ :} By means of Equations \eqref{eq:defxi2} and \eqref{eq:defxip} we have that 
\begin{align*}
\max_i  Z_{i}(\omega) &\leq 4 \sqrt{\Delta\Sigma^2 \log(n)}, \quad \forall \omega\in \xi_n\quad \mbox{ and}\quad
S_n(\omega') \geq \frac{\sqrt{\Delta\Sigma^2}}{\pi}, \quad \forall\omega'\in \xi_n'.
\end{align*}
Therefore, for $n>e^{4\sqrt\pi}$, on the event $\xi_n\cap \xi'_n$ we have that $Z_{(\tilde n)} < \log(n)^{3/2}S_n,$ and thus $\Phi_n^{(\max)}=0$ as desired.

\paragraph{If a jump larger than $u^*$ occurs :} If $u^* = \eps$, then Proposition~\ref{prop:testmax} is satisfied as $\lambda_{\eps,\eps}=0$, \textit{i.e.} no jumps larger than $\eps$ happen. Assume from now on that $u^*<\eps$. By definition of $u^*$, and using that $\sigma(u)$ increases with $u$, we have that $u^* \geq \sqrt{(\Sigma^2+\sigma^2(u^*))\Delta} \log(n)^2$. Let us assume that $N_{n\Delta}(u^*) \geq 1$, \textit{i.e.} from now on we always condition by this event. This assumption, combined with \eqref{eq:defxi}, implies that on $\xi_n$ there exists $i$ such that
$N_{i\Delta}(u^*) - N_{(i-2)\Delta}(u^*) = 1,$
and therefore
$|M_{i\Delta}(u^*,\eps) - M_{(i-1)\Delta}(u^*,\eps) - ( M_{(i-1)\Delta}(u^*,\eps) -  M_{(i-2)\Delta}(u^*,\eps))| \geq u^*.$
In addition, by means of Equation~\eqref{eq:defxi}, we also know that on $\xi_n$
$$|\bar X_{i\Delta}(u^*) - \bar X_{(i-1)\Delta}(u^*) - (\bar X_{(i-1)\Delta}(u^*) - \bar X_{(i-2)\Delta}(u^*))| \leq 2\sqrt{(\Sigma^2+\sigma^2(u^*))\Delta} \log(n).$$
Recalling the definition of $u^*$ and taking $n>e^2$ we can conclude that, on $\xi_n,$ it holds that  $Z_i \geq u^*/2.$

Furthermore, by Equation~\eqref{eq:defxipup} we know that on $\xi'_n$
$$S_n \leq 2\sqrt{2\Delta(\Sigma^2+\sigma^2(u^+))} \leq 2\sqrt{2\Delta(\Sigma^2+\sigma^2(u^*))},$$
which allows to conclude that, for $n>e^2$, on $\xi_n\cap \xi'_n$, it holds 
$$Z_i\geq \frac{S_n\log(n)^2}{{2}\sqrt 2}>S_n \log(n)^{3/2},$$
that is $\Phi_n^{(\max)}=1$, as desired.

\subsubsection{Proof of Proposition~\ref{prop:test6}}

\paragraph{Under $H_0$:} By means of Equations~\eqref{eq:defxip} and  \eqref{eq:defxipe}, for any $\omega'\in \xi'_n$ and $\omega''\in \xi''_n$, we have 
$$S_n(\omega') \geq \frac{\sqrt{\Delta\Sigma^2}}{\sqrt \pi}\quad \quad \textnormal{ and } \quad \quad \bar Y_{n,6}(\omega'') \leq C (\sqrt{\Delta\Sigma^2})^6.$$
Therefore, on $\xi_n''\cap \xi'_n$, we have
$\bar Y_{n, 6} < C\pi^3S_n^6,$
and thus $\Phi_{n,c}^{(6)}=0$, as desired.

\paragraph{If $\Delta \mu_6(u^*)$ is large and no large jump occurs:} 
On the one hand, by Equation~\eqref{eq:defxipe} we know that, on $ \xi_n''\cap \{N_{n\Delta}(u^*) = 0\}$, it holds $$\bar Y_{n,6} \geq \frac{1}{2} \big[2\Delta \mu_6(u^*) + (\Delta(\Sigma^2+\sigma^2(u^*))^3\big].$$
On the other hand, on $ \xi'_n$, by means of Equation~\eqref{eq:defxipup}, we have that 
$S_n \leq 2\sqrt{2\Delta(\Sigma^2+\sigma^2(u^*))}.$
Thus, denoting by $C_c$ an absolute constant depending only on $c$, whenever 
$\Delta \mu_6(u^*) \geq C_c\Delta^3(\Sigma^2+\sigma^2(u^*))^3,$
it holds that $\bar Y_{n, 6} >c S_n^6,$ on $\xi_n''\cap \xi'_n\cap \{N_{n\Delta}(u^*) = 0\}$, for $n$ large enough.
We therefore conclude that $\Phi_{n,c}^{(6)}=1$, as desired.

\subsubsection{Proof of Proposition \ref{prop:LBNsym}}

We begin with some preliminary results proven in Section \ref{lemmathmLB}.

\begin{lemma}\label{lem:Vup3}
For $n$ larger than a universal constant, $\eps >0$ and any $\log(n)^{-1}<\alpha \leq 1$, there exist an event $\xi'''_n$ of probability larger than $1-\alpha$ and two universal constants $c,C>0$ such that the following holds:
\begin{align*}|\E[T^{(3)}_{n}(\eps)|\xi_n''']- \Delta\mu_{3}(\eps)|& \leq c\frac{\Delta^{3/2}(\Sigma^{2}+\sigma^{2}(\eps))^{3/2}}{\sqrt{n\alpha}},\\
\mbox{and}\hspace{1cm}
\V(T^{(3)}_{n}(\eps)|\xi'''_n)&\leq \frac{C}{n}(\Delta\mu_6(\eps) + \Delta^2(\Sigma^{2}+\sigma^{2}(\eps))^{3})
.
\end{align*} 
\end{lemma}
\begin{corollary}\label{cor:Vup3}
For any $\eps >0$ and for any $\log(n)^{-1}<\alpha \leq 1$, there exists an event $\xi'''_n$ of probability larger than $1-\alpha$ and two universal constants $c,C>0$ such that the following holds: 
\begin{align*}|\E[T^{(3)}_{n}(\eps)|\xi_n''',N_{n\Delta}(u^*,\eps)=0]- \Delta\mu_{3}(u^*)| &\leq c\frac{\Delta^{3/2}(\Sigma^{2}+\sigma^{2}(u^*))^{3/2}}{\sqrt{n\alpha}},\\
\mbox{and}\hspace{1cm}\V(T^{(3)}_{n}(\eps)|\xi'''_n,N_{n\Delta}(u^*,\eps)=0)&\leq \frac{C}{n}(\Delta\mu_6(u^*) + (\Delta(\Sigma^{2}+\sigma^{2}(u^*)))^{3}).
\end{align*} 
\end{corollary}

\begin{proof}[Proof of Proposition \ref{prop:LBNsym}]
For some given $\alpha$, let $\xi'''_n$ be an event as in Corollary~\ref{cor:Vup3}. If $3\leq k\leq  6$, thanks to the hypothesis \eqref{mu6} on $\Delta\mu_6(u^*)$, there exists an universal constant $C>0$ such that 
 \begin{align*}
 \V(T^{(3)}_{n}| N_{n\Delta}(u^*)=0, \xi'''_n)&\leq \frac{C}{n}( \Delta(\Sigma^{2}+\sigma^{2}(u^*)))^{3}.
 \end{align*} 
 Therefore, using BCI, we have
 \begin{align}\label{eq:dev3}
\PP\left(\bigg|\frac{T^{(3)}_n-\Delta\mu_{3}(u^*)}{\sqrt{\frac{C}{n}[\Delta(\Sigma^{2}+\sigma^{2}(u^*))]^{3}}}\bigg|>\frac{c/C+1}{\sqrt{\alpha}} | N_{n\Delta}(u^*)=0 \right)\leq 2\alpha.
\end{align} 

\paragraph{Under $H_0$:} $\mu_3(u^*)$ and $\sigma^2(u^*)$ are zero and thus
\begin{align*}
\PP\left(\bigg|\frac{T^{(3)}_n}{\sqrt{\frac{C}{n}\Delta^{3}\Sigma^{6}}}\bigg|>\frac{1}{\sqrt{\alpha}}\right)\leq\alpha.
\end{align*}
Therefore, recalling the definition of $\xi'_n$ (see Equation~\eqref{eq:defxip}), we have that for $c>0$ a large enough absolute constant, with probability larger than $1-\alpha - \mathbb P({\xi'_n}^{\mathsf c})$, it holds
$|T^{(3)}_n| \leq S_n^3 {c}/{\sqrt{\alpha{ n}}} ,$
which means that the test is accepted with probability larger than $1-2\alpha$.

\paragraph{Under $H_{1,\rho_n^{(3)}}^{(3)}$ and conditional to $N_{\Delta n}(u^*) = 0$ :} There exists a constant $C_c>0$, depending only on $c>0$, such that
$$\Delta|\mu_3(u^*)| \geq \Delta \rho \geq  \frac{C_c}{\sqrt{n\alpha}} \sqrt{( \Delta(\Sigma^{2}+\sigma^{2}(u^*)))^{3}}.$$
This implies by Equation~\eqref{eq:dev3} and for $C_c$ large enough depending only on $c$ that with probability larger than $1-\alpha$
$$|T^{(3)}_n| \geq \frac{C_c}{2\sqrt{n\alpha}} \frac{\sqrt{\Delta^3(\Sigma^2+\sigma^2(u^*))^3}}{\sqrt{n}}.$$
For $C_c$ large enough depending only on $c$, by definition of $\xi_n'$ (see \eqref{eq:defxip}) we have that with probability larger than $1-\alpha - \mathbb P({\xi'_n}^{\mathsf c})$ that 
$|T^{(3)}_n| \geq S_n^3{c}/{\sqrt{\alpha  n}}.$
The test is rejected with probability larger than $1-\alpha - \mathbb P({\xi'_n}^{\mathsf c})$. 
\end{proof}

\subsubsection{Proof of Proposition	\ref{prop:LBsym}}
Proposition \ref{prop:LBsym} can be proved with arguments very similar to those used in the proof of Proposition~\ref{prop:LBNsym}. 
\begin{lemma}\label{lem:VupSym} 
For any $\eps >0$ it holds
$\E[T^{(4)}_n(\eps)]=\Delta\mu_{4}(\eps).$
For $n$ larger than an absolute constant and for some universal constant $C>0$, it holds
\begin{align*}
\V(T^{(4)}_n(\eps))&\leq\frac{C}n \Big(\Delta\mu_8(\eps) + [\Delta(\Sigma^2+\sigma^2(\eps))]^4\Big).
\end{align*}
\end{lemma}

\begin{corollary}\label{cor:Vup4}
For any $\eps >0$, it holds $\E[T^{(4)}_n(\eps)|N_{n\Delta}(u^*,\eps) = 0]=\Delta\mu_{4}(u^*)$. Moreover, there exists a universal constant $C>0$ such that \begin{align}
\V(T^{(4)}_n(\eps)|N_{n\Delta}(u^*,\eps) = 0)&\leq\frac{C}n \Big(\Delta\mu_8(u^*) + (\Delta(\Sigma^2+\sigma^2(u^*)))^4\Big).\label{eq:Vup}
\end{align}
\end{corollary}
\begin{proof}[Proof of Proposition	\ref{prop:LBsym}]
The proof follows the same scheme as the one in Lemma~\ref{prop:LBNsym}. Here we only observe that Equation~\eqref{eq:Vup} implies 
\begin{align*}
\V(T^{(4)}_n(\eps)|N_{n\Delta}(u^*,\eps) = 0)&\leq\frac{C}n \Big(\Delta\mu_8(u^*) + \big(\Delta(\Sigma^2+\sigma^2(u^*))\big)^4\Big)\leq \frac{C}n \Big((u^*)^4\Delta\mu_4(u^*) + \big(\Delta(\Sigma^2+\sigma^2(u^*))\big)^4\Big),
\end{align*}
since $\mu_8(u^*)\leq (u^*)^4\mu_4(u^*)$. By means of BCI we thus deduce that
\begin{align*}
\PP\left(\Bigg|\frac{T^{(4)}_n-\Delta\mu_{4}(u^*)}{\sqrt{\frac{C}n \Big((u^*)^4\Delta\mu_4(u^*) + \big(\Delta(\Sigma^2+\sigma^2(u^*))\big)^4\Big)}}\Bigg|>\frac{1}{\sqrt{\alpha}} | N_{n\Delta}(u^*)=0 \right)\leq \alpha,
\end{align*} 
which, using that $\alpha \geq \log(n)^{-1}$ and the definition of $u^*$, implies 
\begin{align*}
\PP\left(\frac{T^{(4)}_n-\frac{1}{2}\Delta\mu_{4}(u^*)}{\sqrt{\frac{C'}n  (\Delta(\Sigma^2+\sigma^2(u^*)))^4}}<\frac{1}{\sqrt{\alpha}} | N_{n\Delta}(u^*)=0 \right)\leq \alpha,
\end{align*} 
for some universal constant $C'$ and for $n$ larger than a universal constant. 
\end{proof}
\subsection{Proofs of the Lemmas involved in the proof of Theorem~\ref{thm:LB}}\label{lemmathmLB}
Hereafter, when there is no ambiguity we drop the dependency $\eps$, writing for example $X$ instead of $X(\eps)$, or $Z_i$ instead of $Z_i(\eps)$ and so on. For the same reason, we sometimes write $N(u)$ instead of $N(u,\eps)$.
\subsubsection{Proof of Lemma \ref{lem:TvTest}}

Let $\Phi$ be such a test for $H_{0}:\ \PP$ against $H_{1}:\ \QQ$. The conditions on $\Phi$ lead to
\begin{align*}
\|\PP^{\otimes n}-\QQ^{\otimes n}\|_{TV} \geq  |\PP_{H_{0}}(\Phi=0) -\PP_{H_{1}}(\Phi=0)| &\geq 1-\alpha_1 - \alpha_0.
\end{align*}

\subsubsection{Proof of Lemma~\ref{lem:defxi}}

\paragraph{If $\nu_\eps = 0$ :} Under $H_0$ we know that all $Z_{i}$ are i.i.d.~realizations of the absolute values of centered Gaussian random variables with variance $2\Delta\Sigma^2$. By Gaussian concentration and using that $\tilde n = \lfloor n/2\rfloor \leq n$, with probability larger than $1-1/n$ it holds that
$\max_{i\leq  { \tilde n}}  Z_{i} \leq 4 \sqrt{\Delta\Sigma^2 \log(n)}$.

\paragraph{If $\nu_\eps \neq 0$ :} By BCI, with probability larger than $1-\alpha$, it holds
$|N_{n\Delta}(u^+,\eps) - \Delta n \lambda_{u^+,\eps}| \leq \sqrt{\Delta n \lambda_{u^+,\eps}/\alpha},$
\textit{i.e.} for $\alpha = 4\log(n)^{-1}$ we have that with probability larger than $1-4\log(n)^{-1}$
$$\log(n)/2 \leq N_{n\Delta}(u^+,\eps) \leq 2\log(n).$$
Furthermore, observe that
 \begin{align*}
\mathbb P(N_{i\Delta}(u^+,\eps) - N_{(i-2)\Delta}(u^+,\eps) \leq 1) &=  \exp(-2\Delta \lambda_{u^+,\eps}) + 2\Delta \lambda_{u^+,\eps} \exp(-2\Delta \lambda_{u^+,\eps})\\ 
&=  \exp(-2\log(n)/n) +2\frac{\log(n)}{n}\exp(-2\log(n)/n).
\end{align*}
It follows
\begin{align*}
\mathbb P \Big(\{\forall  i \leq n, N_{i\Delta}(u^+,\eps) - N_{(i-2)\Delta}(u^+,\eps) \leq 1\}\Big) &= \Big( \exp(-2\log(n)/n) +2\frac{\log(n)}{n}\exp(-2\log(n)/n)\Big)^n\\
&=  \exp(-2\log(n))(1+ 2\log(n)/n)^n \rightarrow 1,
\end{align*}
at a rate which does not depend on $\nu_\eps, \eps, b,\Sigma$.

Finally, by BCI, with probability larger than $1-\alpha$, we have 
$$|\bar X_{i\Delta}(u^*) - \bar X_{(i-1)\Delta}(u^*) - b(u^*)\Delta| \leq \sqrt{(\Sigma^2+\sigma^2(u^*))\Delta \alpha^{-1}}.$$
So, conditional on $\{1\leq N_{n\Delta}(u^+,\eps) \leq 2\log(n)\}$, with probability larger than $1-\log(n)^{-1}$, we have that $\forall i~s.t.~N_{i\Delta}(u^+,\eps) - N_{(i-1)\Delta}(u^+,\eps) \neq 0$
$$|\bar X_{i\Delta}(u^*) - \bar X_{(i-1)\Delta}(u^*) - b(u^{*})\Delta| \leq \sqrt{(\Sigma^2+\sigma^2(u^*))\Delta} \log(n).$$
We conclude observing that, conditional on $\{1\leq N_{n\Delta}(u^+,\eps) \leq 2\log(n)\} $, with probability larger than $1-\log(n)^{-1}$, we have $\forall i~s.t.~N_{i\Delta}(u^+,\eps) - N_{(i-2)\Delta}(u^+,\eps) \neq 0$ $$|\bar X_{i\Delta}(u^*) - \bar X_{(i-1)\Delta}(u^*) - (\bar X_{(i-1)\Delta}(u^*) - \bar X_{(i-2)\Delta}(u^*))| \leq 2\sqrt{(\Sigma^2+\sigma^2(u^*))\Delta} \log(n).$$

\subsubsection{Proof of Lemma \ref{lem:med}}

\paragraph{Preliminary.} 
Denote by  $\tilde Z_i = |(X_{i\Delta}(u^+)-X_{(i-1)\Delta}(u^+))- (X_{(i-1)\Delta}(u^+)-X_{(i-2)\Delta}(u^+))|$ and assume that $\xi_n$ holds. We begin by observing that, for any $\omega\in \xi_n$, we have:
\begin{align}\label{eq:tildeZ}
\frac{1}{\tilde n} \sum_{i=1}^{\tilde n-4\log(n)} \tilde Z_{(i)}(\omega) \leq S_n(\omega) \leq \frac{1}{\tilde n} \sum_{i=1}^{\tilde n} \tilde Z_i(\omega).
\end{align} 
To show \eqref{eq:tildeZ}, let $I:=\{i: \tilde Z_i=Z_i\}$, that is the set where no jumps of size larger than $u^+$ occur between $(i-2)\Delta$ and $i\Delta$. By means of the positivity of the variables $\tilde Z_i$ and $Z_i$, we get 
\begin{align*}
\frac{1}{\tilde n} \sum_{1\leq i \leq \tilde n-2\log(n), i \in I}  Z_{(i)} \leq S_n = \frac{1}{\tilde n} \sum_{i=1}^{\tilde n - 2\log(n)} Z_{(i)}.
\end{align*}
Moreover, since $\#I^c \leq 2\log(n)$ on $\xi_n$, we have
\begin{align*}
\frac{1}{\tilde n} \sum_{1\leq i \leq \tilde n-2\log(n), i \in I}  Z_{(i)} \leq S_n \leq \frac{1}{\tilde n} \sum_{i\in I} Z_{i}.
\end{align*}
Using again that $\PP\{\#I^{\mathsf c}\leq 2\log(n)\cap\xi_n\}=1$, the definition of $I$ and the fact that $\tilde Z_i, Z_i$ are positive, we obtain \eqref{eq:tildeZ}.

\paragraph{Control when $\xi'_{n}$ is given by \eqref{eq:defxipup}.}  Note that
$\mathbb E \tilde Z_i^2 = 2\Delta (\Sigma^2+\sigma^2(u^+)),$
so by Cauchy-Schwartz inequality
$\mathbb E \tilde Z_i \leq \sqrt{2\Delta (\Sigma^2+\sigma^2(u^+))}.$
It follows by BCI, that with probability larger than $1-1/n$
$$\frac{1}{\tilde n} \sum_{i=1}^{\tilde n} \tilde Z_i \leq 2\sqrt{2\Delta (\Sigma^2+\sigma^2(u^+))}.$$
Then, on $\xi_n$, by Equation~\eqref{eq:tildeZ}, with probability larger than $1-1/n$
it holds $S_n \leq 2\sqrt{2\Delta (\Sigma^2+\sigma^2(u^+))}.$

\paragraph{Control when $\xi'_{n}$ is given by \eqref{eq:defxip}.} In this case $\nu_{\eps}=0$, and thus $Z_i = \tilde Z_i$ are i.i.d.~and distributed as the absolute value of a centered Gaussian random variable with variance $2\Delta \Sigma^2$. By Gaussian concentration it then follows that with probability larger than $1-\alpha$
$$\max_i |Z_i| \leq 2\sqrt{2\Delta \Sigma^2 \log(2/\alpha)}.$$
Using that $\mathbb E[ |\mathcal N(0,1)|] = \sqrt{2/\pi}$ and BCI, we conclude that with probability larger than $1-\alpha$
$$\frac{1}{\tilde n} \sum_{i=1}^{\tilde n - 4\log(n)} \tilde Z_{(i)}  {\geq} \sqrt{\frac2\pi}\sqrt{2\Delta \Sigma^2} -\sqrt{\frac{1}{n\alpha}2\Delta \Sigma^2} - \frac{4\log(n)}{n}\sqrt{2\Delta \Sigma^2 \log(2/\alpha)}.$$
By Equation~\eqref{eq:tildeZ}, with probability larger than $1-2^{9}\pi n^{-1}$, it holds
$S_n \geq {\sqrt{\Delta \Sigma^2}}/{\sqrt \pi}.$

\subsubsection{Preliminaries for the proofs of Lemmas \ref{lem:mo6}, \ref{lem:Vup3} and \ref{lem:VupSym}}

If $Y\sim\mathcal{N}(m,\Sigma^{2})$ its moments can be computed through the recursive formula:
\[\E[Y^{k}]=(k-1)\V(Y)\E[Y^{k-2}]+\E[Y]\E[Y^{k-1}],\quad k\in\N.\] 
\begin{lemma}\label{lem:MmtNormale}Let $Y\sim\mathcal{N}(m,\Sigma^{2})$, then it holds \begin{align*}
\E[Y^{3}]&=3\Sigma^{2}m+m^{3},\quad
\E[Y^{4}]=3\Sigma^{4}+6m^{2}\Sigma^{2}+m^{4},\quad
\E[Y^{5}]=15 m\Sigma^4+10m^3\Sigma^2+m^5,\\
\E[Y^{6}]&=15\Sigma^6+45m^2\Sigma^4+15m^4\Sigma^2+m^6,\quad
\E[Y^{7}]=105m^4\Sigma^6+105m^3\Sigma^4+21m^5\Sigma^2+m^7,\\
\E[Y^{8}]&=105\Sigma^8+420m^2\Sigma^6+210m^4\Sigma^4+28m^6\Sigma^2+m^8.
\end{align*}\end{lemma}
Similarly, using the series expansion of the characteristic function, together with the Lévy-Kintchine formula, we get \begin{align}
\label{eq:momentM}\E[M_{\Delta}(\eps)^{k}]=\frac{d^{k}}{du^{k}}\exp\Big(\Delta\int_{|x|\leq \eps}(e^{iux}-1 -iux)\nu_{\eps}(dx)\Big)\bigg|_{u=0},\quad \forall \eps>0.\end{align}

\begin{lemma}\label{lem:MmtSJ}For $\eps>0$, set $\sigma^2(\eps)=\int_{-\eps}^{\eps}y^2\nu(dy)$ and $\mu_{k}(\eps):=\int_{-\eps}^{\eps}y^{k}\nu(dy)$, $k\geq 3$. We have
 \begin{align*}\E[M_\Delta(\eps)]&=0, \quad\E[(M_\Delta(\eps))^2]=\Delta\sigma^2(\eps),\quad \E[(M_\Delta(\eps))^3]=\Delta\mu_{3}(\eps)\\ \E[(M_\Delta(\eps))^4]&=\Delta\mu_{4}(\eps)+3\Delta^{2}\sigma^{4}(\eps),\quad
\E[(M_\Delta(\eps))^5]=\Delta\mu_{5}(\eps)+10\Delta^{2}\sigma^{2}(\eps)\mu_{3}(\eps),\\
\E[(M_\Delta(\eps))^6]&=\Delta\mu_{6}(\eps)+\Delta^{2}\big(10\mu_{3}(\eps)^{2}+15\sigma^{2}(\eps)\mu_{4}(\eps)\big)+15\Delta^{3}\sigma^{6}(\eps),\\
\E[(M_\Delta(\eps))^7]&=\Delta\mu_{7}(\eps)+\Delta^{2}\big(21\sigma^{2}(\eps)\mu_{5}(\eps)+35\mu_{3}(\eps)\mu_{4}(\eps)\big)+105\Delta^{3}\sigma^{4}(\eps)\mu_{3}(\eps),\\
\E[(M_\Delta(\eps))^8]&=\Delta\mu_{8}(\eps)+\Delta^{2}\big(35\mu_{4}(\eps)^{2}+56\mu_{3}(\eps)\mu_{5}(\eps)+28\sigma^{2}(\eps)\mu_{6}(\eps)\big)\\ & \quad+\Delta^{3}\big(280\sigma^{2}(\eps)\mu_{3}(\eps)^{2}+210\sigma^{4}(\eps)\mu_{4}(\eps)\big)+105\Delta^{4}\sigma^{8}(\eps).\end{align*}
More generally if $\nu_\eps$ is a Lévy measure such that $\int_{|x|\leq \eps} \nu(dx) >\Delta^{-1}$, then for any $k\geq 2$ even integer it holds that
$$ \E[M_\Delta(\eps)^k]\leq \bar C_k\Big(\Delta\mu_{k}(\eps) + (\Delta\sigma^{2}(\eps))^{k/2}\Big),$$
where $\bar C_k>0$ is a constant that depends only on $k$.
 \end{lemma}
 \begin{proof}[Proof of Lemma~\ref{lem:MmtSJ}]
 The explicit first 8 moments are computed using Equation~\eqref{eq:momentM}. We prove now the last part of the Lemma.
In accordance with the notation introduced in \eqref{cppapprox} and \eqref{cppY}, denote by $M(\eta,\eps)$ the Lévy process with Lévy measure $\nu_{\eps}\mathbf 1_{[-\eta,\eta]^c}$, $N(\eta,\eps)$ the corresponding Poisson process and by $Y_1$ a random variable with law given by $\frac{\nu_{\eps}\mathbf 1_{[-\eta,\eta]^c}}{\lambda_{\eta,\eps}}$.
 
We have
\begin{equation}\label{dec}
\E[|M_\Delta(\eta,\eps)|^k]\leq 2^{k-1}\Bigg(\E\bigg[\Big|\sum_{i=1}^{N_\Delta(\eta,\eps)}(Y_i-\E[Y_i])\Big|^k\bigg]+\bigg(\frac{\int_{\eta<|x|\leq \eps}x\nu_\eps(dx)}{\lambda_{\eta,\eps}}\bigg)^k\E\big[|N_\Delta(\eta,\eps)-\lambda_{\eta,\eps}\Delta|^k\big]\Bigg).
\end{equation}
By Rosenthal's inequality there exists a constant $\tilde C_k$, depending on $k$ only, such that for any $k \geq 2$, $k$ even\begin{align*}
\E\bigg[\Big|\sum_{i=1}^{N_\Delta(\eta,\eps)}(Y_i-\E[Y_i])\Big|^k | N_\Delta(\eta, \eps) = N\bigg] &\leq \tilde C_k \max\Big( N\mathbb E\big[(Y_1-\E[Y_1])^{k}\big], \big(N \V(Y_1)\big)^{k/2}\Big).
\end{align*}
Averaging over the Poisson random variable $N_\Delta(\eta,\eps)$ and setting $\mu_k(\eta,\eps) = \int_{\eta\leq |x| \leq \eps}y^{k}\nu(dy)$, for $\eta$ small enough we have that
\begin{align*}
\E\bigg[\Big|\sum_{i=1}^{N_\Delta(\eta,\eps)}(Y_i-\E[Y_i])\Big|^k\bigg] &\leq \tilde C_k \max\Big( \Delta\lambda_{\eta,\eps}\mathbb E\big[(Y_1-\E[Y_1])^{k}\big], \mathbb E [N_\Delta(\eta,\eps)^{k/2}]\mathbb \V(Y_1)^{k/2}\Big)\\
&\leq C_k'' \max\Big( \Delta \mu_k(\eta,\eps), [\Delta \mu_2(\eta,\eps)]^{k/2}\Big),
\end{align*}
where $\tilde C_k''$ is some constant that only depends on $k$. Let us now consider the second addendum in \eqref{dec}. By Cauchy-Schwarz inequality and using that $\E\big[|N_\Delta(\eta,\eps)-\lambda_{\eta,\eps}\Delta|^k]\leq C_k (\lambda_{\eta,\eps}\Delta)^{k/2}$ for some constant $C_k$ only depending on $k$, we derive
\begin{align*}
\frac{\int_{\eta<|x|\leq \eps}x\nu_\eps(dx)}{\lambda_{\eta,\eps}}\E\big[|N_\Delta(\eta,\eps)-\lambda_{\eta,\eps}\Delta|^k\big]&\leq \bigg(\frac{\mu_2(\eta,\eps)}{\lambda_{\eta,\eps}}\bigg)^{k/2}\E\big[|N_\Delta(\eta,\eps)-\lambda_{\eta,\eps}\Delta|^k\big]\leq C_k (\Delta\mu_2(\eta,\eps))^{k/2}.
\end{align*}
 We therefore deduce that there exists a constant $\bar C_k$, only depending on $k$, such that 
 \begin{equation}\label{mm}
  \E[|M_\Delta(\eta,\eps)|^k]\leq\bar C_k \max\Big( \Delta \mu_k(\eta,\eps), [\Delta \mu_2(\eta,\eps)]^{k/2}\Big).
  \end{equation} In particular, we conclude that the family of random variables $(M_\Delta(\eta,\eps)^k)_{0<\eta\leq\eps}$ is uniformly integrable. Therefore, using also that the family of processes $M(\eta,\eps)$ converges almost surely to $M_\Delta(\eps)$, as $\eta\to 0$, we derive that
$
\lim_{\eta\to0} \E[|M_\Delta(\eta,\eps)|^k]=\E[|M_\Delta(\eps)|^k].
$
Making $\eta$ converge to $0$ in \eqref{mm} gives the desired result.
 \end{proof}

 \begin{corollary}\label{cor:coco2}
 For $\eps >0$, set $\sigma^2(\eps)=\int_{|y|\leq \eps}y^2\nu(dy)$ and $\mu_{k}(\eps):=\int_{|y|\leq \eps}y^{k}\nu(dy)$, $k \geq 2$.
 It holds that
 \begin{align*}
   \E [X_\Delta(\eps) - b(\eps)\Delta] &= 0,\hspace{1cm}
    \E \big[(X_\Delta(\eps) - b(\eps)\Delta)^3\big] =  \Delta\mu_3(\eps),\hspace{1cm}
  \E [X_\Delta(\eps) - \bar X_{\Delta,n}(\eps)] = 0,\\
    \E \big[(X_\Delta(\eps) - \bar X_{\Delta,n}(\eps))^3\big] &={\Big(1- \frac{1}{(n - \lfloor n/2\rfloor)^2} \Big)}\Delta\mu_3(\eps),\hspace{2.0cm}  \E [Z_1^2(\eps)] = 2 \Delta(\Sigma^2+\sigma^{2}(\eps)),\\
  \E [Z_1^4(\eps)]& = 4\Delta \mu_4(\eps) + 12 (\Delta(\Sigma^2+\sigma^{2}(\eps)))^2,\hspace{1.0cm}
  \E [Z_1^6(\eps)] \geq  2\Delta \mu_6(\eps) + (\Delta(\Sigma^2+\sigma^{2}(\eps)))^3.
 \end{align*}
 \end{corollary}
 \begin{proof}
The result is obtained combining Lemmas~\ref{lem:MmtNormale} and~\ref{lem:MmtSJ} with the independence between $(X_{2\Delta}(\eps) - X_\Delta(\eps))$ and $X_\Delta(\eps)$.
 \end{proof}
 
\begin{corollary}\label{cor:coco} For $\eps>0$, set $\sigma^2(\eps)=\int_{|y|\leq \eps}y^2\nu(dy)$ and $\mu_{k}(\eps):=\int_{|y|\leq \eps}y^{k}\nu(dy)$, $k \geq 2$. It holds for any $k\geq 2$, $k$ even integer that
\begin{align*}
\mathbb E \big[(X_\Delta(\eps) - b(\eps)\Delta)^k\big] \leq \bar C_k ( \Delta\mu_{k}(\eps) + (\Delta(\Sigma^2+\sigma^{2}(\eps)))^{k/2}),\\
\mathbb E [Z_1^k(\eps)] \leq \bar C_k ( \Delta\mu_{k}(\eps) + (\Delta(\Sigma^2+\sigma^{2}(\eps)))^{k/2}),
\end{align*}
where $\bar C_k$ is a constant that depends only on $k$. 
\end{corollary} 
\begin{proof}[Proof of Corollary~\ref{cor:coco}]
By means of Lemma \ref{lem:MmtSJ} we have
\begin{align*}
\mathbb E\big[ (X_\Delta(\eps) - b(\eps)\Delta)^k\big] &= \mathbb E (M_\Delta(\eps) +\Sigma W_\Delta)^k \leq 2^{k-1}\Big(\mathbb E\big[ M_\Delta(\eps)^k\big] +\Sigma^k \mathbb E\big [W_\Delta^k\big]\Big)\\
&\leq C_k (\Delta^{k/2}\Sigma^k + \Delta\mu_{k}(\eps) + (\Delta\sigma^{2}(\eps))^{k/2}),
\end{align*}
where $C_k$ is a constant that depends on $k$ only.
\end{proof}

\subsubsection{Proof of Lemma~\ref{lem:mo6}}

\paragraph{If $\nu_\eps = 0$.} By Corollary~\ref{cor:coco}, there exist universal constants $\overline C_{6}$ and $\overline C_{12}$ such that
$
  \E[Z_i^6] \leq \bar C_6\Delta^3\Sigma^6$ {and} $\mathbb E[Z_i^{12}] \leq \bar C_{12}\Delta^6\Sigma^{12}.$
 Using that $Z_i$ are i.i.d. we get
\begin{align*}
  \E[\bar Y_{n,6}] \leq  \bar C_6\Delta^3\Sigma^6,\quad \mathbb V(\bar Y_{n,6}) \leq  \frac{\bar C_{12}}{n - \lfloor n/2\rfloor}\Delta^6\Sigma^{12}.
\end{align*}
Therefore, by means of BCI, with probability larger than $1-\log(n)^{-1}$ conditional to $N_{n\Delta}(u^*) = 0$ we have
\begin{align*}
 \bar Y_{n,6} \leq \bar C_6\Delta^3\Sigma^6 + \sqrt{\frac{\bar C_{12}\log(n)}{(n - \lfloor n/2\rfloor)}  (\Delta\Sigma^2)^{6}},
\end{align*}which allows to deduce that for $n$ larger than a universal constant, with probability larger than $1-\log(n)^{-1}$, we have 
  $\bar Y_{n,6} \leq 2 \bar C_6\Delta^3\Sigma^6.$

\paragraph{If $\nu_\eps \neq 0$.} By Corollaries~\ref{cor:coco2} and~\ref{cor:coco}, conditional to $N_{n\Delta}(u^*) = 0$, for any $i$ it holds
\begin{align*}
  \E[Z_i^6|N_{n\Delta}(u^*) = 0]& \geq \Delta \mu_6(u^*) + (\Delta(\Sigma^2+\sigma^{2}(u^*)))^3\\
\mathbb E[Z_i^{12}|N_{n\Delta}(u^*) = 0]& \leq \bar C_{12} ( \Delta\mu_{12}(u^*) + (\Delta(\Sigma^2+\sigma^{2}(u^*)))^{6}),
\end{align*}
where $\bar C_{12}$ is the universal constant from Corollary~\ref{cor:coco}. Using that the $Z_i$ are i.i.d. we obtain
\begin{align*}
  \E[\bar Y_{n,6}|N_{n\Delta}(u^*) = 0] \geq \Delta \mu_6(u^*) + (\Delta(\Sigma^2+\sigma^{2}(u^*)))^6\\
\mathbb V(\bar Y_{n,6}|N_{n\Delta}(u^*) = 0) \leq  \frac{\bar C_{12}}{n - \lfloor n/2\rfloor} ( \Delta\mu_{12}(u^*) + (\Delta(\Sigma^2+\sigma^{2}(u^*)))^{6}).
\end{align*}
It follows by BCI that with probability larger than $1-\log(n)^{-1}$, conditional to $N_{n\Delta}(u^*) = 0$, we have
\begin{align*}
 \bar Y_{n,6} \geq  \Delta \mu_6(u^*) + (\Delta(\Sigma^2+\sigma^{2}(u^*)))^3 - \sqrt{\frac{\bar C_{12}\log(n)}{(n - \lfloor n/2\rfloor)} ( \Delta\mu_{12}(u^*) + (\Delta(\Sigma^2+\sigma^{2}(u^*)))^{6})}.
\end{align*}
Since $\mu_{12}(u^*) \leq (u^*)^{6}\mu_{6}(u^*)$, with probability larger than $1-\log(n)^{-1}$ and conditional to $N_{n\Delta}(u^*) = 0$, we have
\begin{align*}
 \bar Y_{n,6} \geq \Delta \mu_6(u^*) + (\Delta(\Sigma^2+\sigma^{2}(u^*)))^3 - \sqrt{\frac{\bar C_{12}\log(n)}{(n - \lfloor n/2\rfloor)} ( (u^*)^6\Delta\mu_{6}(u^*) + (\Delta(\Sigma^2+\sigma^{2}(u^*)))^{6})}.
\end{align*}
Finally, by means of the definition of $u^*$ and for $n$ larger than a universal constant,  with probability larger than $1-\log(n)^{-1}$ conditional to $N_{n\Delta}(u^*) = 0$, it holds
\begin{align*}
 \bar Y_{n,6} \geq \frac{\Delta \mu_6(u^*) + (\Delta(\Sigma^2+\sigma^{2}(u^*)))^3}{2}.
\end{align*}

\subsubsection{Proof of Lemma \ref{lem:Vup3}}

Note that $\mathbb E[ \bar X_{\Delta,n}(\eps)] = b(\eps)\Delta$ and that $\mathbb V(\bar X_{\Delta,n}(\eps)) = (n - \lfloor n/2 \rfloor)^{-1}\Delta (\Sigma^2 + \sigma^2(\eps))$. Write 
$$\xi'''_n = \{|\bar X_{\Delta,n}(\eps) - b(\eps)\Delta| \leq r_n\}\quad \mbox{where}\quad r_n:= r_n(\alpha) = \sqrt{\frac{\Delta (\Sigma^2 + \sigma^2(\eps))}{\alpha(n - \lfloor n /2\rfloor)^{-1}}}.$$
By BCI we have for any $0<\alpha \leq 1$
\begin{align}
\mathbb P(\xi_n''') \geq  1-\alpha.\label{eq:Ek}
\end{align}
Conditional on $\xi'''_n$, by Corollary~\ref{cor:coco2} and the definition of $r_n$,
$$|\E[\bar Y_{n,3}|\xi'''_n] - \Delta\mu_{3}(\eps)| \leq r_n^3 + 3 \Delta(\Sigma^{2}+\sigma^{2}(\eps)) r_n \leq 100 \frac{\big(\Delta (\Sigma^2 + \sigma^2(\eps))\big)^{3/2}}{\sqrt{n\alpha}},$$
for $\alpha \geq \log(n)^{-1}$.

Now we compute $\V(T^{(3)}_n)$. Since $\bar X_{\Delta,n}$ and $\bar X_{\Delta,n}^j$ are independent, as they are computed on two independent samples, the elements of the sum are independent of each other {conditional on the second half of the sample.} Then, conditional on the second half of the sample,
\begin{align*}
\V(T^{(3)}_n|\xi'''_n)&\leq {\Big(1- \frac{1}{(n - \lfloor n/2\rfloor)^2} \Big)^{2}}  \frac{1}{\lfloor n/2\rfloor^2} \sum_{i=1}^{\lfloor n/2\rfloor} \E\Big[\big[(X_{i\Delta}(\eps)- X_{(i-1)\Delta}(\eps) -\bar X_{\Delta,n}(\eps) )^3\big]^2|\xi''_n\Big]\\
&\leq \frac{16}{n} \E\Big[\big[X_{i\Delta}(\eps)- X_{(i-1)\Delta}(\eps) -\Delta b(\eps) \big]^6+ \big[\bar X_{\Delta,n}(\eps)-\Delta b(\eps) \big]^6|\xi''_n\Big]\\
&\leq \frac{16}{n}\Big[\bar C_6\Big(\Delta\mu_{6}(\eps) + \big[\Delta(\Sigma^2+\sigma^{2}(\eps))\big]^{3}\Big) + r_n^6\Big],
\end{align*} 
where $\bar C_6$ is the constant from Corollary~\ref{cor:coco}. Hence, by Equation~\eqref{eq:Ek}, there exists a universal constant $C$ such that
\begin{align*}
\V(T^{(3)}_n|\xi'''_n)&\leq  \frac{C}{n}\Big[ \Delta\mu_{6}(\eps) + \big[\Delta(\Sigma^2+\sigma^{2}(\eps))\big]^{3}\Big].
\end{align*}

\subsubsection{Proof of Lemma \ref{lem:VupSym}}
The main ingredient of the proof consists in establishing expansions of $\V(T^{(4)}_n)$. Computations are cumbersome but not difficult, we only give the main tools here but we do not provide all computations.
By Corollary~\ref{cor:coco2} and since $\bar Y_{n,2}$ and $\bar Y_{n,2}'$ are independent, we have
\begin{align*}
\E[\bar Y_{n,2}]&= 2\Delta(\Sigma^{2}+\sigma^{2}(\eps)),\hspace{1cm}
\E[\bar Y_{n,2}\bar Y_{n,2}']= 4\Delta^{2}(\Sigma^{2}+\sigma^{2}(\eps))^{2},\\
\E[\bar Y_{n,4}]&=12\Delta^{2}(\Sigma^{2}+\sigma^{2}(\eps))^{2}+4\Delta\mu_{4}(\eps).
\end{align*}
In particular, $\E[T^{(4)}_n]=\Delta\mu_{4}(\eps)$.
Next, we have
$
\V[T^{(4)}_n]\leq  9\V[\bar Y_{n,2}\bar Y_{n,2}'] + \V[\bar Y_{n,4}].
$
We analyse these two terms separately.

Since the $Z_i$ in the sum composing $\bar Y_{n,4}$ are i.i.d.~we have
\begin{align*}
\V(\bar Y_{n,4}) &\leq \frac{1}{\lfloor n/2\rfloor^2} \sum_{i=1}^{\lfloor n/2\rfloor}  \mathbb E (Z_i^8)\leq \frac{1}{\lfloor n/2\rfloor}2^9\bar C_8\Big(\Delta\mu_{8}(\eps) + \big[\Delta(\Sigma^2+\sigma^{2}(\eps))\big]^{4}\Big)\\
&\leq \frac{C}{n}\Big(\Delta\mu_{8}(\eps) + \big[\Delta(\Sigma^2+\sigma^{2}(\eps))\big]^{4}\Big),
\end{align*}
where $\bar C_8$ is the constant from Corollary~\ref{cor:coco}, and where $C$ is a universal constant.

Similarly, as the $Z_i$ in the sums composing $\bar Y_{n,2}$ and $\bar Y_{n,2}'$ are i.i.d.~we have
\begin{align*}
\V(\bar Y_{n,2}\bar Y_{n,2}') &\leq 4\mathbb E[\bar Y_{n,2}^2]\mathbb V(\bar Y_{n,2}')
\leq 4 \frac{\bar C_4}{\lfloor n/4\rfloor}\Big(\Delta\mu_{4}(\eps) + \big[\Delta(\Sigma^2+\sigma^{2}(\eps))\big]^{2}\Big)\\
&\hspace{3cm}\times \Big[\frac{\bar C_4}{\lfloor n/4\rfloor}\Big(\Delta\mu_{4}(\eps) + \big[\Delta(\Sigma^2+\sigma^{2}(\eps))\big]^{2}\Big) + 4\Delta^2(\Sigma^{2}+\sigma^{2}(\eps))^2\Big]\\
&\leq \frac{C'}{n}\Big( \big[\Delta(\Sigma^2+\sigma^{2}(\eps))\big]^{4} + \frac{\Delta^2\mu_{4}(\eps)^2}{n}\Big),
\end{align*}
where $\bar C_4$ is the constant from Corollary~\ref{cor:coco}, and where $C'$ is a universal constant. Combining this with the last displayed equation, completes the proof.

 \subsection{Proof of Proposition \ref{prop:exstable}\label{app:stable}}

Let $\beta \in (0,2)$ and consider a L\'evy  measure $\nu$ that has a density with respect to the Lebesgue measure such that there exist two positive constants $c_+>c_->0$ with
$$\frac{c_-}{|x|^{\beta+1}} \leq \frac{d\nu(x)}{dx} \leq \frac{c_+}{|x|^{\beta+1}},\quad \forall x\in[-\eps,\eps]\setminus\{0\}.$$
Let $X(\eps)\sim (b,\Sigma^2,\nu\1_{|x|\leq \eps})$. The characteristic function of $X_\Delta(\eps)$ is
$\tilde \Psi = \exp(\tilde\psi),$
with
$$\tilde\psi(t) =itb-\Delta t^2 \Sigma^2/2 + \Delta \int_{-\eps}^\eps (e^{itu} - 1 - itu) d\nu(u).$$
The rescaled increment, denoted by $\tilde X_{\Delta}(\eps)$, has characteristic function
$ \Psi = \exp(\psi)$ with
$$\psi(t) = - t^2 \frac{\Sigma^2}{2(\Sigma^2 +\sigma^2(\eps))} +  \Delta  \int_{|u|\leq \eps} (e^{itu/\sqrt{\Delta s^{2}}} - 1 - itu/\sqrt{\Delta s^{2}}) d\nu(u),$$ where
$\sigma^2(\eps) = \int_{|u|\leq \eps} u^2 d\nu(u) \in [\frac{2c_-}{2-\beta} \eps^{2-\beta}, \frac{2c_+}{2-\beta} \eps^{2-\beta}].$
From now on, write $s^2 = \Sigma^2 +\sigma^2(\eps)$. We have
$$\psi(t) = - t^2 \frac{\Sigma^2}{2s^2} +  \Delta  \int_{|u|\leq \eps} (e^{itu/\sqrt{\Delta s^{2}}} - 1 - itu/\sqrt{\Delta s^{2}}) d\nu(u).$$
Note that the cumulants of $\tilde X_{\Delta}(\eps)$ are such that $ \mu_1(\eps) = 0$, $\mu_2(\eps)=1$ and for all $k\geq 3$
$$|\mu_{k}(\eps)| \in\Big[\frac{2c_-}{k-\beta} \frac{\eps^{k-\beta}}{(\sqrt{\Delta}s)^{k}}, \frac{2c_+}{k-\beta}\frac{ \eps^{k-\beta}}{(\sqrt{\Delta}s)^{k}}\Big].$$
In the sequel we show that Assumption \eqref{ass:psiK} holds, under the assumption \begin{align}\label{eq:asseptil}
s^{2}\Delta/\eps^2 \geq \bar C \log(n),
\end{align}
where $\bar C = \tilde c^{-2} >0$ is a large enough constant, \textit{i.e.}
$a :=\frac{\eps}{s\sqrt{\Delta}} \leq \frac{\tilde c}{\sqrt{\log(n)}}$.

\subsubsection{Preliminary technical Lemmas}

\begin{lemma}\label{lem:der0}
There exists $c_\beta>0$ a constant that depends only on $\beta, c_+, c_-$ such that the following holds
$$\mathrm{Re}(\psi(t)) \leq  -c_\beta t^2 \mathbf 1_{\{t\eps \leq \sqrt{\Delta s^{2}}\}} - c_\beta \frac{\Delta t^\beta}{\sqrt{\Delta s^{2}}^\beta} \mathbf 1_{\{t\eps > \sqrt{\Delta s^{2}}\}} - \frac{\Sigma^2}{2s^2} t^2  \mathbf 1_{\{t\eps > \sqrt{\Delta s^{2}}\}},\quad t>0,$$
where $\mathrm{Re}({\rm y})$ denotes the real part of ${\rm y}$.
\end{lemma}

\begin{proof}[Proof of Lemma \ref{lem:der0}]
It holds
$$\psi(t) = - t^2 \frac{\Sigma^2}{2s^2} +  \Delta \int_{-\eps}^\eps (e^{itu/\sqrt{\Delta s^{2}}} - 1 - itu/\sqrt{\Delta s^{2}}) d\nu(u)= - t^2 \frac{\Sigma^2}{2s^2} +  {I} .$$We now focus on the study of $I$.
Doing the change of variable $v=tu/\sqrt{\Delta s^{2}}$ we get
$$I =\Delta \sqrt{\Delta s^{2}} \int_{-t\eps/\sqrt{\Delta s^{2}}}^{t\eps/\sqrt{\Delta s^{2}}} (e^{iv} - 1 - iv)\frac{ d\nu(v\sqrt{\Delta s^{2}}/t)}{t}.$$
If $t\eps \leq \sqrt{\Delta s^{2}},$ for $0\leq v \leq 1$, there exists an absolute constant $c>0$ such that $\mathrm{Re}(e^{iv} - 1 - iv) = \cos(v)-1 \leq - cv^2$. In particular,
\begin{align*}
\mathrm{Re}(I) &\leq 2\Delta\sqrt{\Delta s^{2}}c_{-}t^{\beta}\int_{0}^{t\eps/\sqrt{\Delta s^{2}}} \frac{-cv^2}{(v\sqrt{\Delta s^{2}})^{\beta+1}}dv \\
&
= -2\Delta cc_-\frac{t^{\beta}}{\sqrt{\Delta s^{2}}^{\beta}} \int_{0}^{t\eps/\sqrt{\Delta s^{2}}}v^{1-\beta} dv 
= -\frac{2cc_-}{2-\beta} \frac{t^{2}\eps^{2-\beta}}{s^{2}}.
\end{align*}
Since $\sigma^2 (\eps)\in[ 2c_- \frac{\eps^{2-\beta}}{2-\beta},2c_+ \frac{\eps^{2-\beta}}{2-\beta}]$, whenever $t\eps\leq \sqrt{\Delta s^{2}}$ we have
\begin{align*}
\mathrm{Re}(\psi(t)) &\leq   - \frac{cc_-}{c_+ }\frac{\sigma^2(\eps)}{s^2} t^2   - \frac{\Sigma^2}{s^2} t^2
\leq  - \big(\frac{cc_-}{c_+}\land 1\big) t^2,
\end{align*}
using $s^2 = \sigma^2(\eps) +\Sigma^2$.

If $t\eps > \sqrt{\Delta s^{2}}$,
then $\mathrm{Re}(e^{iv} - 1 - iv) = \cos(v) - 1\leq 0$ for any $v\in \mathbb R$. Therefore,
\begin{align*}\mathrm{Re}(I)  &\leq 2\Delta c_-\frac{t^{\beta}}{\sqrt{\Delta s^{2}}^{\beta}}\int_{0}^{t\eps/\sqrt{\Delta s^{2}}} \frac{\cos(v) - 1}{v^{\beta+1}}dv  \leq 2\Delta c_-\frac{t^{\beta}}{\sqrt{\Delta s^{2}}^{\beta}}\int_{0}^{1} \frac{\cos(v) - 1}{|v|^{\beta+1}}dv \\ & \leq -2\Delta cc_-\frac{t^{\beta}}{\sqrt{\Delta s^{2}}^{\beta}}\int_{0}^{1}v^{1-\beta}dv = -\frac{2\Delta cc_-}{(2-\beta)\sqrt{\Delta s^{2}}^{\beta}}t^{\beta},
\end{align*}
 where we used the previous bound on $\cos(v).$
It follows that whenever $t\eps\geq \sqrt{\Delta s^{2}}$
\begin{align*}
\mathrm{Re}(\psi(t)) &\leq - t^2 \frac{\Sigma^2}{2s^2}  -2\Delta \frac{cc_-}{2-\beta}  \frac{t^{\beta}}{\sqrt{\Delta s^{2}}^\beta}. 
\end{align*}
 Putting together the cases $\{t\eps> \sqrt{\Delta s^{2}}\}$ and $\{t\eps\leq \sqrt{\Delta s^{2}}\}$ completes the proof.
\end{proof}

\begin{lemma}\label{lem:der1}
There exists $c_\beta>0$ a constant that depends only on $\beta, c_+, c_-$ such that the following holds for $t>0$
$$|\psi^{(1)}(t)| \leq  c_\beta t\mathbf 1_{\{t\eps \leq \sqrt{\Delta s^{2}}\}}  +c_\beta\Delta\frac{t^{\beta-1}}{\sqrt{\Delta s^{2}}^\beta} \mathbf 1_{\{t\eps > \sqrt{\Delta s^{2}}\}} + \frac{\Sigma^2}{s^2} t  \mathbf 1_{\{t\eps > \sqrt{\Delta s^{2}}\}}.$$
\end{lemma}

\begin{proof}[Proof of Lemma \ref{lem:der1}]First, it holds
$$\psi^{(1)}(t)=-\frac{\Sigma^{2}}{s^{2}}t+\Delta\int_{-\eps}^{\eps}\frac{iu}{\sqrt{\Delta s^{2}}}(e^{iut/\sqrt{\Delta s^{2}}}-1)d\nu(u).$$
The proof is similar to that of Lemma \ref{lem:der0} replacing $e^{itv}-1-itv$ with $ive^{itv}-iv$ which is of order $tv^{2}$ close to $0$. This leads to the bound
$$|\psi^{(1)}(t)| \leq t \frac{\Sigma^2}{s^2} + \frac{c_{\beta}t \sigma^{2}(\eps)}{s^2}\mathbf{1}_{\{t\eps\leq \sqrt{\Delta s^{2}}\}}+c_{\beta}\Delta\frac{t^{\beta-1}}{\sqrt{\Delta s^{2}}^{\beta}} \mathbf 1_{\{t\eps > 1\}}.$$
\end{proof}

\begin{lemma}\label{lem:derd}
For any integer $d \geq 2$ we have
$$|\psi^{(d)}(t)| \leq c_{\beta} \Big(\frac{\eps}{\sqrt{\Delta s^{2}}}\Big)^{d-2}+ \frac{\Sigma^2}{s^2}\mathbf 1_{\{d=2\} }\leq c_\beta \Big(\frac{\eps}{\sqrt{\Delta s^{2}}}\Big)^{d-2}:= c_{\beta}a^{d-2},$$
where $c_\beta>0$ is a constant that depends only on $\beta, c_+,c_-$.
\end{lemma}
\begin{proof}[Proof of Lemma \ref{lem:derd}]
We begin by observing that, for any $m\geq 2$, it holds
$$\psi^{(m)}(t)  = \frac{\Delta}{\sqrt{\Delta s^{2}}^{m}} \int_{-\eps}^\eps e^{itu/\sqrt{\Delta s^{2}}} (ui)^{m} d\nu(u).$$
The proof directly follows by performing the change of variable $v=tu/\sqrt{\Delta s^{2}}$ and using that $\beta<2$.
\end{proof}

The quantities $c_{int}$ and $K$ defined as in Section \ref{sec:prfthm2} appear in Lemmas \ref{lem:grandt} and \ref{lem:petitt} and in Section \ref{sec:propstable}.

\begin{lemma}\label{lem:grandt}
Assume that there exists $\bar C\geq c_{int}$ such that $\frac{s^2\Delta}{\eps^2}\geq \bar C^2\log (n).$
Then, there exists $C_\beta>0$ that depends only on $\beta, c_+, c_-$ such that the following holds. For all $m\leq K:=c_{int}^2\log(n)$,
$$\Big|\frac{(\exp(\psi))^{(m)}}{\exp(\psi)}\Big| \leq C_{\beta}^m f^m,$$
where $$f(t) = (c_\beta\lor 1)\sqrt{\Delta s^{2}} \eps^{-1}\mathbf{1}_{\{t\eps/\sqrt{\Delta s^{2}}\leq 1\}}  +(c_\beta\lor 1)\Delta\frac{t^{\beta-1}}{\sqrt{\Delta s^{2}}^{\beta}} \mathbf 1_{\{t\eps > \sqrt{\Delta s^{2}}\} } + \frac{\Sigma^2}{s^2} t  \mathbf 1_{\{t\eps> \sqrt{\Delta s^{2}}\}},$$ 
and $c_{\beta}$ is defined in Lemma \ref{lem:der1}.
\end{lemma}
\begin{proof}[Proof of Lemma \ref{lem:grandt}]
In accordance with the previous notation, we set $a:=\frac{\eps}{\sqrt{\Delta s^{2}}}$ and we observe that $Ka^{2}=c_{int}^{2}\log n\frac{\eps^2}{\Delta s^{2}} \leq  \frac{c_{int}^2}{\bar C^2} \leq 1$.
The proof is similar to the proof of Lemma \ref{lem:rec} considering instead the induction hypothesis $H(m) : \forall d \in \mathbb N, |R_m^{(d)}| \leq \big(4(c_{\beta}\lor 1)\big)^mf^m (1+ma)^d.$ Assumption $H(1)$ holds since $R_1=\psi^{(1)}$ and Lemma \ref{lem:derd} gives $|\psi^{( d)}| \leq c_{\beta}a^{d-2}$ for $d\geq 2$. To show that $H(m)$ implies $H(m+1)$ we use Lemma \ref{lem:derd} joined with $0\leq f^{-1}\leq a$ and $ma^{2}\leq 1$.   
\end{proof}

\begin{lemma}\label{lem:petitt}
Set $t_{\min} = c\log(n)$, $a=\frac{\eps}{\sqrt{\Delta s^{2}}}$ and assume that $1\leq c_{int}^2 \leq c$. For any $m\leq K$, and any $t\in [t_{min}\land a^{-1},a^{-1}]$ there exists a constant $C_\beta>0$, depending only on $\beta, c_+, c_-$, such that 
$$\Big|\frac{(\exp(\psi))^{(m)}}{\exp(\psi)}\Big| \leq C^m_{\beta} \tilde f^m,$$
where $ \tilde f(t) =  (c_\beta \lor 1) t$ and $c_{\beta}$ is defined as in Lemma \ref{lem:der1}.
\end{lemma}

\begin{proof}[Proof of Lemma \ref{lem:petitt}]
The proof is similar to the proof of Lemma \ref{lem:rec} considering instead the induction hypothesis, for $C_\beta \geq 4$ and $t\in [t_{\min}\land a^{-1},a^{-1}]$, $H(m) : \forall d \in \mathbb N, |R_m^{(d)}| \leq C_\beta^m\tilde f^m (1+m)^d$. The result holds since for all $m \leq K$, we have by assumption that $(1+m) \leq 1+K \leq 2t_{\min}$.\end{proof}

\subsubsection{Proof Proposition~\ref{prop:exstable}}\label{sec:propstable}

For any integer $k\geq 1$ we write
\begin{align*}
\int_{c\log(n)}^\infty |\Psi_\eps^{(k)}(t)|^2dt&=\int_{c\log(n)}^\infty |\Psi_\eps^{(k)}(t)|^2\1_{[\frac{1}{a},\infty)}(t)dt+\int_{c\log(n)}^\infty |\Psi_\eps^{(k)}(t)|^2\1_{[0,\frac{1}{a})}(t)dt=:T_1+T_2.
\end{align*}
\paragraph{Control of $T_1$.}

First, on the interval $t\in[\frac1a,\infty)$, by Lemma \ref{lem:grandt} there exists a constant $\bar c_\beta$, that depends only on $\beta, c_+,c_-$, such that 
$$f(t) \leq\bar c_\beta\bigg(\Delta\frac{t^{\beta-1}}{\sqrt{\Delta s^{2}}^{\beta}}+ \frac{t\Sigma^2}{s^2}\bigg).$$
Moreover, by Lemma~\ref{lem:der0}, there exists a constant $\bar c_\beta'>0$, that depends only on $\beta, c_+,c_-$, such that
$$|\exp(\psi(t))|\leq \exp\bigg(- 2\bar c_{\beta}'\Big(\frac{t^2\Sigma^2}{s^2}\Big)\lor  \Big( \frac{\Delta t^\beta}{\sqrt{\Delta s^{2}}^{\beta}}\Big)\bigg),\quad ta>1,$$
and
$$|\exp(\psi(t))|\leq \exp(-2\bar c_\beta'/{a^{2}}),\quad ta>1.$$
Therefore, by Equation~\eqref{eq:asseptil}
\begin{align}\label{eq:asseptil2}
|\exp(\psi(t))|\leq n^{-\bar C^2 \bar c_\beta' } := n^{-\kappa(\bar C)},\quad ta>1.
\end{align}
Using the previous inequalities, it follows from Lemma~\ref{lem:grandt} that there exists $C_\beta>0$ depending only on $\beta, c_+, c_-$ such that \begin{align*}
\int_{(1/a) \lor (c\log(n))}^{+\infty} |\Psi^{(m)}(t)|^2 dt 
&\leq \exp(-\bar c_\beta'/a^{2}) \int_{1/a}^{+\infty} C_{\beta}^{2m}\bar c_\beta^{2m}\Big(\frac{\Delta t^{\beta-1}}{\sqrt{\Delta s^{2}}^{\beta}} + \frac{t\Sigma^2}{s^2}\Big)^{2m} \exp\Big(- {\bar c_\beta'} \big(\frac{\Delta t^\beta}{\sqrt{\Delta s^{2}}^{\beta}} + \frac{t^2\Sigma^2}{s^2}\big)\Big)dt\\
&\leq \exp(-\bar c_\beta'/a^{2}) (C_{\beta}\bar c_\beta)^{2m} \int_{1/a}^{+\infty} \Big(\frac{\Delta t^{\beta-1}}{\sqrt{\Delta s^{2}}^{\beta}}\Big)^{2m} \exp\Big(- \bar c_{\beta}'\big(\frac{t^2\Sigma^2}{s^2}\big)\lor  \big( \frac{\Delta t^\beta}{\sqrt{\Delta s^{2}}^{\beta}}\big)\Big)dt\\
&\quad+\exp(-\bar c_\beta'/a^{2}) (C_{\beta}\bar c_\beta)^{2m} \int_{1/a}^{+\infty} \Big(\frac{t\Sigma^2}{s^2}\Big)^{2m} \exp\Big(- \bar c_{\beta}'\big( \frac{t^2\Sigma^2}{s^2}\big)\lor  \big(\frac{\Delta t^\beta}{\sqrt{\Delta s^{2}}^{\beta}}\big)\Big)dt\\
:&= A_1+A_2.
\end{align*}

We first consider the term $A_{1}$. By definition of $A_1,$ we immediately get
\begin{align*}
A_1 &\leq\exp(-\bar c_\beta'/a^{2})(C_{\beta}\bar c_\beta)^{2m} \int_{1/a}^{+\infty} \Big(\frac{\Delta t^{\beta}}{\sqrt{\Delta s^{2}}^{\beta}} {t^{\beta-2}}\Big)^m \exp\Big(- \bar c_\beta' \frac{\Delta t^\beta}{\sqrt{\Delta s^{2}}^{\beta}}\Big)dt.
\end{align*}
Using $\beta<2$ and \eqref{eq:asseptil} which implies $1/a\geq \bar C\sqrt{ \log(n)}>1$, we obtain 
\begin{align*}
A_1 &\leq \exp(-\bar c_\beta'/a^{2}) (C_{\beta}\bar c_\beta)^{2m} \int_{1/a}^{+\infty} \Big(\frac{\Delta t^{\beta}}{\sqrt{\Delta s^{2}}^{\beta}}  {a^{2-\beta}}\Big)^m \exp\Big(- \bar c_\beta' \frac{\Delta t^\beta}{\sqrt{\Delta s^{2}}^{\beta}}\Big)dt\\
&\leq \exp(-\bar c_\beta'/a^{2}) (C_{\beta}\bar c_\beta)^{2m} \int_{1/a}^{+\infty} \Big(\frac{\Delta t^{\beta}}{\sqrt{\Delta s^{2}}^{\beta}} \Big)^m \exp\Big(- \bar c_\beta' \frac{\Delta t^\beta}{\sqrt{\Delta s^{2}}^{\beta}}\Big)dt.
\end{align*}
We then consider the change of variable $v =\Delta t^\beta/\sqrt{\Delta s^{2}}^{\beta}$ and use Equation \eqref{eq:asseptil2} to get 
\begin{align*}
A_1&\leq {n^{-\kappa(\bar C)}} (C_{\beta}\bar c_\beta)^{2m}\frac{\sqrt{\Delta s^2}}{\beta\Delta^{1/\beta}} \int_{\Delta/\eps^\beta}^{+\infty}v^{m +1/\beta -1} \exp(-\bar c_\beta' v) dv.
\end{align*}
Finally,
\begin{align*}
A_1
&\leq {n^{-\kappa(\bar C)}}  \frac{\sqrt{\Delta s^2}}{\beta\Delta^{1/\beta}} (C_\beta \bar c_\beta)^{2m} ({\bar c_\beta'})^{{-m-1/\beta}} {\Gamma(m +1/\beta )}.
\end{align*}
It follows that, for $\bar C$ large enough and depending only on $c_{int}, \beta, c_+, c_-,c_{\max}$ and for any $m\leq K = c_{int}^2 \log(n)$, $ \Delta n^{c_{\max}}\geq 1$ and $ \log(s)/\log(n) \leq c_{\max}$, we have $A_1\leq n^{-4} m!.$

Let us now consider the term $A_{2}$. Note that if $\Sigma = 0$ then $A_2 = 0$; in the sequel we assume that $\Sigma\neq 0$. It holds  
\begin{align*}
A_2 &\leq \Bigg[\exp(-\bar c_\beta'/a^{2})  (C_\beta\bar c_\beta)^{2m} \int_{s^2/\Sigma^2 \lor 1/a}^{+\infty} \Big(\frac{t\Sigma^2}{s^2}\Big)^{2m} \exp\Big(- \big(\bar c_\beta' \frac{t^2\Sigma^2}{s^2}\big)\Big)dt\Bigg]\\
&\quad+ \Bigg[\exp(-\bar c_\beta'/a^{2})  (C_\beta\bar c_\beta)^{2m} \int_{ 1/a}^{s^2/\Sigma^2 \lor 1/a} \Big(\frac{t\Sigma^2}{s^2}\Big)^{2m} \exp\Big(-  \big(\bar c_\beta' \frac{\Delta t^\beta}{\sqrt{\Delta s^2}^\beta}\big)\Big)dt\Bigg]:= A_{2,1}+ A_{2,2}.\end{align*}We first consider the term $A_{2,1} $.  Using \eqref{eq:mmtN}, we have
\begin{align*}
A_{2,1}
&\leq\exp(-\bar c_\beta'/a^{2}) (C_\beta\bar c_\beta)^{2m} \Big(\frac{\Sigma^2}{s^2}\Big)^{2m} \int_{s^2/\Sigma^2}^{+\infty} t^{2m} \exp\Big(- \bar c_\beta' \frac{t^2\Sigma^2}{s^2}\Big)dt\\
&\leq \exp(-\bar c_\beta'/a^{2}) (C_\beta\bar c_\beta)^{2m} \Big(\frac{\Sigma^2}{s^2}\Big)^{2m} \exp(- \bar c_\beta' \frac{s^2}{\Sigma^2}) \int_{0}^{+\infty} t^{2m} \exp\Big(- \bar c_\beta' \frac{t^2\Sigma^2}{2s^2}\Big)dt\\
&\leq \exp(-\bar c_\beta'/a^{2}) (C_\beta\bar c_\beta)^{2m} \Big(\frac{ \Sigma^2}{s^2}\Big)^{2m} \exp\Big(- \bar c_\beta' \frac{s^2}{\Sigma^2}\Big)  \times 2\sqrt{2\pi} 2^m m!  \Big(\frac{s^2}{\bar c_\beta' \Sigma^2}\Big)^{m+\frac12}\\
&\leq \exp(-\bar c_\beta'/a^{2})  \Big(\frac{s^2}{ \Sigma^2}\Big)^{1/2} \exp\Big(- \bar c_\beta' \frac{s^2}{\Sigma^2}\Big)  \times \frac{10}{\sqrt{\bar c_\beta'}} \Big(\frac{2C_\beta^2 \bar c_\beta^2}{\bar c_\beta'}\Big)^{m} m!.
\end{align*}
Then, from \eqref{eq:asseptil2} and for a constant $\bar C$  large enough depending only on $\beta, c_+, c_-, c_{int}$, it holds for any $m\leq K = c_{int}^2 \log(n)$ that $A_{2,1}
\leq n^{-4}  m!.$
Now, we consider the second term,
\begin{align*}
A_{2,2}
&=\exp(-\bar c_\beta'/a^{2}) (C_\beta\bar c_\beta)^{2m} \Big(\frac{\Sigma^2}{s^2}\Big)^{2m} \int_{ 1/a}^{s^2/\Sigma^2 \lor 1/a} t^{2m} \exp\Big(- \bar c_\beta'  \frac{\Delta t^\beta}{\sqrt{\Delta s^2}^\beta}\Big)dt\\
&\leq \exp(-\bar c_\beta'/a^{2}) (C_\beta\bar c_\beta)^{2m}  \int_{ 1/a}^{s^2/\Sigma^2 \lor 1/a} \exp\Big(- \bar c_\beta'  \frac{\Delta t^\beta}{\sqrt{\Delta s^2}^\beta}\Big)dt.
\end{align*}
We apply the change of variable $v = \Delta t^\beta/\sqrt{\Delta s^2}^\beta$
\begin{align*}
A_{2,2}
&\leq \exp(-\bar c_\beta'/a^{2})\frac{\sqrt{\Delta s^2}^\beta}{\beta \Delta^{1/\beta}} (C_\beta\bar c_\beta)^{2m} \int_{\Delta/\eps^\beta}^{+\infty}\Big(\frac{\sqrt{\Delta s^{2}}^\beta}{\Delta}v\Big)^{1/\beta-1} \exp(-\bar c_\beta' v)dv\\
&\leq \exp(-\bar c_\beta/a^{2})  \frac{1}{\beta} (C_\beta\bar c_\beta)^{2m}  \frac{\sqrt{\Delta s^{2}}}{\Delta^{1/\beta}} (\bar c_\beta')^{-1/\beta+1} \Gamma(1/\beta).
\end{align*}
So for any $m\leq K = c_{int}^2 \log(n)$, $n^{c_{\max}}\Delta \geq 1$ and $\log s /\log n  \leq c_{\max}$ we have if $\bar C$ is large enough depending only on $c_{int}, \beta, c_+, c_-,c_{\max}$ (see \eqref{eq:asseptil2}) that $A_{2,2}\leq n^{-4} m! .$

Gathering both bounds on $A_{2,1}, A_{2,2}$, for any $m\leq K = c_{int}^2 \log(n)$ we have, if $\bar C$ is large enough depending only on $c_{int}, \beta, c_+, c_-, c_{\max}$, that 
$A_2 = A_{2,1}+A_{2,2}\leq 2n^{-4} m!.$
Finally, gathering all terms, we derive that \eqref{ass:psiK} holds on the set $[\frac{1}{a},\infty)$:
$$\int_{(1/a) \lor (c\log(n))}^{+\infty} |\Psi^{(m)}(t)|^2 dt \leq A_1+A_2 \leq 3n^{-4}m! ,\quad m\leq K=c_{int}^{2}\log n.
$$

\paragraph{Control of $T_2$.}
Whenever $1/a \leq t_{min}=c\log(n)$, $T_2=0$. In what follows, assume that $1/a > t_{min}$.
By definition of $\tilde f$ in Lemma~\ref{lem:petitt} there exists $c_\beta>0$, depending only on $\beta, c_+, c_-$, such that 
$\tilde f(t) \leq \bar c_\beta t.$
Moreover, Lemma~\ref{lem:der0} implies that there exists $\bar c_\beta'>0$, depending only on $\beta, c_+, c_-$, such that
$|\exp( \psi(t))| \leq \exp(-2\bar c_\beta't_{min}^2),~t\in[t_{\min},1/a].$ Then, by \eqref{eq:asseptil}
\begin{align}\label{eq:asseptil3}
|\exp(\psi(t))|\leq n^{-c^2 \bar c_\beta' } := n^{-\kappa(c)},~~t\in[t_{\min},1/a].
\end{align}
This, together with  Lemma~\ref{lem:petitt}, implies that there exists a constant $C_\beta>0$ that depends only on $\beta, c_+, c_-$ such that
\begin{align*}
\int_{c\log(n)}^{1/a} |\Psi^{(m)}(t)|^2 dt &\leq C_\beta^{2m}\exp(- \bar c_\beta' t_{\min}^2) \int_{c\log(n)}^{1/a}   \big(\bar c_\beta t\big)^{2m} \exp\big(- \bar c_\beta' t^2\big)dt.
\end{align*}
Using \eqref{eq:mmtN} and \eqref{eq:asseptil3} we get
\begin{align*}
\int_{c\log(n)}^{1/a} |\Psi^{(m)}(t)|^2 dt &\leq2\sqrt{2\pi} \exp(-\bar c_\beta t_{\min}^2) C_\beta^{2m}(2\bar c_\beta)^{2m} 2^m m! (\bar c_\beta')^{-m+1/2}\\
&\leq  10 n^{-\kappa(c)} m!  \sqrt{\bar c_\beta'} (4C_\beta\bar c_\beta/ \sqrt{\bar c_\beta'})^{2m}.
\end{align*}
We conclude taking $c \geq 0$ a large enough constant depending only on $\beta, c_+, c_-, c_{int}$ (see \eqref{eq:asseptil3})
\begin{align*}
\int_{c\log(n)}^{1/a} |\Psi^{(m)}(t)|^2 dt &\leq n^{-4}m!,\quad \forall m \leq K = c_{int}^2 \log(n).
\end{align*}

\paragraph{Conclusion.}
Putting both parts of the integral together we have
\begin{align*}
\int_{c\log(n)}^{+\infty} |\Psi^{(m)}(t)|^2 dt \leq 3n^{-4}m!, \quad \forall m \leq K = c_{int}^2 \log(n).
\end{align*}
This concludes the proof of Proposition \ref{prop:exstable}.

\subsection{Proof of Proposition~\ref{prop:Siglar}}

Write $\tilde X_{\Delta}(\eps) = (X_{\Delta }(\eps) -\Delta b(\eps))/{\sqrt{\Delta s^2}}$ for the rescaled increment, $\tilde M_{\Delta }(\eps) = { M_{\Delta }(\eps)}/{\sqrt{\Delta s^2}}$ for the pure-jump part of the rescaled increment and $s^2 = \sigma^2(\eps) + \Sigma^2$.
We first state the following lemma, proven at the end of the proof of this proposition.
\begin{lemma}\label{lem:concforM}
Assume that $\Delta \lambda_{0,\eps} > 1$ and $\eps \leq \sqrt{\Delta (\Sigma^2 + \sigma^2(\eps))}$. It holds for any $x>0$ that
\begin{align*}\PP\big(|\tilde M_{\Delta}(\eps)|>x\big)
&\leq  2\exp\Big(-\frac{x^2}{17} \Big)  + 2\exp(-x/2).
\end{align*}
\end{lemma}
Using Lemma~\ref{lem:concforM}, we get
$\PP\big(\big| \tilde M_{\Delta }(\eps)\big|\geq x\big)\leq 2\exp(-{x^2}/{17}) + 2 \exp(-x/2).$ 
This implies that for $\mu$ being the measure corresponding to the distribution of the rescaled jump component $|\tilde M_{\Delta}(\eps)|$, and for any $k \in \mathbb N^*$ \begin{align}\E[|\tilde M_{\Delta}(\eps)|^{k}]&=
k\int_0^{+\infty} x^{k-1} \mu(dx) \leq2 k\int_0^{+\infty} x^{k-1} \Big(\exp(-{x^2}/{17}) + \exp(-x/2)\Big) dx 
\leq 4\times 3^{k+2}  k!. \label{eq:mm}
\end{align}
The rescaled increment $\tilde X_{\Delta}(\eps)$ has characteristic function given by $ \Psi =  \Phi _{1} \Psi_{2},$
for $ \Phi_{1}(t) = \exp(-\frac{t^2\Sigma^2}{2s^{2}})$ and $ \Psi_{2}$ the characteristic function of  $\tilde M_{\Delta}(\eps)$. As $\int \mu(dx)=1$ and  $(e^{iut})^{(d)}=(ui)^d e^{uit}$, we get using  Equation~\eqref{eq:mm} for any $d\in \mathbb N^*$,
$| \Psi_2^{(d)}(t)| \leq 4\times 3^{d+2}  d!.$
This implies that
\begin{align}\label{eq:psider}
|\Psi^{(k)}|^2=\Big|\sum_j {k \choose j} \Phi_{1}^{(j)}  \Psi_{2}^{(k-j)}\Big|^{2} \leq 16\times 3^{2k+4}  \max_{j\leq k} | \Phi_{1}^{(j)}|^2 ((k-j)!)^2.
\end{align}
Set $r^2 = {\Sigma^2}/{s^{2}}$ and note that $(c_\Sigma^2+1)^{-1/2} \leq r \leq 1$. We have by definition of Hermite polynomials
\begin{align*}
\Phi_{1}^{(j)}(t) = r^j H_j(r t) \exp(-t^2r^2/2).
\end{align*}
Since $|H_j(u)| \leq j \times j!(|u|^j+1)$ we have that
\begin{align*}
\int_{c\log(n)}^{+\infty} |\Phi_{1}^{(j)}(t)|^2 dt &\leq 2r^{2j} j^2 \times (j!)^2 \int_{c\log(n)}^{+\infty} (|rt|^{2j}+1) \exp(-t^2r^2) dt\\
&\leq j^{3j} r^{2j}  \exp(-c^2\log(n)^2r^2/2)  \int_{c\log(n)}^{+\infty} (|rt|^{2j}+1) \exp(-t^2r^2/2) dt\\
&\leq (2j)^{5j} \sqrt{2\pi(c_\Sigma^2+1)}  \exp\Big(-\frac{c^2\log(n)^2}{2(c_\Sigma^2+1)}\Big).
\end{align*}
By Equation~\eqref{eq:psider} it follows
\begin{align*}
\int_{c\log(n)}^{+\infty} |\Psi^{(k)}(t)|^2dt &\leq  16\times 3^{2k+4}  \max_{j\leq k} ((k-j)!)^2 (2j)^{5j} \sqrt{2\pi(c_\Sigma^2+1)}  \exp\Big(-\frac{c^2\log(n)^2}{2(c_\Sigma^2+1)}\Big)\\
&\leq  16\times 3^{2k+4} 2^{5k} k^{7k} \sqrt{2\pi(c_\Sigma^2+1)}  \exp\Big(-\frac{c^2\log(n)^2}{2(c_\Sigma^2+1)}\Big),
\end{align*}
taking $c$ large enough depending only on $c_{\Sigma}$ completes the proof.

\begin{proof}[Proof of Lemma \ref{lem:concforM}]
For $\eps \geq \eta>0$, write $v_{\eta, \eps} = \lambda_{\eta,\eps}^{-1}\int_{\eta \leq |x|\leq \eps} x^2 d\nu$. Using notations and definitions from the proof of Lemma \ref{lem:ass0}, we have for $\eps \geq \eta>0$ and using $\lambda_{\eta,\eps}v_{\eta, \eps} \leq s^{2}$ we get for all $x>0$
\begin{align}\PP\big(|\bar M_{\Delta}(\eta,\eps)|>x\big|N_\Delta(\eta,\eps)=N\big)&
\leq 2\exp\Bigg(-\frac12\frac{x^2}{\frac{N}{\Delta \lambda_{\eta,\eps}} +\frac13\frac{\eps}{\sqrt{\Delta s^{2}}} x}\Bigg).\label{eq:concNfix}
\end{align}
Now by assumption on $\lambda_{0,\eps}$, there exists $\bar \eta >0$ such that for any $\eta \leq \bar \eta$, we have $\Delta \lambda_{\eta,\eps} \geq 1$. 
Moreover, for $\eta \leq \bar \eta$, since $N_\Delta(\eta,\eps)$ is a Poisson random variable of parameter $\Delta \lambda_{\eta,\eps} \geq 1$, we have using the Chernoff bound, 
\begin{align*}
\mathbb P(N_\Delta(\eta,\eps) - \Delta \lambda_{\eta,\eps} \geq x) &\leq \exp(-x^2/(2e^4 \Delta \lambda_{\eta,\eps} )),\hspace{1cm}0 \leq x \leq e^2 \Delta \lambda_{\eta,\eps} \\
\mathbb P(N_\Delta(\eta,\eps) - \Delta \lambda_{\eta,\eps} \geq x)&\leq \exp(-x/2), \hspace{3.5cm}x\geq e^2 \Delta \lambda_{\eta,\eps}.
\end{align*}
Combining Equation~\eqref{eq:concNfix} and the last displayed equation, one gets for for $\eta \leq \bar \eta$ and $x>0$
\begin{align*}\PP\big(|\bar M_{\Delta}(\eta,\eps)|>x\big)
&{\leq \PP\Big(\big\{|\bar M_{\Delta}(\eta,\eps)|>x\,\cap\,N_\Delta(\eta,\eps) \leq (1+e^2)\Delta \lambda_{\eta,\eps} + x\big\}\bigcup\big\{N_\Delta(\eta,\eps) \geq (1+e^2)\Delta \lambda_{\eta,\eps} + x\big\}\Big) }\\
&\leq 2\exp\Bigg(-\frac12\frac{x^2}{\frac{\Delta \lambda_{\eta,\eps}(1+e^2) + x}{\Delta \lambda_{\eta,\eps}} +\frac13\frac{\eps}{\sqrt{\Delta s^{2}}} x}\Bigg) +\exp(-x/2)
 \leq 2\exp\Big(-\frac{x^2}{17} \Big)  + 2\exp(-x/2) ,
\end{align*}
where we use ${\eps}/{\sqrt{\Delta s^{2}}}  \leq 1$.
The last equation holds for any $\eta \leq \bar \eta$. Having $\eta\to0$, we conclude as in the proof of Lemma \ref{lem:ass0} to transfert the result from $\bar M_{\Delta}(\eps)$ to $\tilde M_{\Delta}(\eps)$.
\end{proof}

\section{Some bounds in total variation distance\label{appB}}

 \begin{lemma}\label{TV:gaussians}
 If $\mu_1,\mu_2\in\R$, $0\leq \Sigma_1^2\leq \Sigma_2^2$ and $\Sigma_2\in \R_{>0}$, then
$$\|\No(\mu_1,\Sigma_1^2)^{\otimes n}-\No(\mu_2,\Sigma_2^2)^{\otimes n}\|_{TV}\leq 1-\bigg(\frac{\Sigma_1}{\Sigma_2}\bigg)^n+\sqrt n \frac{|\mu_1-\mu_2|}{\sqrt{2\pi\Sigma_2^2}},\quad \forall n\geq 1.$$\end{lemma}
\begin{proof}
Since the Lemma is trivially true if $\Sigma_1=0$ we can assume that $\Sigma_1>0$ without loss of generality. By triangular inequality, for any $n\geq 1$ we have:
$$\|\No(\mu_1,\Sigma_1^2)^{\otimes n}-\No(\mu_2,\Sigma_2^2)^{\otimes n}\|_{TV}\leq \|\No(\mu_1,\Sigma_1^2)^{\otimes n}-\No(\mu_1,\Sigma_2^2)^{\otimes n}\|_{TV}+ \|\No(\mu_1,\Sigma_2^2)^{\otimes n}-\No(\mu_2,\Sigma_2^2)^{\otimes n}\|_{TV}.$$
Denote by $\varphi_n(x_1,\dots,x_n)=\frac{1}{(2\pi)^{n/2}}e^{-\frac{x_1^2+\dots+x_n^2}{2}}$ the density of a standard normal distribution on $\R^n$ and observe that
\begin{align*}
\|\No(\mu_1,\Sigma_1^2)^{\otimes n}&-\No(\mu_1,\Sigma_2^2)^{\otimes n}\|_{TV}=\|\No(0,\Sigma_1^2)^{\otimes n}-\No(0,\Sigma_2^2)^{\otimes n}\|_{TV}\\
&=\frac{1}{2}\int_{\R^n}\bigg|\varphi_n(z_1,\dots,z_n)-\Big(\frac{\Sigma_1}{\Sigma_2}\Big)^n\varphi_n\Big(\frac{\Sigma_1}{\Sigma_2}z_1,\dots,\frac{\Sigma_1}{\Sigma_2}z_n\Big)\bigg|dz_1\dots dz_n\\
&\leq \frac{1}{2}\int_{\R^n}\Big|1-\Big(\frac{\Sigma_1}{\Sigma_2}\Big)^n\Big|\varphi_n(z_1,\dots,z_n)dz_1\dots dz_n\\
&\quad +  \frac{1}{2}\int_{\R^n}\Big(\frac{\Sigma_1}{\Sigma_2}\Big)^n\Big|\varphi_n(z_1,\dots,z_n)-\varphi_n\Big(\frac{\Sigma_1}{\Sigma_2}z_1,\dots,\frac{\Sigma_1}{\Sigma_2}z_n\Big)\Big|dz_1\dots dz_n\\
&= \frac{1-\Big(\frac{\Sigma_1}{\Sigma_2}\Big)^n}{2}+\frac{1}{2}\int_{\R^n}\Big(\frac{\Sigma_1}{\Sigma_2}\Big)^n\Big(\varphi_n\Big(\frac{\Sigma_1}{\Sigma_2}z_1,\dots,\frac{\Sigma_1}{\Sigma_2}z_n\Big)-\varphi_n(z_1,\dots,z_n)\Big)dz_1\dots dz_n\\
&= \frac{1-\Big(\frac{\Sigma_1}{\Sigma_2}\Big)^n}{2}+\frac{1}{2}\Big(\frac{\Sigma_1}{\Sigma_2}\Big)^n\Big(\Big(\frac{\Sigma_2}{\Sigma_1}\Big)^n-1\Big)=1-\Big(\frac{\Sigma_1}{\Sigma_2}\Big)^n.
\end{align*}
Let $U$ be any $n\times n$ orthonormal matrix with first column given by $\big(\frac{1}{\sqrt n},\dots,\frac{1}{\sqrt n} \big)^t$, set $\alpha=\frac{\mu_1-\mu_2}{\Sigma_2}$ and $\beta=U^t(\alpha,\dots,\alpha)^t=(\sqrt n \alpha,0,\dots,0)^t$. Observe that $\varphi_n(U x)=\varphi_n(x)$ for any $x\in\R^n.$ We have:
\begin{align*}
 \|\No(\mu_1,\Sigma_2^2)^{\otimes n}-\No(\mu_2,\Sigma_2^2)^{\otimes n}\|_{TV}&=\frac{1}{2\Sigma_2^n}\int_{\R^n}\bigg|\varphi_n\big(\frac{x_1-\mu_1}{\Sigma_2},\dots,\frac{x_n-\mu_1}{\Sigma_2}\big)-\varphi_n\big(\frac{x_1-\mu_2}{\Sigma_2},\dots,\frac{x_n-\mu_2}{\Sigma_2}\big)\bigg|dx_1\dots dx_n\\
 &=\frac{1}{2}\int_{\R^n}\bigg|\varphi_n\big(z_1-\alpha,\dots,z_n-\alpha\big)-\varphi_n(z_1\dots,z_n)\bigg|dz_1\dots dz_n.
\end{align*} 
Let $z=(z_1,\dots,z_n)^t$. By the change of variable $y=U^t z$ we get
\begin{align*}
\int_{\R^n}\bigg|\varphi_n\big(z_1-\alpha,\dots,z_n-\alpha\big)&-\varphi_n(z_1\dots,z_n)\bigg|dz_1\dots dz_n\\ &=\int_{\R^n}|\varphi_n\big(U y-(\alpha,\dots,\alpha)^t\big)-\varphi_n\big(U y\big)|dy_1\dots dy_n\\
&=\int_{\R^n}|\varphi_n(U(y-\beta))-\varphi_n(Uy)|dy_1\dots dy_n\\
&=\int_{\R^n}|\varphi_n(y_1-\sqrt n \alpha, y_2,\dots, y_n)-\varphi_n(y_1,\dots,y_n)|dy_1\dots dy_n\\
&=\int_{\R^n}\bigg|\varphi_1(y_1-\sqrt n \alpha)\prod_{j=2}^n\varphi_1(y_j)-\prod_{j=1}^n\varphi_1(y_j)\bigg|dy_1\dots dy_n\\
&=\int_\R|\varphi_1(y_1-\sqrt n \alpha)-\varphi_1(y_1)|dy_1=2\int_{-\frac{\sqrt n|\alpha|}{2}}^{\frac{\sqrt n|\alpha|}{2}}\varphi_1(y)dy\leq \frac{2\sqrt n|\alpha|}{\sqrt{2\pi}}.
\end{align*}
We deduce that
$$ \|\No(\mu_1,\Sigma_2^2)^{\otimes n}-\No(\mu_2,\Sigma_2^2)^{\otimes n}\|_{TV}\leq  \frac{\sqrt n|\mu_1-\mu_2|}{\sqrt{2\pi}\Sigma_2}.$$
\end{proof}
\begin{lemma}\label{TV:CPP}
 Let $(X_i)_{i\geq 1}$ and $(Y_i)_{i\geq1}$ be sequences of i.i.d. random variables a.s. different from zero and $N$, $N'$ be two Poisson random variables with $N$ (resp. $N'$) independent of $(X_i)_{i\geq 1}$ (resp. $(Y_i)_{i\geq 1}$). Denote by $\lambda$ (resp. $\lambda'$) the mean of $N$ (resp. $N'$). Then, for any $n\geq 1$
    $$\bigg\|\mathscr{L}\bigg(\sum_{i=1}^N X_i\bigg)^{\otimes n}-\mathscr{L}\bigg(\sum_{i=1}^{N'}Y_i\bigg)^{\otimes n}\bigg\|_{TV}\leq n (\lambda\wedge\lambda')\|X_1-Y_1\|_{TV}+1-e^{-n|\lambda-\lambda'|}.$$
\end{lemma}
\begin{proof}
The proof is a minor extension of the proof of Theorem 13 in \cite{MR}. For the ease of the reader we repeat here the essential steps.
Without loss of generality, let us suppose that $\lambda\geq \lambda'$ and write $\lambda=\alpha+\lambda'$, $\alpha\geq 0$.
  By  triangle inequality,
  \begin{align}
  \bigg\|\Li\bigg(\sum_{i=1}^N X_i\bigg)^{\otimes n}-\Li\bigg(\sum_{i=1}^{N'}Y_i\bigg)^{\otimes n}\bigg\|_{TV} &\leq \bigg\|\Li\bigg(\sum_{i=1}^N X_i\bigg)^{\otimes n}-\Li\bigg(\sum_{i=1}^{N''}X_i\bigg)^{\otimes n}\bigg\|_{TV}\nonumber\\
  &\quad + \bigg\|\Li\bigg(\sum_{i=1}^{N''} X_i\bigg)^{\otimes n}-\Li\bigg(\sum_{i=1}^{N'}Y_i\bigg)^{\otimes n}\bigg\|_{TV},\label{eq:ti}
  \end{align}
  where $N''$ is a random variable independent of $(X_i)_{i\geq 1}$ and with the same law as $N'$.
 
  Let $P$ be a Poisson random variable independent of $N''$ and $(X_i)_{i\geq 1}$ with mean $\alpha$. Then,
  \begin{align}\bigg\|\Li\bigg(\sum_{i=1}^N X_i\bigg)^{\otimes n}-\Li\bigg(\sum_{i=1}^{N''}X_i\bigg)^{\otimes n}\bigg\|_{TV} &=\bigg\|\Li\bigg(\sum_{i=1}^{N''+P} X_i\bigg)^{\otimes n}-\Li\bigg(\sum_{i=1}^{N''}X_i\bigg)^{\otimes n}\bigg\|_{TV}\nonumber \\ &\leq \bigg\|\delta_0-\Li\bigg(\sum_{i=1}^{P}X_i\bigg)^{\otimes n}\bigg\|_{TV} \label{eq:subn} \\ \nonumber
  &= \PP\bigg(\bigg(\sum_{i=1}^{P}X_i\bigg)^{\otimes n}\neq (0,\dots,0)\bigg)\leq 1-e^{-n\alpha},
  \end{align}
  where \eqref{eq:subn} follows by Lemma \ref{subadditivity}.

 In order to bound the second addendum in \eqref{eq:ti} we condition on $N'$ and use again Lemma \ref{subadditivity} joined with the fact that $\Li(N')=\Li(N'')$. Denoting by $N_1',\dots, N_n'$ $n$ independent copies of $N'$ and by $X_{i,j}$ (resp. $Y_{i,j}$), $j=1,\dots,n$ and $i\geq 1$, $n$ independent copies of $X_i$ (resp. $Y_i$), we have:
  \begin{align*}
 & \bigg\|\Li\bigg(\sum_{i=1}^{N''} X_i\bigg)^{\otimes n}-\Li\bigg(\sum_{i=1}^{N'}Y_i\bigg)^{\otimes n}\bigg\|_{TV}\\
 &=\sum_{m_1\geq 0,\dots,m_n\geq 0} \bigg\|\Li\bigg(\sum_{i=1}^{m_1} X_{i,1},\dots,\sum_{i=1}^{m_n} X_{i,n}\bigg)-\Li\bigg(\sum_{i=1}^{m_1}Y_{i,1},\dots,\sum_{i=1}^{m_n}Y_{i,n}\bigg)\bigg\|_{TV}\PP(N'_1=m_1)\dots\PP(N_n'=m_n)\\
 &\leq \sum_{m_1\geq 0,\dots,m_n\geq 0} \big(m_1 \|X_{1,1}-Y_{1,1}\|_{TV}+\dots+m_n \|X_{1,n}-Y_{1,n}\|_{TV}\big)\PP(N'_1=m_1)\dots\PP(N_n'=m_n)
 \\&=\|X_1-Y_1\|_{TV}(\E[N_1']+\dots+\E[N_n'])=n\lambda'\|X_1-Y_1\|_{TV}.
  \end{align*}
  
\end{proof}

\begin{lemma}\label{subadditivity}
Let $X$, $Y$ and $Z$ be three random variables on $\R^N$, $N\geq 1$, with $Z$ independent of $X$ and $Y$. Then,
\begin{equation}\label{add}
\|\Li(X+Z)-\Li(Y+Z)\|_{TV}\leq \|X-Y\|_{TV}.
\end{equation}
Let $(X_i)_{i=1}^m$ and $(Y_i)_{i=1}^m$ be independent random variables on $(\R^N,\mathscr B(\R^N))$, $N\geq 1$. Then, for any $m\geq 1$,
\begin{equation}\label{eq:subadd}
\bigg\|\sum_{i=1}^m X_i-\sum_{i=1}^m Y_i\bigg\|_{TV}\leq \sum_{i=1}^m\|X_i-Y_i\|_{TV}.
\end{equation}
\end{lemma}
\begin{proof}
To prove \eqref{add} one can use the variational definition of the total variation. Denoting by $P_{A}$ the law of a random variable $A$, it holds
\begin{align*}
2\|\Li(X+Z)-\Li(Y+Z)\|_{TV}&=\sup_{\|f\|_\infty\leq 1}\int f(x)\big(P_{X+Z}(dx)-P_{Y+Z}(dx)\big)\\
&=\sup_{\|f\|_\infty\leq 1}\int f(x) \bigg(\int P_Z(dx-z)( P_X(dz)- P_Y(dz))\bigg)\\
&=\sup_{\|f\|_\infty\leq 1}\int \int f(x)  P_Z(dx-z)( P_X(dz)- P_Y(dz)).
\end{align*}
Denote by $g_f(z)=\int f(x) P_Z(dx-z)$ and observe that $g_f$ is measurable with $\|g_f\|_\infty\leq \|f\|_\infty$. It then follows that
\begin{align*}
\sup_{\|f\|_{\infty}\leq 1} \int g_f(z)( P_X(dz)- P_Y(dz))&\leq \sup_{\|g\|_\infty\leq 1}\int g(z) ( P_X(dz)- P_Y(dz))=2\|X-Y\|_{TV}.
\end{align*}
Equation \eqref{eq:subadd} is straightforward using \eqref{add}. Indeed, by induction, it suffices to prove the case $n=2$.
Let $\tilde X_2$ be a random variable equal in law to $X_2$ and independent of $Y_1$ and of $X_1$. By triangle inequality we deduce that
\begin{align*}
\|\Li(X_1+X_2)-\Li(Y_1+Y_2)\|_{TV}&\leq \|\Li(X_1+X_2)-\Li(Y_1+\tilde X_2)\|_{TV}+\|\Li(Y_1+\tilde X_2)-\Li(Y_1+Y_2)\|_{TV}
\end{align*}
and by means of \eqref{add} we conclude that
\begin{align*}
\|\Li(Y_1+\tilde X_2)-\Li(Y_1+Y_2)\|_{TV}&\leq \|\tilde X_2-Y_2\|_{TV}= \| X_2-Y_2\|_{TV},\\
\|\Li(X_1+ X_2)-\Li(Y_1+\tilde X_2)\|_{TV}&\leq\|X_1-Y_1\|_{TV}.
\end{align*}
\end{proof}
\begin{lemma}\label{TV:product}
Let $P$ and $Q$ be probability density. For any $n\geq 1$ it holds:
$$\|P^{\otimes n}-Q^{\otimes n}\|_{TV}\leq \sqrt{2n\|P-Q\|_{TV}}.$$
\end{lemma}
\begin{proof}
Let $H(P,Q)$ denote the Hellinger distance between $P$ and $Q$, i.e. 
$$H(P,Q)=\sqrt{\int \bigg(\sqrt{\frac{dP}{d\mu}}-\sqrt{\frac{dQ}{d\mu}}\bigg)^2d\mu}$$
where $\mu$ it is a common dominating measure for $P$ and $Q$.
It is well known (see e.g. Lemma 2.3 and Property (iv) page 83 in \cite{tsybakov2009introduction}) that
$$\frac{H^2(P,Q)}{2}\leq \|P-Q\|_{TV}\leq H(P,Q), \quad H^2(P^{\otimes n},Q^{\otimes n})=2\bigg(1-\bigg(1-\frac{H^2(P,Q)}{2}\bigg)^n\bigg).$$
In particular $H^2(P^{\otimes n},Q^{\otimes n})\leq nH^2(P,Q)$ and 
$
\|P^{\otimes n}-Q^{\otimes n}\|_{TV}\leq \sqrt{n H^2(P,Q)} \leq \sqrt{2n\|P-Q\|_{TV}},
$
as desired.
\end{proof}

\end{document}